\title{Subgraph densities in Markov spaces}
\author{D\'avid Kunszenti-Kov\'acs\footnote{Research supported by ERC
Consolidator Grant 648017}, L\'aszl\'o Lov\'asz\footnote{Research supported by
ERC Synergy Grant No.~810115.} ~and Bal\'azs Szegedy\footnote{Research was
partially supported by the NKFIH "\'Elvonal" KKP 133921 grant.}\\
HUN-REN Alfr\'ed R\'enyi Institute of Mathematics\\
Budapest, Hungary}

\documentclass[12pt,fleqn]{article}
\usepackage{amsmath,amssymb,graphicx,bbm,bm,delarray,pict2e}
\usepackage{ntheorem}
\usepackage[hypertex]{hyperref}
\usepackage{yfonts,mathrsfs,dsfont}
\usepackage[utf8]{inputenc}
\usepackage[normalem]{ulem}
\usepackage{srcltx,xcolor}
\long\def\ignore#1{}

\begin{document}
\newtheorem{theorem}{Theorem}[section]
\newtheorem{prop}[theorem]{Proposition}
\newtheorem{algorithm}[theorem]{Algorithm}
\newtheorem{lemma}[theorem]{Lemma}
\newtheorem{claim}{Claim}
\newtheorem{corollary}[theorem]{Corollary}
\newtheorem*{thm*}{Theorem \ref{THM:MAIN-KAB}}
\theorembodyfont{\rmfamily}
\newtheorem{remark}[theorem]{Remark}
\newtheorem{definition}[theorem]{Definition}
\newtheorem{example}[theorem]{Example}
\newtheorem{conj}{Conjecture}
\newtheorem{prob}{Problem}
\newtheorem{step}{Step}
\newtheorem{alg}{Algorithm}

\newenvironment{proof}{\medskip\noindent{\bf Proof. }}{\hfill$\square$\medskip}
\newenvironment{proof*}[1]{\medskip\noindent{\textbf{ Proof of #1}}}{\hfill$\square$\medskip}

\def\R{\mathbb{R}}
\def\one{\mathbbm1}
\def\T{^{\sf T}}
\def\Pr{{\sf P}}
\def\E{{\sf E}}
\def\Q{{\mathbf Q}}
\def\bd{\text{bd}}
\def\eps{\varepsilon}
\def\wh{\widehat}
\def\cork{\text{\rm corank}}
\def\rank{\text{\rm rank}}
\def\Ker{\text{\rm Ker}}
\def\rk{\text{\rm rank}}
\def\supp{\text{\rm supp}}
\def\diag{\text{\rm diag}}
\def\sep{\text{\rm sep}}
\def\tr{\text{\rm tr}}
\def\vol{\text{\rm vol}}
\def\gen{\text{\rm gen}}
\def\Det{\text{\rm Det}}
\def\lin{\text{\rm lin}}
\def\intl{\int\limits}
\def\disc{\text{\sf disc}}
\def\tv{{\rm tv}}
\def\deg{\text{\rm deg}}
\def\Hom{\text{\rm Hom}}

\def\AA{\mathcal{A}}\def\BB{\mathcal{B}}\def\CC{\mathcal{C}}
\def\DD{\mathcal{D}}\def\EE{\mathcal{E}}\def\FF{\mathcal{F}}
\def\GG{\mathcal{G}}\def\HH{\mathcal{H}}\def\II{\mathcal{I}}
\def\JJ{\mathcal{J}}\def\KK{\mathcal{K}}\def\LL{\mathcal{L}}
\def\MM{\mathcal{M}}\def\NN{\mathcal{N}}\def\OO{\mathcal{O}}
\def\PP{\mathcal{P}}\def\QQ{\mathcal{Q}}\def\RR{\mathcal{R}}
\def\SS{\mathcal{S}}\def\TT{\mathcal{T}}\def\UU{\mathcal{U}}
\def\VV{\mathcal{V}}\def\WW{\mathcal{W}}\def\XX{\mathcal{X}}
\def\YY{\mathcal{Y}}\def\ZZ{\mathcal{Z}}
\def\Rho{\rho}

\def\Ab{\mathbf{A}}\def\Bb{\mathbf{B}}\def\Cb{\mathbf{C}}
\def\Db{\mathbf{D}}\def\Eb{\mathbf{E}}\def\Fb{\mathbf{F}}
\def\Gb{\mathbf{G}}\def\Hb{\mathbf{H}}\def\Ib{\mathbf{I}}
\def\Jb{\mathbf{J}}\def\Kb{\mathbf{K}}\def\Lb{\mathbf{L}}
\def\Mb{\mathbf{M}}\def\Nb{\mathbf{N}}\def\Ob{\mathbf{O}}
\def\Pb{\mathbf{P}}\def\Qb{\mathbf{Q}}\def\Rb{\mathbf{R}}
\def\Sb{\mathbf{S}}\def\Tb{\mathbf{T}}\def\Ub{\mathbf{U}}
\def\Vb{\mathbf{V}}\def\Wb{\mathbf{W}}\def\Xb{\mathbf{X}}
\def\Yb{\mathbf{Y}}\def\Zb{\mathbf{Z}}

\def\ab{\mathbf{a}}\def\bb{\mathbf{b}}\def\cb{\mathbf{c}}
\def\db{\mathbf{d}}\def\eb{\mathbf{e}}\def\fb{\mathbf{f}}
\def\gb{\mathbf{g}}\def\hb{\mathbf{h}}\def\ib{\mathbf{i}}
\def\jb{\mathbf{j}}\def\kb{\mathbf{k}}\def\lb{\mathbf{l}}
\def\mb{\mathbf{m}}\def\nb{\mathbf{n}}\def\ob{\mathbf{o}}
\def\pb{\mathbf{p}}\def\qb{\mathbf{q}}\def\rb{\mathbf{r}}
\def\sb{\mathbf{s}}\def\tb{\mathbf{t}}\def\ub{\mathbf{u}}
\def\vb{\mathbf{v}}\def\wb{\mathbf{w}}\def\xb{\mathbf{x}}
\def\yb{\mathbf{y}}\def\zb{\mathbf{z}}

\def\Abb{\mathbb{A}}\def\Bbb{\mathbb{B}}\def\Cbb{\mathbb{C}}
\def\Dbb{\mathbb{D}}\def\Ebb{\mathbb{E}}\def\Fbb{\mathbb{F}}
\def\Gbb{\mathbb{G}}\def\Hbb{\mathbb{H}}\def\Ibb{\mathbb{I}}
\def\Jbb{\mathbb{J}}\def\Kbb{\mathbb{K}}\def\Lbb{\mathbb{L}}
\def\Mbb{\mathbb{M}}\def\Nbb{\mathbb{N}}\def\Obb{\mathbb{O}}
\def\Pbb{\mathbb{P}}\def\Qbb{\mathbb{Q}}\def\Rbb{\mathbb{R}}
\def\Sbb{\mathbb{S}}\def\Tbb{\mathbb{T}}\def\Ubb{\mathbb{U}}
\def\Vbb{\mathbb{V}}\def\Wbb{\mathbb{W}}\def\Xbb{\mathbb{X}}
\def\Ybb{\mathbb{Y}}\def\Zbb{\mathbb{Z}}

\def\Af{\mathfrak{A}}\def\Bf{\mathfrak{B}}\def\Cf{\mathfrak{C}}
\def\Df{\mathfrak{D}}\def\Ef{\mathfrak{E}}\def\Ff{\mathfrak{F}}
\def\Gf{\mathfrak{G}}\def\Hf{\mathfrak{H}}\def\If{\mathfrak{I}}
\def\Jf{\mathfrak{J}}\def\Kf{\mathfrak{K}}\def\Lf{\mathfrak{L}}
\def\Mf{\mathfrak{M}}\def\Nf{\mathfrak{N}}\def\Of{\mathfrak{O}}
\def\Pf{\mathfrak{P}}\def\Qf{\mathfrak{Q}}\def\Rf{\mathfrak{R}}
\def\Sf{\mathfrak{S}}\def\Tf{\mathfrak{T}}\def\Uf{\mathfrak{U}}
\def\Vf{\mathfrak{V}}\def\Wf{\mathfrak{W}}\def\Xf{\mathfrak{X}}
\def\Yf{\mathfrak{Y}}\def\Zf{\mathfrak{Z}}

\def\lljav#1{\textcolor{black}{#1}}
\def\bjav#1{\textcolor{black}{#1}}
\def\daku#1{{#1}}

\maketitle

\tableofcontents

\begin{abstract}
We generalize subgraph densities, arising in dense graph limit theory, to
Markov spaces (symmetric measures on the square of a standard Borel space).
More generally, we define an analogue of the set of homomorphisms in the form
of a measure on maps of a finite graph into a Markov space. The existence of
such homomorphism measures is not always guaranteed, but can be established
under rather natural smoothness conditions on the Markov space and sparseness
conditions on the graph. This continues a direction in graph limit theory in
which such measures are viewed as limits of graph sequences.

\end{abstract}

\section{Introduction}

Dense graph limit theory is arguably the most complete graph limit theory:
There is a rather satisfactory duality between the local and global points of
view; subgraph densities and large scale structures (such as Szemer\'edi
partitions) are connected via the counting lemma and the inverse counting
lemma; limit objects, called graphons, are well known. Furthermore, the problem
of soficity does not arise: every potential limit object is the limit of finite
graphs.

Substantial work has been done pushing these things into the sparse regime.
Graphons are bounded functions on $[0,1]^2$, and a natural next step is to
explore the regime of ``unbounded graphons''. Borgs, Chayes, Cohn and Zhao
\cite{BCCZ} extended various results in dense graph limit theory to
``$L^p$-graphons'' (symmetric functions in $L^p([0,1]^2)$); here only
degree-restricted simple graphs were guaranteed to have finite densities. The
authors \cite{KKLSz} introduced a very general framework that was among other
things meant to encode homomorphism convergence of multigraphs, and allows for
finite densities for all decorated graphs in all limit objects. In the simplest
case, this corresponds to symmetric functions in
$L^\omega=\bigcap_{p\in[1,\infty)}L^p([0,1]^2)$, which is the largest function
space in which all elements have finite densities for all simple graphs. Other
work in this direction includes \cite{BCCH,BCCL,VRoy,Frenk}.

To go beyond unbounded graphons, the authors \cite{KLSz1} developed a limit
theory for not necessarily dense graphs, in which limit objects are symmetric
measures on $[0,1]^2$ called ``s-graphons''. (The $[0,1]$ interval can be
replaced by any standard Borel space.) Backhausz and Szegedy \cite{BackSz}
developed a stronger convergence theory with similar
limit objects, which they call ``graphops''. While these approaches have the
potential to unify various branches of graph limit theory, both of them are
based on convergence notions which could be called ``global convergence'' or
``right convergence''. The local point of view seems to be lost: subgraph
densities and subgraph distributions in general have not been defined in
symmetric measures on $[0,1]^2$.

Graphons or more generally unbounded graphons correspond to measures on
$[0,1]^2$ that are absolutely continuous with respect to the uniform measure.
{\it The main purpose of this paper is to study local aspects of graph limit
theory for singular measures.} Our results show that this is possible as long
as the measure has certain smoothness properties, whilst the graph to be mapped
has certain sparseness properties. The smoother the measure, the more finite
graphs will have well-defined densities in them. This leads to a remarkable
hierarchical viewpoint on graph limit theory, where smoothness of limit objects
corresponds to certain sparsity properties of graph sequences. At the top of
this hierarchy are the bounded and the $L^\omega$-graphons of \cite{KKLSz} as
the smoothest objects. $L^p$-graphons from \cite{BCCZ} form the next level of
smoothness.

The hope that one may extend local properties to singular measures has already
emerged in a previous work by the authors of the present paper. In \cite{KLSz2}
we investigated random orthogonal representations of finite graphs by vectors
on $n$-dimensional unit sphere $S^n$. As it turns out, such representations can
also be viewed as random homomorphisms into a singular measure defined on
$S^n\times S^n$: namely, the uniform distribution $\eta_n$ on orthogonal pairs
of vectors in $S^n$. Quite surprisingly, for every finite graph $H$ and
sufficiently large natural number $n$ one can introduce a robust notion for the
density of $H$ in $\eta_n$. Moreover one can introduce a measure on copies of
$H$ in $\eta_n$; if the total measure is finite, then one can normalize it to a
probability measure. See Section \ref{SSEC:ORTH} and also \cite{KLSz2} as
source of concrete, illustrative examples supporting the more general and
abstract content of the present paper.

To keep our treatment relatively simple, in this paper we address a special
case of s-graphons, which we call Markov spaces and (in their bipartite
version) bi-Markov spaces. A Markov space is a standard Borel sigma-algebra $(J,\BB)$
endowed with a probability measure  $\eta$ on $\BB\times\BB$. We restrict our
attention to symmetric measures on $\BB\times\BB$. We will denote the marginal
distribution of $\eta$ on $J$ by $\pi$. For a (finite) graph $G$, the uniform
distribution on $E(G)$ defines a Markov space. Markov spaces are essentially
equivalent to reversible Markov chains with a specified stationary
distribution. Graphops and s-graphons can be obtained by adding a probability
measure on the points, generalizing the uniform distribution on the nodes of a
graph. See also Remark \ref{REM:NODE-MEASURE}.

\medskip

We address the following three questions:

\medskip

(i) {\it How to define a reasonable notion of the density of a graph $G=(V,E)$
in a Markov space $(J,\BB,\eta)$?}

\medskip

Subgraph densities play a crucial role in graph limit theory, in the definition
of local convergence, extremal graph theory and graph property testing, just to
name a few applications; they also arise as Feynman integrals in quantum
physics (see e.g.~\cite{GliJaf}, Section 8.2). Subgraph densities can be viewed
as analogues of the moments of functions defined on product
spaces (cf.~\cite{LSz2010} and \cite{HomBook}, Appendix A4). We can calculate
densities of finite graphs in analytic objects representing graph limits such
as graphons and graphings. More general Markov spaces are also known to
represent limits of finite graphs, but the right notion for subgraph density is
still missing. Examples can be given showing that subgraph densities satisfying
reasonable conditions cannot be defined in full generality.

\medskip

(ii) {\it How to define the homomorphism set $\Hom(G,\eta)$ where $G$ is a
finite graph and $(J,\BB,\eta)$ is a Markov space?}

\medskip

If $H$ is a simple finite graph then $\Hom(G,H)$ is a subset of the set
$V(H)^V$ of all maps from $V=V(G)$ to $V(H)$. However, if we consider an
edge-weighted graph $H$, then there is no general, natural way to interpret
$\Hom(G,H)$ as a subset of $V(H)^V$. Rather, the edge weights induce a function
on $V(H)^V$, the function value being the product of the edge weights of the
images of the edges of $G$ under the corresponding vertex map. More generally,
if $(J,\BB,\pi)$ is a probability space and $W:~J^2\to [0,1]$ is a graphon,
then our interpretation of $\Hom(G,W)$ is a measure $\eta^G$ on $J^V$ whose
density function (Radon--Nikodym derivative) with respect to $\pi^V$ is the
function
\begin{equation}\label{wgformula}
W^G(x_1,x_2,\dots,x_n):=\prod_{(i,j)\in E(G)}W(x_i,x_j),
\end{equation}
where $V=\{1,2,\dots,n\}$. With this definition, the total measure
$\eta^G(J^V)$ is the familiar homomorphism density $t(G,W)$. If we apply this
definition to a graphon that represents a finite graph $H$ by its adjacency
function $V(H)\times V(H)\to\{0,1\}$, then we obtain the counting measure on
$\Hom(G,H)$ normalized by the number $|V(H)|^{|V(G)|}$ of all maps from $V(G)$
to $V(H)$.

Our goal is to introduce similar measures representing $\Hom(G,\eta)$ for
Markov spaces. The fact that generalized homomorphism sets are represented by
measures and not by sets is perfectly in line with the fact that the "edge set"
of a Markov space is not a set either: It is represented by the measure $\eta$
which tells us how to choose a random edge. Unfortunately, the product formula
in \eqref{wgformula} does not make sense if $\eta$ is singular with respect to
$\pi^2$, and so we have to use different methods to define $\eta^G$.

Our main approach relies on axiomatizing the properties of homomorphism
measures. We introduce some relatively simple and natural properties (related
to, but different from, the notion of a Markov random field; see Appendix
\ref{APP:B}) that are strong enough to uniquely define the measures $\eta^G$.
This also allows us to define the subgraph density
\[
t(G,\eta)=\eta^G(J^V),
\]
answering (i) in this case. We warn that $t(G,\eta)$ can be infinite. However,
if the total measure $t(G,\eta)$ is finite, then we can turn the measure
$\eta^G$ into a probability measure by normalizing it. These normalized
versions can then be used to define random copies of $G$ in $\eta$.

\medskip

(iii) {\it Can Markov spaces be approximated by a sequence $G_1,G_2,\dots$ of
finite graphs, so that the density of every finite graph $F$ in $G_n$ (suitably
normalized) tends to the density of $F$ in the limit space?}

\medskip

A Markov space that can be approximated this way will be called {\it sofic}. In
the case of dense graphs, the limit objects (graphons) are sofic; this takes an
easy construction via sampling. In the case of bounded-degree graphs, soficity
of the limit objects (involution-invariant distributions or graphings) is the
famous Aldous--Lyons conjecture, which is stronger than the soficity problem
for finitely generated groups.

\bigskip

We offer two approaches (and their combination) to these problems.

\medskip

(a) The first approach builds on the fact that the generalization of (i) to
(ii) allows for a recursive definition of these ``Hom-measures''. The
``axioms'' for these measures enable us to build up the measure corresponding
to a graph $G$ recursively from the measures pertaining to smaller graphs, by
attaching their nodes one-by-one. The independence of the construction from the
order in which the graph is built up is the main difficulty of this approach,
and in fact it does not hold in general (see Example \ref{EXA:KAB-ORT} below).
We can prove this independence for triangle-free graphs (under smoothness
assumptions on the measure $\eta$).

\medskip

(b) In the second approach, we consider approximations of $\eta$ by sequences
of graphons. By considering the densities of subgraphs within each graphon of
the sequence, and taking their limit, one naturally obtains a notion of
subgraph densities (more generally, homomorphism measures) in $\eta$. However,
in order to obtain a robust, well-defined notion through this approach, we have
to make sure that subgraph densities in these approximating sequences have a
limit, and that this limit is independent of the sequence considered. This
independence also hinges on certain smoothness properties of $\eta$.

Soficity is clearly related to our second approach, the discretization of the
Markov space, which can be used to produce a sofic approximation.

The equivalence of these two approaches is a nontrivial problem that is also
addressed in this paper. As remarked above, our methods do not work in full
generality; very likely there is some theoretical limitation on how far one can
go with defining $\eta^G$ in arbitrary Markov spaces. However, the full
analysis of this problem is left as an important open question.

\medskip

In the next part of the introduction we will state our main definitions and
results more precisely. We start with our definition of discretized Markov
spaces. Let $\PP=\{J_1,J_2,\dots,J_n\}$ be a finite measurable partition of the
space $J$ such that every partition class has positive $\pi$-measure. If the Markov
space is given by the measure $\eta$ on $J\times J$, then it makes sense to
``project" $\eta$ to $\PP$. When restricted to a product set $J_i\times J_j$,
the new measure $\eta_\PP$ is a scaled version of $\pi^2$ such that
$\eta_\PP(J_i\times J_j)=\eta(J_i\times J_j)$ holds. In other words, the
Radon--Nikodym derivative $W$ of $\eta_\PP$ with respect to $\pi^2$ is a
graphon whose value on $J_i\times J_j$ is constant $\eta(J_i\times
J_j)\pi(J_i)^{-1}\pi(J_j)^{-1}$.

We call a sequence of partitions of $J$ a {\it generating partition sequence},
if the partition classes are Borel, have positive $\pi$-measure, each partition
is a refinement of the previous one, and the partition classes generate all
Borel sets. \daku{ If we are more interested in generating the measure algebra
rather than the Borel sets proper, i.e., we only require the partition classes
to generate a sigma-algebra whose $\pi$-completion contains all Borel sets, we
obtain the slightly more general class of {\it exhausting partition sequences}
(see Section \ref{SEC:PARTITIONS}). }

Let $(\PP_i)_{i=1}^\infty$ be an exhausting partition sequence. We can then try
to define $\eta^G$ as the limit of homomorphism measures of $G$ in the graphons
$\eta_{\PP_i}$. The existence of these limits and the independence from the
chosen partition sequence is nontrivial and not always true.

The definition of what we mean by the convergence $\eta_{\PP_i}^G\to\eta^G$
should also be clarified.
Let $(J^m,\BB^m)$ be the product
sigma-algebra of $m$ copies of $(J,\BB)$. A set of the form
$C=B_1\times\dots\times B_m$, where $B_i\in\BB$, will be called a {\it box}.
For measures $(\mu,\mu_1,\mu_2,\dots)$ on $\BB^m$, the relation $\mu_n\to\mu$
{\it on boxes} means that $\mu_n(C)\to\mu(C)$ on every box $C$. This is a
rather weak notion of convergence, and in fact it is equivalent to weak
convergence if we put a compact topology on $(J,\BB)$, and all of the measures
$\mu_n$ as well as $\mu$ have the same marginals -- this is left as an exercise to the reader.

If $\eta_{\PP_i}^G\to\eta^G$ on boxes for every exhausting partition sequence,
then we say that $\eta^G$ is {\it partition approximable}. \daku{Note that this in particular means that $t(G,\eta_{\PP_i})\to t(G,\eta)$.}

Since graphons can be approximated by finite graphs via sampling, and the
homomorphism measures of these finite graphs approximate the homomorphism
measure of the graphon, the results on partition approximability of $\eta^G$ can be interpreted as a partial answer to the soficity problem (iii).

\medskip

Now we turn to the definitions needed to generalize homomorphism sets. As we
mentioned above, our goal is to construct measures $\eta^{G[S]}$ on $\BB^S$ for
each induced subgraph $G[S]$ of a graph $G=(V,E)$. This measure should depend
on the induced subgraph $G[S]$ only\footnote{More formally, if $f:S_1\to S_2$ is an isomorphism between $G[S_1]$ and $G[S_2]$, then it contra-variantly induces a function $f^\#:J^{S_2}\to J^{S_1}$ by $f^\#(x)_v:=x_{f(v)}$, and we require the pushforward measure $f_*^\#\mu^{G[S_2]}:=\mu^{G[S_2]}(f^\#)^{-1}$ to be equal to $\mu^{G[S_1]}$.}. Intuitively, the measure $\eta^G$
represents some kind of normalized homomorphism counting of $G$ in $\eta$.

To motivate our approach, consider a graphon $W:~J^2\to\mathbb{R}_+$,
representing $\eta$. Then $\eta^G$ is the measure on $J^V$ whose Radon--Nikodym
derivative with respect to $\pi^V$ is equal to $W^G$ (see \eqref{wgformula}).
These measures satisfy a certain log-modularity property, which relates
$\eta^G$ to measures corresponding to smaller graphs. Assume that $V(G)=V=U\cup
T$ such that there is no edge between $U\setminus S$ and $T\setminus S$, where
$S=U\cap T$. Then $W^GW^{G[S]}=W^{G[U]}W^{G[T]}$. We can rewrite this equation:
\begin{equation}\label{prodprop}
W^G/W^{G[S]}=(W^{G[U]}/W^{G[S]})(W^{G[T]}/W^{G[S]}).
\end{equation}
Note that these quotients have a natural meaning even if $W$ is allowed to
vanish. For example
\[
(W^G/W^{G[S]})(x_1,x_2,\dots,x_n)=\prod_{(i,j)\in E(G)\setminus E(S)}W(x_i,x_j).
\]
One of the key observations is that equation (\ref{prodprop}) has a measure
theoretic interpretation that allows us to extend it to singular measures. The
function $W^G/W^{G[S]}$ can be interpreted as a disintegration of the measure
$\eta^G$ with respect to $\eta^{G[S]}$ (see Proposition \ref{PROP:DISINT}). The
only condition that we need for this type of disintegration is that the
marginal of $\eta^G$ on $J^S$ be absolutely continuous with respect to
$\eta^{G[S]}$. We will call this property of the family of measures {\it
decreasing}.

For a Markov space $\eta$ for which the homomorphism measures $\eta^{G[S]}$ are
defined and have the decreasing property, disintegration yields a family of
measures $\nu_{S,T,x}$, where $x\in J^S$ $(S\subseteq T\subseteq V)$, and
$\nu_{S,T,x}$ is a measure on $J^{T\setminus S}$.

In particular, when $U\subseteq V$ is such that $S=U\cap T$, $V=U\cup T$ and there are no edges between $T\setminus U$ and $U\setminus T$, the equation
(\ref{prodprop}) translates to the condition
\begin{equation}\label{EQ:MARKOV}
\nu_{S,V,x}=\nu_{S,U,x}\times\nu_{S,T,x}.
\end{equation}
In the special case when $S=\emptyset$, this means that the measure assigned to
the disjoint union of two graphs is the product of the measures assigned to
them. We will call \eqref{EQ:MARKOV} the {\it Markovian property} of the family
of the measures. As stated above, this type of Markovian property is not simply
a property of a measure $\eta^G$ by itself, but instead it describes how
various measures $\eta^G$ corresponding to a graph and its induced subgraphs
are related to each other.

In addition, we impose the natural condition that the family of measures
$(\eta^{G[S]}:~S\subseteq V)$ is {\it normalized}, in the sense that
$\eta^{K_2}=\eta$ for a single edge $K_2$ and $\eta^{K_1}=\pi$ for a single
node $K_1$. We say that {\it $G$ is well-measured in $\eta$} if there is a
family of measures $(\eta^{G[S]}:~S\subseteq V)$ that is normalized, decreasing
and Markovian. It will be an important additional property that $\eta^G$ is
finite. This easily implies that all other measures $\eta^{G[S]}$ are finite. In
this case we say that $G$ is well-measured in $\eta$ with finite density.

This Markovian property concept \eqref{EQ:MARKOV} allows for a recursive
construction of measures for a graph utilizing measures of smaller graphs,
decomposing $G$ along a cutset. To initialize the construction, we need the
normalized property. To apply it, we need a proper cutset of nodes in $G$; this
is not available for complete graphs, and this is our main reason for
\lljav{having to exclude} triangles.

\lljav{This recursive construction has important consequences.}

\begin{theorem}\label{THM:UNIQUE} \lljav{If $G$ is a triangle-free graph, and
there is a normalized, decreasing and Markovian family of measures on its
induced subgraphs, then this family is uniquely determined.}
\end{theorem}

\lljav{This implies, in particular, that $\eta^{G[S]}$ depends on the induced
subgraph $G[S]$ only. More precisely, if $G[S]$ and $G[T]$ are isomorphic
induced subgraph and $\xi:~S\to T$ is an isomorphism, then the pushforward of
$\eta^{G[S]}$ to $J^T$ is $\eta^{G[T]}$.}

\lljav{We will use this construction to prove the existence of such families of
measures,} but it will be a nontrivial question under what conditions are the
measures independent of the choice of the particular way of building up the
graph.

To guarantee the decreasing property for our measures thus constructed we will
have to assume the decreasing property for small stars, which translates to the
following ``smoothness'' property of the Markov space $\eta$. Choose a point
$x$ from the stationary distribution $\pi$, and make $k$ independent single
steps $y_1,\dots,y_k$ each starting from $x$ so that $(x,y_i)\sim\eta$, and the $y_i$'s are conditionally mutually independent given $x$. Let $\sigma_k$ denote the joint
distribution of $(y_1,\dots,y_k)$. We say that the Markov space is {\it
$k$-loose}, if $\sigma_k$ is absolutely continuous with respect to $\pi^k$. We
shall also make use of a further refinement of this notion: a $k$-loose Markov
space is {\it $(k,p)$-loose}, if the Radon--Nikodym derivative
$d\sigma_k/d\pi^k$ is in $L^p(\pi^k)$. This technical condition will turn out
to be equivalent to the property that the complete bipartite graph $K_{k,p}$ is
well-measured in $\eta$ with finite density (see Corollary \ref{COR:KAB-WELL}).

The following result (see Section \ref{SSEC:TRIANGLE-FREE}) will allow us to
define the measure $\eta^G$ (which is not a finite measure in general, see
Example \ref{EXA:GRAPHON}).

\begin{theorem}\label{THM:MAIN-KAB}
Let $G=(V,E)$ be a triangle-free graph, and let $\Mb=(J,\BB,\eta)$ be a Markov
space such that every complete bipartite subgraph $K_{a,b}$ of $G$ is
well-measured in $\Mb$. Then $G$ is well-measured in $\Mb$.
\end{theorem}

The condition implies that the Markov chain is $k$-loose, where $k$ is the
maximum degree of $G$. If the graph contains no $4$-cycles, then the stars are
the only complete bipartite subgraphs; hence we can state the following
corollary.

\begin{corollary}\label{COR:MEASURE}
Let $(J,\BB,\eta)$ be a $k$-loose Markov space. Then every graph of girth at
least $5$ and with all degrees at most $k$ is well-measured in $\eta$.
\end{corollary}

It will turn out that densities of bipartite graphs are much better behaved,
and we have more transparent formulas for them. Using these formulas, we will
prove the following  (see Sections \ref{SSEC:PART-APPROX} and
\ref{subsect:thm1.4}).

\begin{theorem}\label{THM:DENSITY}
Let $(J,\BB,\eta)$ be an $(a,b)$-loose Markov space. Let $G=(V,E)$ be a
bipartite graph with bipartition $V=U\cup W$ such that $\deg(w)\leq a$ for all
$w\in W$ and $\deg(u)\leq b$ for all $u\in U$. \lljav{Then $G$ is well-measured
in $\eta$ with finite density, and $\eta^G$ is partition approximable.}
\end{theorem}

Densities of cycles are particularly interesting because of their connection
with operator theory. Every Markov space $(J,\BB,\eta)$ acts naturally as a
bounded operator $\Ab_\eta$ on $L^2(J,\pi)$. We prove the next theorem (see
Theorem \ref{cycdense2}).

\begin{theorem}\label{THM:SCHATTEN}
If the $k$-th Schatten norm of $\Ab_\eta$ is finite for some $k\in\mathbb{N}$,
then $C_k$ is well-measured in $\eta$ \lljav{with finite density}, $\eta^{C_k}$
is partition approximable, and $t(C_k,\eta)=\tr(\Ab_\eta^k)$.
\end{theorem}

Note that the finiteness of the $k$-th Schatten norm implies that $\Ab_\eta$ is
a compact operator with eigenvalues $\{\lambda_i\}_{i=1}^\infty$ such that the
series $\sum_{i=1}^\infty \lambda_i^k$ is absolutely convergent, and we have
$\tr(\Ab_\eta^k)=\sum_{i=1}^\infty \lambda_i^k$. This suggests a third way to
define the density of a graph in Markov spaces using spectral approximations
provided that the operator $\Ab_\eta$ is compact. This direction, however, is
not explored in this paper.

\section{Preliminaries}

\subsection{Notation}

All graphs considered are finite and simple. A {\it bipartite graph} is a graph
that is $2$-colorable. A {\it bigraph} is a bipartite graph with a fixed
bipartition, where the order of bipartition classes is also specified.
Formally, a bigraph is a triple $G=(U,W,E)$, where $E\subseteq U\times W$. Let
$K_{a,b}$ denote the complete bipartite graph with bipartition $U\cup W$, where
$|U|=a$ and $|W|=b$. With a slight abuse of notation, we also denote the
bigraph $(U,W,U\times W)$ by $K_{a,b}$.

For a map $x\in J^V$, we denote by $x_V$ the image of $V$ under this map, as a
labeled set $(x_v:~v\in V)$.

For a measurable space $(I,\AA)$, we denote by $\Mf(I,\AA)$ (or simply by
$\Mf(\AA)$) the set of all finite measures on $\AA$. If $\mu\in\Mf(\AA)$ and
$f\in L^1(\mu)$, then we define the measure $f\cdot\mu$ and the number $\mu(f)$
by
\[
(f\cdot\mu)(B) = \intl_B f \,d\mu,\qquad \mu(f)= (f\cdot\mu)(J)=\intl_J f \,d\mu.
\]

If $(I,\AA,\pi)$ is a probability space and $f,g:~I\to\R$ are measurable
functions, then we define
\[
\langle f,g\rangle = \langle f,g\rangle_\pi = \intl_J f(x)g(x)\,d\pi(x).
\]
If $V$ is a finite set, $\emptyset\not=S\subseteq V$, and $\mu$ is a measure on
$\AA^V$, then we denote by $\mu^S$ the marginal of $\mu$ on $\AA^S$.

\subsection{Derivative and disintegration}\label{SEC:DISINT}

Let $\mu$ and $\nu$ be two measures on the same Borel space $(J,\BB)$. We say
that a function $f:~J\to[0,\infty]$ is the {\it Radon--Nikodym derivative} of
$\nu$ with respect to $\mu$, denoted by $f=d\nu/d\mu$, if $\nu=f\cdot\mu$. Note
that we allow infinite values for $f$, under the convention $\int_B \infty
\,d\mu=0$ whenever $\mu(B)=0$. The existence of the Radon--Nikodym derivative is
usually stated for two sigma-finite measures, but we'll need a slightly more
general fact (see Appendix \ref{APP:A}).

Let $(I,\AA)$ and $(J,\BB)$ be measurable spaces, and let $\Phi=(\mu_x:~x\in
I)$ be a family of sigma-finite measures on $(J,\BB)$. We say that $\Phi$ is a
{\it measurable family}, if $\mu_x(B)$ is a measurable function of $x$ for
every $B\in\BB$. \daku{Note that if $\Phi$ is a family of \emph{probability}
measures, then this essentially corresponds to a Markov kernel (see Subsection
\ref{SEC:MARKOV}).}

We need the following version of the Disintegration Theorem (which is usually
stated for the case when $\alpha=\gamma^1$); see e.g. \cite{ChPol} or
\cite{Bog}, Section 10.6.

\begin{prop}\label{PROP:DISINT}
Let $(I,\AA)$ and $(J,\BB)$ be standard Borel spaces. Let $\alpha$ be a
sigma-finite measure on $\AA$, and let $\gamma$ be a sigma-finite measure on
$\AA\times\BB$. Then there is a measurable family $\Phi$ of measures on $\BB$
such that
for every
bounded measurable function $f:~I\times J\to[0,\infty)$,
\begin{equation}\label{EQ:DISINT-D}
\intl_{I\times J} f(x,y) d\gamma(x,y)=\intl_{I}\intl_{J}f(x,y) d\mu_x(y)d\alpha(x).
\end{equation}
if and only if the marginal $\gamma^1$ of $\gamma$ on $\AA$ is absolutely
continuous with respect to $\alpha$. Furthermore, the measurable family
$\Phi=(\mu_x:~x\in I)$ is uniquely determined up to changing $\mu_x$ for $x$ in
a zero $\alpha$-measure subset of $I$.
\end{prop}

We say that the measurable family $\Phi$ is a {\it disintegration} of the
measure $\gamma$ with respect to the measure $\alpha$.

A key component of our constructions will be a ``reverse'' of the
disintegration, essentially integrating a measurable family $\Phi$ with respect
to a suitable measure $\alpha$ to obtain a sigma-finite measure on the product
space. Indeed, consider a measurable family $\Phi$ of finite measures on $\BB$.
For $\alpha\in\Mf(\AA)$, define
\begin{equation}\label{EQ:DISINT}
\alpha[\Phi](A\times B) = \intl_A \mu_x(B)\,d\alpha(x) \qquad(A\in\AA,~B\in\BB).
\end{equation}
If the measures $\mu_x$ are finite and uniformly bounded, this extends to a
finite measure $\alpha[\Phi]$ on $\AA\times\BB$, whose disintegration with
respect to $\alpha$ is trivially $\Phi$.

The marginal of $\alpha[\Phi]$ on $\AA$ is the measure $g\cdot\alpha$, where
$g(x)=\mu_x(J)$. The marginal of $\alpha[\Phi]$ on $\BB$ is the mixture of
$\Phi$ by $\alpha$. The definition also implies that if all of the $\mu_x$, as
well as $\alpha$, are probability distributions, then so is $\alpha[\Phi]$, and
$\alpha$ is the marginal of $\alpha[\Phi]$ on $\AA$. Conversely, if
$\alpha[\Phi]$ and $\alpha$ are probability distributions, then $\mu_x$ is a
probability distribution for $\alpha$-almost all $x\in J$.

An important example of this construction will be the family of distributions
of transition probabilities in a Markov chain; see Section \ref{SEC:MARKOV}
below.

The sigma-finite extension also goes through in case we drop the uniform
boundedness condition, by partitioning $I$ into countably many measurable parts
corresponding to the level sets $[k,k+1)$ ($k\in\mathbb{N}_0$) of the total
measure function $\mu_x(J)$. On each such set uniform boundedness is satisfied,
and the above applies.

Note, however, that for general families $\Phi$ of sigma-finite measures, one
quickly encounters technical difficulties with the extension. Fortunately, as
shown in the following lemma, a family $\Phi$ that arises from a disintegration
is well-behaved.

\begin{lemma}\label{LE:DISINT}
Let $(I,\AA)$ and $(J,\BB)$ be standard Borel spaces. Let $\alpha_1,\alpha_2$
be sigma-finite measures on $\AA$ with $\alpha_2$ absolutely continuous with
respect to $\alpha_1$, and let $\gamma$ be a sigma-finite measure on
$\AA\times\BB$ such that the marginal $\gamma^1$ of $\gamma$ on $\AA$ is
absolutely continuous with respect to $\alpha_1$. Let $\Phi=(\mu_x:~x\in I)$ be
disintegration of $\gamma$ with respect to $\alpha_1$. Then $\alpha_2[\Phi]$
defined via \eqref{EQ:DISINT} extends to a sigma-finite measure on
$\AA\times\BB$.
\end{lemma}

\begin{proof}
As $\alpha_2$ is absolutely continuous with respect to $\alpha_1$, and both are
sigma-finite, we may write $I$ as a countable disjoint union $\bigcup_{k} I_k$
of measurable sets with $\alpha_j(I_k)<\infty$ for all $k$ and $j=1,2$, and it
suffices then to prove the existence of the appropriate extension on each
$I_k\times J$. We may therefore without loss of generality restrict our
attention to the case of both $\alpha_1$ and $\alpha_2$ being finite.

As $\gamma$ is sigma-finite, consider a partition of $I\times J$ into a
countable disjoint union $\bigcup_{k} G_k$ of measurable sets with
$\gamma(G_k)$ finite. Let $\gamma_{G_k}$ be the restriction of $\gamma$ to
$G_k$, and $\Phi_k=(\mu_{k,x}:~x\in I)$ its disintegration with respect to
$\alpha_1$. Note that
\[
\intl_{I} \mu_{k,x}(J) d\alpha(x)=\intl_{I}\intl_{J}1 d\mu_{k,x}(y)d\alpha(x)
=\intl_{I\times J} 1 d\gamma_{G_k}(x,y)=\gamma_{G_k}(I\times J)
\]
is finite, hence we have that $\mu_{k,x}$ is finite for $\alpha_1$-a.e. $x\in I$.

Since the $G_k$'s are disjoint, we have that $\gamma_{G_k}\perp\gamma_{G_\ell}$
for any $k\neq\ell$, and thus also for $\alpha_1$-a.e. $x\in I$,
$\mu_{k,x}\perp\mu_{\ell,x}$. Consequently, we have that $\mu_x=\sum_{k}
\mu_{x,k}$ and it is sigma-finite for $\alpha_1$-a.e. $x\in I$. Since
$\alpha_2$ is absolutely continuous with respect to $\alpha_1$, for product
sets $A\times B$ ($A\in\AA, B\in\BB$), we by \eqref{EQ:DISINT} clearly have
\[
\alpha_2[\Phi](A\times B)=\sum_k \alpha_2[\Phi_k](A\times B).
\]
Since each $\alpha_2[\Phi_k]$ extends to a sigma-finite measure on
$\AA\times\BB$, so does their countable sum $\alpha_2[\Phi]$.
\end{proof}

\begin{remark}
By the above, if $\Phi$ is the disintegration of a sigma-finite measure
$\gamma$ with respect to $\alpha$, we have $\gamma=\alpha[\Phi]$.
\end{remark}

\subsection{Markov property}

Let $(J,\BB)$ be a standard Borel space, and let $V$ be a finite set. Let
$\MM=(\mu_S\in \Mf(\BB^S):~S\subseteq V)$ be a family of measures. We say that
$\MM$ is {\it decreasing}, if for $S\subseteq T\subseteq V$, the marginal
$(\mu_T)^S$ is absolutely continuous with respect to $\mu_S$. A trivial example
of such a family is $\mu_S=\phi^S$ for any $\phi\in\Mf(\BB^V)$, which we call
the {\it marginal family} defined by $\phi$.

If $\MM=(\mu_S\in \Mf(\BB^S):~S\subseteq V)$ is a decreasing family of
sigma-finite measures, then for $S\subseteq T$, the Disintegration Theorem
(Proposition \ref{PROP:DISINT}) gives a measurable family of measures
$N_{S,T}=(\nu_{S,T,x}:~x\in J^S)$ on $\BB^{T\setminus S}$ such that
\begin{equation}\label{EQ:NST-DEF}
\mu_S[N_{S,T}] = \mu_T.
\end{equation}
This definition implies that the Radon--Nikodym derivative $d(\mu_T)^S/d\mu_S$
exists and it can be expressed as
\begin{equation}\label{EQ:RN-MT}
\frac{d(\mu_T)^S}{d\mu_S}(x) = \nu_{S,T,x}(J^{T\setminus S})
\end{equation}
for $\mu_S$-almost all $x\in J^S$.

We can informally think of $\nu_{S,T,x}$ as the measure on extensions of $x$
from $S$ to $T$. This motivates the following ``chain rule''. For $S,T\subseteq
V$ with $S\cap T=\emptyset$ and $x\in J^S$, $y\in J^T$, we denote by $xy\in
J^{S\cup T}$ the union of the maps $x$ and $y$. Then for $S\subseteq T\subseteq
U\subseteq V$, we can first extend $x\in J^S$ to an $xy\in J^T$, and then
extend $xy$ to $U$. Defining $N_{T,U,x}=(\nu_{S,T,xy}:~y\in J^{T\setminus S})$,
we can write this as
\begin{equation}\label{EQ:CHAIN}
\nu_{S,U,x} = \nu_{S,T,x}[N_{T,U,x}]
\end{equation}
for $\mu_S$-almost all $x\in J^S$. Indeed, for every $A\in \BB^S$,
$B\in\BB^{T\setminus S}$ and $C\in\BB^{U\setminus T}$, using
\eqref{EQ:DISINT-D},
\begin{align*}
\intl_A \nu_{S,U,x}(B\times C) \,d\mu_S(x) &= \mu_U(A\times B\times C) = \intl_{A\times B} \nu_{T,U,xy}(C)\,d\mu_T(xy)\\
&=  \intl_{A}\intl_{B} \nu_{T,U,xy}(C)\,d\nu_{S,T,x}(y)\,d\mu_S(x).
\end{align*}
This holds for every $A\in\BB^S$, which proves \eqref{EQ:CHAIN}.

\medskip

Let $G=(V,E)$ be a finite simple graph. Let $\MM=(\mu_S:~S\subseteq V)$ be a
family of sigma-finite measures, with the corresponding disintegrations
$N_{S,T}$. We say that $\MM$ is {\it Markovian}, or has the {\it Markov
property} (with respect to $G$), if it is decreasing, and for any two sets $U,
W\subseteq V$ and $S=U\cap W$ such that no edge connects $U\setminus S$ and
$W\setminus S$, and for $\mu_S$-almost all $x\in J^S$, we have
\begin{equation}\label{EQ:NU-INDEP}
\nu_{S,U\cup W,x}=\nu_{S,U,x}\times\nu_{S,W,x}.
\end{equation}

\begin{lemma}\label{LEM:MARKOV-YDEP}
A decreasing family $\MM=(\mu_S:~S\subseteq V)$ of sigma-finite measures has
the Markov property with respect to a graph $G$ if and only if
\begin{equation*}
\nu_{U, U\cup W,xy} = \nu_{U\cap W,W,x}
\end{equation*}
holds for all $U,W\subseteq V$ \lljav{with no edges connecting $U\setminus W$
and $W\setminus U$, for $\mu_{U\cap W}$-almost all $x\in J^{U\cap W}$ and for
$\nu_{U\cap W,U,x}$-almost all $y\in J^{U\setminus W}$.}
\end{lemma}

In particular, the measure on the left is independent of $y$ almost everywhere.

\begin{proof}
\lljav{To prove the necessity of the condition, let $U$ and $W$ be as in the
lemma, and set} $S=U\cap W$. Suppose that \eqref{EQ:NU-INDEP} holds, then for
all $B\in\BB^{U\setminus W}$ and $C\in\BB^{W\setminus U}$, and $\mu_S$-almost
all $x\in J^S$ we have
\[
\nu_{S,U\cup W,x}(B\times C)= \nu_{S,U,x}(B)\nu_{S,W,x}(C),
\]
but also by \eqref{EQ:CHAIN} and the chain rule,
\begin{align*}
\nu_{S,U\cup W,x}(B\times C) &= \nu_{S,U,x}[N_{U,U\cup W,x}](B\times C)\\
&= \intl_B \nu_{U,U\cup W,xy}(C)\,d\nu_{S,U,x}(y).
\end{align*}
It follows that
\[
\nu_{U,U\cup W,xy}(C) = \nu_{S,W,x}(C)
\]
must hold for all $C\in\BB^{W\setminus S}$, $\mu_S$-almost all $x\in J^S$ and
$\nu_{S,U,x}$-almost all $y\in J^{U\setminus S}$. This proves the necessity of
the condition in the Lemma. The reverse implication follows by a similar
computation.
\end{proof}

Markovian measure families are related to, but different from, Markov random
field on graphs. See Appendix \ref{APP:B} for the details of this connection
(which we don't use in this paper).

\subsection{Markov spaces, graphons and bigraphons}\label{SEC:MARKOV}

A {\it Markov space} consists of a sigma-algebra $(J,\BB)$, together with a
probability measure $\eta$ on $(J\times J,\BB\times\BB)$ whose marginals are
equal. In this paper, we assume that $(J,\BB)$ is a standard Borel
sigma-algebra. In the probability literature, $\eta$ is often called the {\it
ergodic flow}, or {\it ergodic circulation}, and its marginals
$\pi(A)=\eta(A\times J)=\eta(J\times A)$ are the {\it stationary distribution}
of the Markov space $(\BB,\eta)$. A Markov space is {\it symmetric}, if
$\eta(A\times B)=\eta(B\times A)$ for all $A,B\in\BB$. We note already here that beyond Remark \ref{REM:GRAPH2MARKOV}, all Markov spaces will be assumed to be symmetric unless stated otherwise.

Markov spaces are intimately related to Markov chains. A Markov chain is
usually defined on a sigma-algebra $(J,\BB)$, specifying a probability
measure $P_u$ on $\BB$ for every $u\in J$, called the {\it step distributions}.
One assumes that for every $A\in \BB$, the value $P_u(A)$ is a measurable
function of $u\in J$. The map $u\mapsto P_u$ is called a {\it Markov scheme} or
{\it Markov kernel}. To get a Markov space, we also have to assume that the
Markov chain has a {\it stationary distribution} $\pi$ on $\BB$ satisfying
\begin{equation}\label{EQ:PI-DEF1}
\intl_J P_u(A)\,d\pi(u) =\pi(A)
\end{equation}
for all $A\in\BB$, and we fix such a distribution. (A Markov scheme may have
none or more than one stationary distributions.) Then
\begin{equation}\label{EQ:MARKOV-DISINT}
\eta(A\times B) = \intl_A P_u(B)\,d\pi(u)\qquad(A,B\in\BB)
\end{equation}
defines a Markov space. Conversely, every Markov space arises from an
essentially unique Markov scheme this way; this can be constructed by
disintegrating $\eta$ with respect to $\pi$ (see Section \ref{SEC:DISINT}). The
Markov scheme is time-reversible precisely when this Markov space is symmetric.

As a generalization of the notion of bigraphs, we define a {\it bi-Markov space} as a
quintuple $\Mb=(I,J,\AA,\BB,\eta)$, where $(I,\AA)$ and $(J,\BB)$ are standard
Borel spaces, and $\eta$ is a probability measure on $\AA\times\BB$. We denote
the marginals of $\eta$ \daku{on $I$ and $J$ by $\pi_I$ and $\pi_J$},
respectively. While a bi-Markov space does not directly define a Markov chain, the
disintegration of $\eta$ according to $\pi_I$ still makes sense, and gives a
measurable family $(P_u:~u\in I)$ of measures on $(J,\BB)$ such that
\begin{equation}\label{EQ:MARKOV-DISINT2}
\eta(A\times B) = \intl_A P_u(B)\,d\pi_I(u)
\end{equation}
for $A\in\AA$ and $B\in\BB$, similarly to the symmetric case. However, from
a point $u\in I$ you step to a point $w\in J$, so the step cannot be repeated.

In a bigraph $G=(U,W,E)$, we can interchange the bipartition classes to obtain
another bigraph $G^*=(W,U,E^*)$, which is isomorphic to $G$ as an undirected
graph. Similarly, for every bi-Markov space $\Mb=(I,J,\AA,\BB,\eta)$, we can construct
the reverse bi-Markov space $\Mb^*=(J,I,\BB,\AA,\eta^*)$.

\begin{remark}\label{REM:GRAPH2MARKOV}
If we identify the Borel spaces $(I,\AA)$ and $(J,\BB)$ (which is usually
possible), we get an (asymmetric) Markov space, which is a generalization of
directed graphs (digraphs). A symmetric Markov space is a generalization of
undirected graphs, and a bi-Markov space is a generalization of bigraphs. If we
identify $(I,\AA)$ and $(J,\BB)$ and also assume that $\pi_I=\pi_J$, then we look
at a generalization of Eulerian digraphs; these are also equivalent to (not
necessarily reversible) Markov chains with a fixed stationary distribution.

Extending our results to digraphs (Eulerian or not) would be interesting, but
in this paper we only deal with Markov spaces generalizing undirected graphs
and bigraphs: symmetric Markov spaces and bi-Markov spaces. For the rest of this paper,
we drop the adjective ``symmetric''.
\end{remark}

Let $(J,\BB,\pi)$ be a standard Borel probability space, and let
$W:~J^2\to\R_+$ be a {\it graphon}, a symmetric integrable function with
respect to $\pi$. In the theory of dense graph limits, graphons are assumed to
be bounded by $1$, but since then, much of the theory has been extended to the
unbounded case \cite{KKLSz,BCCZ}. If a graphon is bounded, then it can be scaled to a
$1$-bounded graphon. We call $W$ {\it $1$-regular}, if $\int_J
W(x,y)\,d\pi(y)=1$ for all $x$.

Every $1$-regular graphon $W$ determines a Markov space
$\eta_W=W\cdot(\pi\times\pi)$. Trivially, $\eta_W$ is absolutely continuous
with respect to $\pi^2$. Conversely, if we have a Markov space for which $\eta$
is absolutely continuous with respect to $\pi\times\pi$, then the
Radon--Nikodym derivative $W=d\eta/d\pi^2$ is a corresponding $1$-regular
graphon.

Let $(I,\AA,\pi_I)$ and $(J,\BB,\pi_J)$ be standard Borel probability spaces. A
{\it bigraphon} is a bounded measurable function $W:~I\times J\to\R_+$. The
bigraphon is $1$-regular, if
\begin{equation}\label{EQ:W-MARGIN}
\intl_I W(x,\,\cdot\,)\,d\pi_I(x) = \intl_J W(\,\cdot\,,y)\,d\pi_J(y) =1.
\end{equation}
Every $1$-regular bigraphon defines a bi-Markov space by
\[
\eta = W\cdot(\pi_I\times\pi_J).
\]

\subsection{Graphops and linear functionals}

Let us survey some notions related to Markov spaces with a functional analysis
flavor; these were introduced in the theory of {\it action convergence}
\cite{BackSz}.

Every Markov space defines an operator $\Ab=\Ab_\eta:~L^1(\pi)\to L^1(\pi)$ by
\[
(\Ab_\eta f)(x)=\E(f(x')) = P_x(f) = \intl_J f(y)\,dP_x(y),
\]
where $x'$ is the point obtained by a random step from $x$. The integral on the
right is well-defined for $\pi$-almost-all $x\in J$. We call $\Ab$ the {\it
adjacency operator} of the Markov space. This operator is contractive with
respect to any $L^p$-norm $(p\ge 1)$. Hence it maps every subspace $L^p(\pi)$
into itself, and $\|\Ab\|_{p\to p}=1$ for every $p\in [1,\infty]$. The
adjacency operator is monotone, self-adjoint, and $1$-regular (which means that
$\one_J$ is an eigenfunction with eigenvalue $1$). A monotone and self-adjoint
bounded linear operator $L^\infty(\pi)\to L^1(\pi)$ is called a {\it graphop},
so the adjacency operator, restricted to $L^\infty(\pi)$, is a $1$-regular
graphop.

We also note that for every $B\in\BB$ and $\pi$-almost-all $x$,
\begin{equation}\label{EQ:GG-ONE}
(\Ab\one_B)(x)  = P_x(B),
\end{equation}
since for every $A\in\BB$,
\begin{align*}
\intl_A (\Ab\one_B)(x) \,d\pi(x) &= \intl_A\intl_J \one_B(y)\,dP_x(y)\,d\pi(x)
= \intl_{J^2}\one_A(x)\one_B(y)\,d\eta(x,y)\\
&= \eta(A\times B) = \intl_A P_x(B)\,d\pi(x).
\end{align*}
Theorem 6.3 in \cite{BackSz} implies that, conversely, every self-adjoint, monotone,
$1$-regular and contractive operator $\Ab:~L^p(J,\pi)\to L^p(J,\pi)$ $(p\ge 1)$
is the adjacency operator of a Markov space with stationary measure $\pi$.

It is clear that the $k$-th power of the adjacency operator is itself an
adjacency operator of a Markov space. In the Markov chain setting, this
corresponds to considering $k$ consecutive steps as one. The edge measure of
this new Markov space will be denoted by $\eta^k$.

If a Markov space is defined by an $L^2$-graphon (a function in $L^2(\pi^2)$),
then its adjacency operator $\Ab$ is a Hilbert-Schmidt operator, and hence it
is compact. It is well known that for a symmetric operator $\Ab$ on a Hilbert
space and any integer $k\ge 1$, $\Ab^k$ is compact if and only if $\Ab$ is
compact. Often we'll be concerned with Markov spaces for which a finite power
of $\Ab_\eta$ is defined by a graphon, and so $\Ab_\eta$ is a compact operator.
However, see Example \ref{EXA:POWERS} for a Markov space with an ``almost''
compact adjacency operator, to which extensions of our results would be
particularly desirable.

\begin{remark}\label{REM:NODE-MEASURE}
The finite version of the probability measure $\eta$ of a Markov space is the
uniform measure on the edges of a finite graph. The marginal $\pi$ is the
stationary distribution of the random walk, where the probability of a vertex
is proportional to its degree. It is natural to introduce the uniform measure
on the vertices as well. In the general case, this means endowing a Markov
space $(J,\BB,\pi)$ with an additional probability measure $\lambda$ on
$(J,\BB)$. This richer structure would then include non-regular graphons,
general (not necessarily $1$-regular) graphops, and s-graphons as defined in
\cite{KLSz1}. Putting it in a slightly sloppy form,
\[
\text{reversible Markov chain} + \text{stationary distribution} \cong \text{Markov space}
\]
and
\[
\text{Markov space} + \text{vertex distribution} \cong \text{graphop} \cong \text{s-graphon}.
\]
Extending the results of this paper to the case when a vertex-measure is present
is an important task for further research.
\end{remark}

For bi-Markov spaces, the operator $\Ab$ can be defined just as above, except
that $\Ab$ will not be self-adjoint.

\subsection{Partitions}\label{SEC:PARTITIONS}

Let $(J,\BB,\pi)$ be a standard probability space, and let
$\PP=\{J_1,\dots,J_n\}$ be a finite, measurable, non-degenerate partition of
$J$ (this means that $J_i\in\BB$ and $\pi(J_i)>0$). Let $\wh{\PP}$ denote the
(finite) set algebra generated by the partition classes in $\PP$. We denote by
$\PP^k$ the partition of $J^k$ whose classes are the product sets
$J_{i_1}\times\dots\times J_{i_k}$.

\begin{definition}\label{DEF:EXHAUST-FAM}
Let $\RR$ be a countable family of Borel sets. Let $\sigma({\RR})$ denote the
sigma-algebra generated by $\RR$. We say that $\RR$ is {\it generating}, if
$\sigma({\RR})=\BB$. We say that $\RR$ is {\it exhausting} with respect to a
measure $\pi$ on $(J,\BB)$, if for every $A\in\BB$ there is a set
$B\in\sigma({\RR})$ such that $\pi(A\triangle B)=0$. Clearly every generating
family is exhausting.

A {\it partition sequence} is a sequence $(\PP_i)_{i=1}^\infty$ of finite
measurable nondegenerate partitions of $(J,\BB,\pi)$ such that $\PP_{i+1}$ is a
refinement of $\PP_i$. We associate with every partition sequence the set
families $\RR=\bigcup_i\PP_i$ and $\wh\RR=\bigcup_i \wh\PP_i$. We say that a
partition sequence is generating [exhausting], if the family $\RR$ is
generating [exhausting].
\end{definition}

It is easy to see that every set in $\wh\RR$ is a finite union of disjoint
members of $\RR$. The family $\wh\RR$ is closed under finite union, finite
intersection, and complementation, so it is a set algebra.

We note that there is not much difference between talking about exhausting or
generating partition sequences: every exhausting partition sequence can be
transformed in a generating one by changing partitions on a $\pi$-null-set (see
Appendix \ref{APP:C}).

\section{Subgraph densities: known cases}\label{SSEC:SUB-DENSE}

We recall a couple of special classes of Markov spaces where subgraph densities
have been introduced and studied.

\subsection{Graphons}

Subgraph densities (or, to be more exact, homomorphism densities) can be
defined for bounded graphons. In fact, all densities are still finite if we
extend our attention to unbounded symmetric functions $W:~[0,1]^2\to\R_+$ in
$L^\omega$, see \cite{KKLSz}. If the degrees of the graphs mapped into the
graphon are bounded by $p$, then subgraph densities can actually be defined for
all of $L^p$-graphons \cite{BCCZ}. Subgraph densities can also be defined in
graphings, but this seems to be rather different from the dense case. It is
possible that this notion cannot be extended to all Markov spaces; but we will
be able to do so for Markov spaces which are sufficiently rich.

For the question to make sense in more general situations, we modify the
normalization of subgraph densities. Recall that for a graphon
$W:~J^2\to[0,1])$, the density of a graph $G=(V,E)$ in $W$ is defined by the
integral
\begin{equation}\label{EQ:T-DEF}
t(G,W)=\pi^V(W^G)=\intl_{J^V} W^G(x)\,d\pi^V(x),
\end{equation}
where
\begin{equation}\label{EQ:WG-DEF}
W^G(x) = \prod_{ij\in E(G)} W(x_i,x_j) \qquad (x\in J^V).
\end{equation}
If $W=W_H$ is the graphon associated with a graph $H$, then
\[
t(G,W_H)=t(G,H)=\frac{\hom(G,H)}{|V(H)|^{|V(G)|}}
\]
is the homomorphism density of $G\to H$. In this paper we use the normalization
\begin{equation}\label{EQ:TGH}
t^*(G,W)=\frac{t(G,W)}{t(K_2,W)^{|E(G)|}}.
\end{equation}

Note that the right hand side of \eqref{EQ:TGH} is invariant under scaling the
function $W$. If
\[
t(K_2,W)=\intl_{J^2} W\,d\pi^2 =1,
\]
(in particular, if $W$ is $1$-regular) we have $t^*(G,W)=t(G,W)$ for every $G$.

It will be very useful to consider the measure $W^G\cdot\pi^V$  with density
function $W^G$ on $J^V$. This measure has nice properties, for example, it is
Markovian. We call this the {\it density measure} of $G$ in $W$. This
construction will be particularly useful when we generalize the above formulas
to the case when $W$ is not bounded. Then the density \eqref{EQ:T-DEF} may be
infinite, but we still obtain a sigma-finite measure $W^G\cdot\pi^V$ on maps $V\to
J$. See Section \ref{SEC:UNBOUNDED} for a detailed discussion of this
generalization.

For a bigraph $G=(S,T,E)$ and a bigraphon $\Mb=(I,\AA,J,\BB,\pi_I,\pi_J,W)$, there
is a natural version of the subgraph density:
\begin{equation}\label{EQ:T1-DEF}
t(G,W)=\intl_{I^S}\intl_{J^T}\prod_{ij\in E} W(x_i,y_j)d\pi_J^T(y) d\pi_I^S(x).
\end{equation}
Clearly $t(G,W)=t(G^*,W^*)$.

\subsection{Orthogonality spaces}\label{SSEC:ORTH}

Consider the Borel sets in the $(d-1)$-dimensional unit sphere $S^{d-1}$, and
let $\eta$ be the uniform measure on orthogonal pairs of vectors in $S^{d-1}$.
This class of Markov spaces was studied in detail in \cite{KLSz2}. Maps $V\to
S^{d-1}$ that map edges onto orthogonal pairs are called {\it
ortho-homomorphisms}.

\begin{example}\label{EXA:KAB-ORT}
The case of complete bipartite graphs will be important. Consider the complete
bigraph $K_{a,b}=(U,W,U\times W)$, where $|U|=a$, $|W|=b$, and $a+b=d+1$. The
first relevant example is mapping the 4-cycle into $S^2$. Let $x\in
(S^{d-1})^V$ be an ortho-homomorphism. Since the image of $U$ spans a subspace
that is orthogonal to the subspace spanned by the image of $W$, one of the
color classes must be mapped onto linearly dependent vectors. If this
degenerate color class is $U$, then $W$ can be mapped freely into $x(U)^\perp$.
So every homomorphism is degenerate, and if $a,b\ge2$, then there are two
possible degenerations. This means that there is no ``natural'' or
``canonical'' way of defining a measure on ortho-homomorphisms. It is also easy
to observe that the trouble is caused by the fact that making $d$ random single steps
each starting from a given point of $S^{d-1}$, we obtain $d$ linearly dependent points, so the
joint distribution of these $d$ points is singular.

To motivate some of our later arguments, let us try to construct an
ortho-homomorphism of the 4-cycle into $S^3$ by mapping the nodes one-by-one.
The first three nodes can be mapped in an arbitrary order (taking care of the
orthogonality of images of edges). Almost surely the neighbors of the fourth
node will be neither equal nor antipodal, and so this node must be mapped
either on the image of its non-neighbor, or on its antipodal. Leaving instead one of
its neighbors for last, the other pair of non-neighbors will be parallel, so we
obtain a totally different distribution.
\end{example}

It was shown in \cite{KLSz2} that a canonical ``nice'' Markovian sigma-finite
measure on the ortho-homomorphisms into $S^{d-1}$ can be defined for every
graph $G$ not containing $K_{a,b}$ with $a+b=d+1$. Furthermore, the density of
$G$ in $\eta_d$ can also be defined (it may be infinite). The construction
followed the same lines as our treatment in Section \ref{SEC:RAND-MAP} below,
providing explicit formulas in this special case.

\section{Trees}\label{SEC:TREES}

The case of mapping trees into Markov spaces is easy, but it will be a very
useful starting point for the more general case. For a tree $F$, we denote by
$L(F)$ the set of its leaves and by $M(F)$ the set of its interior nodes. In
the case of a tree denoted by $F$, we will set $L=L(F)$ and $M=M(F)$. So
$F\setminus L= F[M]$ is the subtree induced by the internal nodes of $F$.

Let $S_n$ and $P_n$ denote the star and the path with $n$ edges, respectively.
Unless stated otherwise, we label $V(S_k)=\{0,1,\dots,k\}$ with $0$ in the
center. The tree consisting of a single edge $uv$ can be viewed either as a
path $P_1$, or as a star $S_1$. We distinguish them by letting $P_1$ have two
leaves, so $L(P_1)=\{u,v\}$ and $M(P_1)=\emptyset$, and designating one of the
nodes of $S_1$ (say $u$) as its center, and the other one as its leaf, so that
$L(S_1)=\{v\}$ and $M(S_1)=\{u\}$. It will be convenient to consider the tree
$S_0$ with a single node $u$, where we have $L(S_0)=\emptyset$ and
$M(S_0)=\{u\}$.

Let $F$ be a tree and $uv\in E(F)$. The subtree $F_1$ of $F$ induced by $u$ and
all nodes separated from $u$ by the edge $uv$ is called a {\it branch of $F$
attached at $u$}. We denote by $F\setminus F_1$ the subtree obtained by
deleting from $F$ the nodes in $V(F_1)\setminus\{u\}$.

\subsection{Random mappings of trees}\label{SEC:RANDMAP}

Our first step is to show that a random mapping of a tree into a Markov space
can be defined in a robust (and, as we shall see, useful) way. This simple
construction is well-known (branching Markov chains etc.), but we need some
special properties of it.

\begin{definition}\label{DEFI:TREEMAP}
Let $F=(V,E)$ be tree, and $(J,\BB,\eta)$, a Markov space. We define a {\it
random homomorphism of $F$ into $\eta$} as a random map \lljav{$\hb:~V(F)\to
J$}, recursively as follows. If $|V(F)|=1$, then we define $\hb$ as a random
point from $\pi$. If $|V(F)|>1$, then let $u$ be a leaf of $F$, incident with a
single edge $uv$. \lljav{The random map $\hb':~V(F\setminus u)\to J$ is already
constructed. We let $\hb|_{V\setminus u}=\hb'$, and we define $\hb(u)$ by
making a Markov step from the point $\hb'(v)$.} We denote the distribution of
$\hb$ by $\eta^F$. In formula, for $W\in\BB^{V\setminus u}$ and $A\in\BB$,
\begin{equation}\label{EQ:TAU-DEF}
\eta^F(W\times A)= \intl_W P_{x_u}(A) \,d\eta^{F'}(x).
\end{equation}
\end{definition}

We can also describe this construction slightly differently. Let
$(v_1,\dots,v_n)$ be a search order of $V(F)$, i.e. an ordering for which every
node $v_i$ different from the ``root'' $v_1$ is adjacent to exactly one earlier
node $v_{i'}$ ($1\le i'<i$). We select $\hb(v_1)$ from $\pi$, and for
$i=2,\dots,n$ we generate $\hb(v_i)$ by making a Markov step from
$\hb(v_{i'})$. We call this the {\it sequential construction} of the random
map.

\begin{lemma}\label{LEM:TREE-WELLDEF}
The recursive definition \eqref{EQ:TAU-DEF} gives a distribution $\eta^F$ that
is independent of the leaf chosen. Equivalently, if constructed sequentially,
it is independent of the search order chosen.
\end{lemma}

\begin{proof}
We proceed by induction on the number of vertices in $F$. If $|V(F)|=1$, then
clearly $\eta^F=\pi$, and if $F=P_1$, then we have
\[
\eta^{P_1}(A_1\times A_2)=\eta^{S_1}(A_1\times A_2)= \intl_{A_1}
P_{x_1}(A_2)\,d\pi(x_1) = \eta(A_1\times A_2),
\]
which remains the same when the indices are interchanged by symmetry. Now
suppose that $|V(F)|>2$, and let $u,w$ be two leaves, with neighbors $v$ and
$z$, respectively ($v=z$ is possible). Then $u$ and $w$ are not adjacent, and
so $F''=F\setminus u\setminus w$ is a tree. We have
\begin{equation}\label{EQ:TAU-DEF2}
\intl_{\prod(A_i:~i\in V\setminus u)}
P_{x_v}(A_u)\,d\eta^{F'}(x) = \intl_{\prod(A_i:~i\in V\setminus \{u,w\})}
P_{x_v}(A_u)P_{x_z}(A_w) \,d\eta^{F''}(x).
\end{equation}
We get the same if the roles of $u$ and $w$ are interchanged.
\end{proof}

\subsection{Marginals and conditioning on trees}\label{SSEC:MARGINALS}

We need some properties and associated constructions for the measure $\eta^F$,
where $F$ is a tree. The marginal $(\eta^F)^U$ on a set $U\subseteq V$ is
particularly simple when $U=V(F_1)$ for a subtree $F_1$, since then we can
start a search order of $F$ with a search order of $F_1$, which implies that
\begin{equation}\label{EQ:TREE-MARGIN0}
(\eta^F)^{V(F_1)} = \eta^{F_1}.
\end{equation}
Another simple but useful fact about node sets $U$ of subtrees is that we can
condition on any map $x\in J^U$, since a random extension of it can be
constructed in a well-defined way.

\lljav{In the case when $U=L$ is the set of leaves of $F$, we will denote the
marginal $(\eta^F)^L$ by $\sigma^F$.}

We also need conditioning on maps $z\in J^U$, where $U\subseteq V$ is a general
subset. This is not straightforward, since the measure of a singleton $z$
according to the marginal $\alpha:=(\eta^F)^U$ is typically zero. However, we
can use disintegration: Using the marginal $\alpha=(\eta^F)^U$, Proposition
\ref{PROP:DISINT} implies that there is a measurable family
$\Theta=\Theta^{F,U}=(\theta_z^F:~z\in J^U)$ of distributions on $\BB^{V\setminus U}$
such that $\eta^F=\alpha[\Theta]$, or explicitly
\begin{equation}\label{EQ:THETA-DEF}
\intl_A \theta^F_z(B)\,d\alpha(z) = \eta^F(A\times B)
\end{equation}
for all $A\in J^{V\setminus U}$ and $B\in J^U$. It will be convenient to define
$\Theta^{S_0,\emptyset}$ (recall that $S_0$ is the tree with a single node, no leaves) by
$\theta_\emptyset=\pi$ for the empty sequence $\emptyset$.

We can (informally) think of $\theta_z$ as the distribution of a random copy of
$F$, conditional on the set $U$ being mapped by $z$. Note, however, that
$\theta_z$ is determined only up to an $\alpha$-nullset of mappings $z$. This
fact (and that $\theta_z$ is only implicitly defined) make this construction
useless without some smoothness condition on $\eta$.

It is easy to extend the definition of $\eta^F$ to forests $F$, by taking the
product measure over the connected components. This way we have a measure
$\eta^{F[S]}$ for every $S\subseteq V(F)$. This family of measures, however,
does not have the decreasing property: for example, the marginal of $\eta^F$ on
the set $L$ of leaves is not necessarily absolutely continuous with respect to
$\pi^L$. In the next subsection we introduce properties of the Markov chain that
fixes this (and will play a crucial role for more general graphs as well.)

\subsection{Looseness}\label{SSEC:LOOSE}

We start with one of our main definitions.

\begin{definition}\label{DEF:LOOSE}
We say that the tree $F$ is {\it loose} in the Markov space $\Mb=(J,\BB,\eta)$,
if $\sigma^F$ is absolutely continuous with respect to $\pi^{L(F)}$. In this
case we can define the Radon--Nikodym derivative
\begin{equation}\label{EQ:SF-DEF}
s^F(z)=\frac{d\sigma^F}{d\pi^L}(z)
\end{equation}
\lljav{(determined for $\pi^L$-almost all $z\in J^L$).}
\end{definition}

For the tree $F=S_0$ with a single node $u$, we define $s^F(z_0)=1$ for the
empty sequence $z_0$. The edge $F=S_1$ (with one endpoint in $L$) is loose in
every Markov space, since $\sigma^F=\pi$, and so $s^F(z)\equiv 1$. The edge
$F=P_1$ (with both endpoints in $L$) is loose in $\eta$ if and only if $\eta$
is induced by some (possibly unbounded) graphon $W$; we have then
$s^{P_1}(x_1,x_2)=W(x_1,x_2)$. If $\eta$ is induced by a graphon, then every
tree is loose in $\eta$ (cf.~Section \ref{SEC:UNBOUNDED}).

If $s^F$ exists, then
\begin{equation}\label{EQ:S-INT}
\intl_{J^L} s^F(z)\,d\pi^L(z) = \sigma^F(J^L)=\eta^F(J^V)=1,
\end{equation}
and hence $s^F(z)$ is finite for $\pi^L$-almost all $z$.

If $F$ is loose in $\Mb$, then we can disintegrate $\eta^F$ with respect to
$\pi^L$, to get a measurable family $\Psi^F=(\psi^F_z:~z\in J^L)$ of measures
on $\BB^{M(F)}$ such that
\begin{equation}\label{EQ:PSI-DEF}
\pi^L[\Psi^F]=\eta^F.
\end{equation}
It is easy to see that these measures relate to those obtained by
disintegrating with respect to $\sigma^F$ by the equation
\begin{equation}\label{EQ:PSI-THETA}
\psi^F_z= s^F(z)\theta^F_z.
\end{equation}
We note that the measures $\psi^F_z$ are finite for almost all $z$, but they
are not probability measures in general. In fact,
\begin{equation}\label{EQ:PS-J-SK}
\psi^F_z(J^{M(F)}) = s^F(z)\theta^F_z(J^{M(F)})=s^F(z).
\end{equation}
The measures $\psi^F_z$ are not necessarily absolutely continuous with respect
to $\pi^M$ or $\eta^{F\setminus L}$, but we can state the following simple
lemma:

\begin{lemma}\label{LEM:PSI-NULL}
If $F$ is loose in $\eta$, then for every set $B\in\BB^M$ with
$\eta^{F\setminus L}(B)=0$, we have $\psi^F_z(B)=0$ for $\pi^L$-almost all
$z\in J^L$.
\end{lemma}

\begin{proof}
Indeed, $\eta^{F\setminus L}(B)=0$ implies that $\eta^F(A\times B) = 0$ for
every $A\in\BB^L$ (just start a search order of $F$ with $M(F)$). In particular
\[
\intl_{J^L}\psi^F_z(B)\,d\pi^L(z) = \eta^F(J^L\times B)=0,
\]
thus $\psi^F_z(B)=0$ for $\pi^L$-almost all
$z\in J^L$.
\end{proof}

The property of looseness is {\it not} inherited by subtrees; in fact, for the
two most important special trees, monotonicity goes in different directions. It
is easy to see that if the star $S_k$ ($k\ge 2$) is loose in $\eta$, then so is
$S_j$ for $j<k$. On the other hand, if a path $P_k$ ($k\ge 1$) is loose in
$\eta$, then so is $P_j$ for $j>k$.

\begin{theorem}\label{THM:LOOSENESS}
Let $(I,\AA,\eta)$ be a Markov space, let $F$ be a tree, let $F_1$ be a branch
of $F$, and let $F_2$ be obtained from $F$ by removing this branch. If both
trees $F_1$ and $F_2$ are loose in $\eta$, then so is $F$.
\end{theorem}

\begin{proof}
Let $F_1$ be attached at $u$, and let $e$ be the edge of $F_1$ incident with
$u$. Let $L_i=L\cap V(F_i)$, then $L(F_1)=L_1\cup\{u\}$, and $L(F_2)$ is either
$L_2$ or $L_2\cup\{u\}$. Let $F'=F\setminus e$, then $F'$ is a forest with two
components $F_1'=F_1\setminus u$ and $F_2$. Let $\tau=\pi\times
(\sigma^{F_1'})^{L_1}$, which is a distribution on $J^{L(F_1)}$, where the
first factor corresponds to $u$.

Let $\lambda_x$ denote the marginal of $\eta^{F_2}$ on $L_2$ conditioned on
$u\mapsto x$, and let $\Lambda=(\lambda_x:~x\in J)$. A random map from
$\lambda_x$ can be generated by using a search order of $F_2$ starting with
$u$. We can also denote $\lambda_x$ by $\lambda_z$ for $z\in J^{L(F_1)}$,
simply ignoring the coordinates other than $z_u$. Then
\[
\sigma^F = \sigma^{F_1}[\Lambda] \qquad\text{and}\qquad (\sigma^{F'})^L = \tau[\Lambda].
\]
By Lemma \ref{LEM:MARGINALS}, we have $\sigma^{F_1}\ll\tau$, and hence by Lemma
\ref{LEM:0-EQUI},
\begin{align}\label{EQ:LFI}
\sigma^F=\sigma^{F_1}[\Lambda]\ll\tau[\Lambda]=(\sigma^{F'})^L.
\end{align}
Clearly $\eta^{F'}=\eta^{F_1'}\times \eta^{F_2}$. Using that
$(\eta^{F_1'})^{L_1}=(\eta^{F_1})^{L_1}$, we have
\begin{align*}
(\sigma^{F'})^L&=(\eta^{F'})^L=(\eta^{F_1'}\times \eta^{F_2})^L = (\eta^{F_1'})^{L_1}\times (\eta^{F_2})^{L_2}
= (\eta^{F_1})^{L_1}\times (\eta^{F_2})^{L_2}\\
&= (\sigma^{F_1})^{L_1}\times (\sigma^{F_2})^{L_2}.
\end{align*}
By hypothesis, $\sigma^{F_i}\ll\pi^{L(F_i)}$ and hence
$(\sigma^{F_i})^{L_i}\ll\pi^{L_i}$. This implies that
$
(\sigma^{F'})^L\ll\pi^L,
$
and combined with \eqref{EQ:LFI}, we are done.
\end{proof}

The notion of $k$-looseness defined in the Introduction is the special case of
looseness of the tree $F=S_k$ (the star with $k$ leaves). We have
$\sigma^{S_k}=\sigma_k$; note that $\sigma_1=\pi$. For every $k$, $\sigma_k$ is
a measure on $k$-tuples of points of $J$ (ordered, but $\sigma_k$ is invariant
under permuting the nodes). If $\eta$ is $k$-loose, then we can define the
function
\begin{equation}\label{EQ:SK-DEF}
s_k(x_1,\dots,x_k) = s_k^\eta(x_1,\dots,x_k) = \frac{d\sigma_k}{d\pi^k}(x_1,\dots,x_k).
\end{equation}
Also recall that $\eta$ is {\it $(k,p)$-loose}, if the function $s_k$ is not
only in $L^1(\pi^k)$ (which follows by the definition) but in $L^p(\pi^k)$.

With this notion, we have the following corollary to Theorem
\ref{THM:LOOSENESS}.

\daku{
\begin{corollary}\label{COR:LOOSENESS}
For any $k\geq 2$ and tree $F$, if the maximum degree satisfies $2\leq \Delta(F)\leq k$, then $F$ is loose in every $k$-loose Markov space.
\end{corollary}
\begin{proof} The proof is by induction on the size of $F$.
As previously mentioned, $P_\ell$ is $2$-loose for all $\ell\geq2$. Also, note that any tree $F$ with maximum degree between 2 and $k$ is either a star (and thus loose by definition), a path of length $\geq 2$, or we can split off a branch $F_1$ such that both it and the remainder $F_2=F\setminus F_1$ have at least 3 vertices, in which case we are done by induction and Theorem \ref{THM:LOOSENESS}.
\end{proof}}
\medskip

\noindent{\bf Looseness in bi-Markov spaces.} We don't define looseness of a general
tree for bi-Markov spaces, we define $k$-looseness only. Let $\Mb=(I,J,\AA,\BB,\eta)$
be a bi-Markov space, and let \lljav{$(P_x:~x\in I)$} be the disintegration of $\eta$
defined in \eqref{EQ:MARKOV-DISINT2}. Select a point $u\in I$ from $\pi_I$, and
select $k$ independent points $x_1,\dots,x_k\in J$ from the distribution $P_u$.
We say that $\Mb$ is {\it $k$-loose \lljav{from $I$}}, if the joint
distribution $\sigma_{I,k}$ of $(x_1,\dots,x_k)$ is absolutely continuous with
respect to $\pi_J^k$. If this is the case, we can define the Radon--Nikodym
derivative
\begin{equation}\label{EQ:SK-DEF2}
s_{k,I}(x_1,\dots,x_k) = \frac{d\sigma_{I,k}}{d\pi_J^k}(x_1,\dots,x_k).
\end{equation}
We define $k$-looseness \lljav{from $J$} analogously. \lljav{We also define
{\it $(k,p)$-looseness from $I$} and {\it from $J$} analogously.} (Note that
$k$-looseness from $I$ does not imply $k$-looseness from $J$ in general.)

\section{Random mapping by tree decomposition}\label{SEC:RAND-MAP}

\subsection{Sequential tree decomposition}\label{SSEC:TREE-DEC}

A {\it sequential tree decomposition}\footnote{Not to be confused with ``tree
decomposition'' in the theory of graph minors.} of a graph $G$ is a sequence
$(F_1,\dots,F_m)$ of edge-disjoint trees, so that $G=\bigcup_i F_i$, and
$V(F_i)\cap V(F_1\cup\dots\cup F_{i-1})=Z_i$ is the set of leaves of $F_i$, for
$i=1,\dots,m$. In particular, $F_1$ is a singleton tree.

Let us list some special constructions of sequential tree decompositions.

\medskip

{\bf Edge decomposition.} A trivial construction is to start with singleton
trees for each node, and continue with attaching $P_1$'s to get the edges.

\medskip

{\bf Star decomposition.} A less trivial decomposition is the following. Let
$V=(v_1,\dots,v_n)$ be any ordering of $V$. For each node $i$, we construct the
star $F_i$ centered at $i$, with edges connecting $i$ to earlier nodes. This
decomposition will be particularly well-behaved if $G$ is bipartite, and the
ordering starts with singleton trees for the nodes in one bipartition class,
and continues with the full stars of the nodes in the other class.

\medskip

{\bf Subdivision decomposition.} Another useful example is obtained when $G$ is
a subdivision of a graph $H$ with any number of new nodes on each edge. The
sequence starts with the nodes in $U=V(H)$ as singleton trees, and then it
continues with the paths replacing the original edges (in any order).

\medskip

{\bf Double star decomposition.} Select an edge $ij$ in a bipartite graph $G$;
then $ij$ and the edges adjacent to it form a tree $F_{ij}$ (a double star).
The graph $G$ arises from $G'=G\setminus\{i,j\}$ by attaching the tree
$F_{ij}$. Continuing this with $G'$ instead of $G$, we get a sequential
tree-decomposition of $G$ (in backwards order).

\medskip

{\bf \lljav{Open} ear decomposition.} An ear decomposition \lljav{into paths}
is a further example (this will not concern us here).

\subsection{Sequential construction of measures}\label{SSEC:MAIN-CONSTRUCT}

Let $\Mb=(J,\BB,\eta)$ be a Markov space, and let $\FF=(F_1,\dots,F_m)$ be a
sequential tree decomposition of the graph $G$. We construct a random mapping
$x:~V\to J$ as follows. We select $x(F_1)$ from distribution $\pi$. Assuming
that the nodes in $F_1\cup\dots\cup F_{i-1}$ have been mapped $(i\le m)$, we
choose the image of $M(F_i)$ from the conditional distribution
$\theta^{F_i}_{x(L(F_i))}$ (defined in Section \ref{SSEC:MARGINALS}). The
distribution of this random map will be denoted by $\Rho_\FF$.

There are two major problems with this construction:

\medskip

--- First, the disintegration $\theta^{F_i}_z$ is determined only up to a set of
$\sigma^{F_i}$-measure zero, and there is no guarantee that the construction will not
produce an image of $L(F_i)$ that falls in a zero-set of $\sigma^{F_i}$ with
positive probability. As a trivial example, an edge decomposition has this
problem if $\eta$ is not absolutely continuous with respect to $\pi\times\pi$.

\medskip

--- Second, even if this does not happen, the distribution we construct may
depend on the specific decomposition into trees. This problem actually occurs
even in the case of the star decomposition of bipartite graphs; see Example
\ref{EXA:KAB-ORT}. One of our main results (Theorem \ref{THM:MAIN-KAB}) says
that in a sense these are the only bad examples.

\medskip

Both problems can be handled by making an appropriate looseness assumption
about $\eta$ and sparseness assumption about $G$. To describe these remedies,
suppose that a graph $G=(V,E)$ has a sequential tree decomposition
$\FF=(F_1,\dots,F_m)$ such that every tree $F_i$ is loose in $\Mb$. Set
$L_i=L(F_i)$ and $M_i=M(F_i)$. Define the functions $s^{F_i}$ by
\eqref{EQ:SF-DEF} and let
\begin{equation}\label{EQ:F-SI}
f_\FF(x) = \prod_{i=1}^m s^{F_i}(x_{L_i}).
\end{equation}
Let $\Rho=\Rho_\FF$ be the distribution on $\BB^V$ constructed above, and
define the measure
\begin{equation}\label{EQ:ETAG-DEF}
\eta_\FF = f_\FF\cdot\Rho_\FF.
\end{equation}
It is clear from this definition that $\eta_\FF$ is sigma-finite.

It will be useful to express this definition in a recursive way. The sequence
$\FF'=(F_1,\dots,F_{m-1})$ is a sequential tree decomposition of the graph
$G'=(V',E')=F_1\cup\dots\cup F_{m-1}$. We use the measurable family
$\Psi^{F_m}=(\psi_z^{F_m}:~z\in J^{L_m})$ defined in \eqref{EQ:PSI-DEF}. With
some abuse of notation, sometimes it is useful to consider $\Psi_m$ as indexed
by vectors $z\in J^{V'}$ (nodes in $V'\setminus L_m$ considered as dummies).
Then by definition
\begin{equation}\label{EQ:PIK-PSIK}
\pi^{L_m}[\Psi^{F_m}]=\eta^{F_m},
\end{equation}
and it is easy to check that
\begin{equation}\label{EQ:RECURS0}
\eta_\FF= \eta_{\FF'}[\Psi^{F_m}].
\end{equation}
We can use \eqref{EQ:PIK-PSIK} and \eqref{EQ:RECURS0} as a recursive definition
of $\eta_\FF$. We also define the ``density of $G$ in $\eta$'' as
\begin{equation}\label{EQ:TGETA}
t_\FF(G,\eta) = \eta_\FF(J^V)=\intl_{J^V} \prod_{i=1}^k s^{F_i}(x_{L_i})\,d\pi_J(x).
\end{equation}
Let us note that \eqref{EQ:RECURS0} implies that
\begin{equation}\label{EQ:FF-ABS-CONT}
(\eta_\FF)^{V'}\ll \eta_{\FF'}.
\end{equation}

To address the first problem described above, let us note the following. Assume
that $\eta_{\FF'}$ is already given. Note that the measures $\psi^{F_m}_z$ are
determined by \eqref{EQ:PIK-PSIK} up to a set of indices $z\in J^{L_m}$ of
$\pi^{L_m}$-measure $0$. If a set $Z\subset J^{L_m}$ satisfies $\pi^{L_m}(Z)=0$
but $(\eta_{\FF'})(Z\times J^{V'\setminus L_m})=(\eta_{\FF'})^{L_m}(Z)>0$, then
changing $\Psi^{F_m}$ for these indices $z\in Z$ will change the right hand
side of \eqref{EQ:RECURS0}, and we are in trouble. So for the recursive
construction to work, we need that $(\eta_{\FF'})^{L_m}\ll\pi^{L_m}$.

\begin{definition}\label{DEF:SMOOTH}
Let us say that the sequential tree decomposition $\FF=(F_1,\dots,F_m)$ is {\it
smooth} in $\Mb$, if $(\eta_{(F_1,\dots,F_{i-1})})^{L_i}\ll\pi^{L_i}$ for
$i=1,\dots,m$.
\end{definition}

We are going to show that star-decompositions are smooth in many triangle-free
graphs, and all tree decompositions are smooth in graphons.

Our main special case will be star-decompositions. Let $G$ be a graph with
maximum degree at most $k$. Let $\FF_p=(F_1,\dots, F_n)$ be a sequential star
decomposition of $G$, determined by an ordering $p=(v_1,\dots,v_n)$ of the
nodes, where $v_i$ is the center of $F_i$. We set $\eta_p=\eta_{\FF_p}$.

\subsection{Unbounded graphons}\label{SEC:UNBOUNDED}

Our first application of the general scheme described above is the case of
Markov spaces $\Mb=(J,\mathcal{B},\eta)$ with the property that $\eta$ is
absolutely continuous with respect to $\pi\times\pi$. It is convenient to
represent such Markov spaces by the Radom-Nikodym derivative
$W=d\eta/d(\pi\times\pi)$, which is a non-negative, symmetric measurable
function $W:~J\times  J\to\mathbb{R}$ with the property that $\int_x
W(x,y)d\pi=1$ holds for every $y\in J$. In particular we have that the $L^1$
norm of $W$ is $1$. We call measurable functions with this property {\it
$1$-regular graphons}. Note that every $1$-regular graphon $W$ uniquely
determines a Markov space $\Mb_W=(J,\mathcal{B},\eta_W)$ where
\begin{equation}\label{EQ:ETA-W-DEF}
\eta_W=W\cdot\pi^2
\end{equation}
In the rest of this section we are going to omit the subscript $W$ wherever no
confusion can arise.

The $1$-regularity of the graphon implies that the transition probabilities for
this Markov space are given by
\begin{equation}\label{EQ:W-KERNEL}
P_x(A) = \intl_A W(x,y)\,d\pi(y).
\end{equation}

\begin{lemma}\label{LEM:UNBOUND-LOOSE}
Every tree is loose in $\eta$. In other words, the measure $\eta$ is $k$-loose
for every natural number $k$.
\end{lemma}

\begin{proof}
The identity $\eta^{F} = W^F\cdot\pi^{V(F)}$ is easily checked for trees
$F=(V,E)$, using \eqref{EQ:W-KERNEL}. This implies that $\eta^F\ll\pi^V$, and
hence
\[
\sigma^F=(\eta^F)^{L(F)}\ll(\pi^V)^{L(F)}=\pi^{L(F)}.
\]
Thus $F$ is loose in $\eta$, proving the lemma.
\end{proof}

A convenient special property of such Markov spaces comes from the fact that
the function $W$ can be directly used to produce homomorphism measures for
every finite graph $G$:

\begin{theorem}\label{THM:MU-ETA}
Let $W$ be a $1$-regular graphon, and let $\FF=(F_1,\dots,F_m)$ be a sequential
tree-decomposition of a graph $G=(V,E)$. Then $\FF$ is smooth in $\eta$, and
\begin{equation}\label{EQ:ETAG-DEF-W}
\eta_\FF = W^G\cdot\pi^V.
\end{equation}
\end{theorem}

In particular, it follows that $\eta_\FF$ is independent of the decomposition
and $\eta^G = W^G\cdot\pi^V$ is well-defined.

\begin{proof}
We express the measures in the construction of $\eta_\FF$ as integrals of $W$.
First, let $F=(V,E)$ be a tree. It is easy to see that, by the definition of
$\eta^F$ and by \eqref{EQ:W-KERNEL}, that
\begin{equation}\label{EQ:ETA-W-F}
\eta^F = W^F\cdot\pi^V.
\end{equation}
This implies that for $A\in\BB^L$,
\begin{equation}\label{EQ:SIGMA-W-F}
\sigma^F(A) = \eta^F(A\times J^M) = \intl_{A\times J^M}W^F\,d\pi^V
\end{equation}
and for $B\in\BB^M$ and $z\in J^L$,
\begin{equation}\label{EQ:PSI-W-F}
\psi_z^F(B)=\intl_{B} W^F(z,y)\,d\pi^M(y).
\end{equation}

Now let $\FF=(F_1,\dots,F_m)$ be a sequential tree-decomposition of a graph
$G=(V,E)$. We are going to prove by induction on $m$ that \lljav{this
decomposition satisfies \eqref{EQ:ETAG-DEF-W}. This will imply that the
decomposition is smooth.}

Let $\FF'=(F_1,\dots,F_{m-1})$ and $G'=(V',E')=F_1\cup\dots\cup F_{m-1}$. To
prove that $G$ satisfies \eqref{EQ:ETAG-DEF-W}, we use the recurrence
\eqref{EQ:RECURS0}, along with \eqref{EQ:ETAG-DEF-W} for $G'$ and
\eqref{EQ:PSI-W-F}. Let $A\in \BB^{V'}$ and $B\in\BB^{M_m}$, then
\begin{align*}
\eta_\FF(A\times B) =& \intl_A \left(\intl_{B} W^{F_m}(z_{L_m},y)\,d\pi^{M_m}(y)\,W^{G'}(z)\right)\,d\pi^{V'}(z,w)\\
=& \intl_{A\times B} W^G\,d\pi^V
\end{align*}
(here $w$ is the vector of dummy variables in $J^{V'\setminus L_m}$). This
proves \eqref{EQ:ETAG-DEF-W}.

\lljav{To prove that $\FF$ is smooth, it suffices to note that
\eqref{EQ:ETAG-DEF-W} implies that $\eta_{\FF'} \ll \pi^{V'}$, and hence
$(\eta_{\FF'})^{L(F_m)}\ll\pi^{L(F_m)}$. This holds for all other prefixes of
$\FF$ by the same argument.}
\end{proof}

A direct application of Theorem \ref{THM:MU-ETA} implies that the formalism of
this paper is a consistent extension of earlier results in bounded graphon
theory.

\begin{corollary}\label{COR:MU-ETA}
Let $W$ be a $1$-regular graphon and let $G=(V,E)$ be a finite graph. Then
\[
t(G,\eta)=t(G,W)=\eta^G(J^V)=\intl_{J^V}W^G~d\pi^V.
\]
\end{corollary}

Note, however, that this value may be infinite (see Example
\ref{EXA:NON-COMPACT}).

If $W$ has stronger properties, then we can strengthen the $k$-looseness
property of graphons to $(k,p)$-looseness.

\begin{lemma}
Let $k,p$ be natural numbers. For $x\in J$ let $f(x)$ denote the $L^p$-norm of
the function $y\mapsto W(x,y)$. If $f\in L^k(\pi)$, that is,
$
\intl_J(
\intl_J W(x,y)^p\,d\pi(y)
)^{k/p}
\,d\pi(x)<\infty
$,
then $\eta$ is
$(k,p)$-loose.
\end{lemma}

\begin{proof}
For $x\in J$ let $H_x:J^k\to\mathbb{R}$ denote the function defined by
$H_x(y_1,y_2,\dots,y_k):=W(x,y_1)W(x,y_2)\dots W(x,y_k)$. It is easy to see
that the $L^p$-norm of $H_x$ on $(J^k,\pi^k)$ is equal to $f(x)^k$. Thus by the
convexity of $L^p$-norm we have that the $L^p$-norm of $H:=\int_J H_x~d\pi(x)$
is at most $\int_x f(x)^k~d\pi$ and so the condition of the lemma implies that
the $L^p$-norm of $H$ is finite. This implies that $\eta$ is $(k,p)$-loose.
\end{proof}

This lemma has two immediate corollaries.

\begin{corollary}\label{COR:PP-LOOSE}
Let $W$ be a $1$-regular $L^p$-graphon for some natural number $p>1$. Then
$\eta$ is $(p,p)$-loose.
\end{corollary}

\begin{corollary}
Let $W$ be a $1$-regular graphon such that for some $c\in\mathbb{R}$ we have
that $\int_y W(x,y)^p~d\pi\leq c$ holds for every $x$. Then $W$ is
$(k,p)$-loose for every natural number $k$.
\end{corollary}

\ignore{ ***\lljav{ The following example (for which we are grateful to the
anonymous referee of our paper) shows that the conclusion of Corollary
\ref{COR:PP-LOOSE} cannot be strengthened: There is an $L^p$-graphon that is
not $(p+1,p)$-loose.}

\begin{example}\label{EXA:PP0}
\lljav{Fix $p\ge 1$, and let $(J_n:~ n=1,2,\dots)$ be a measurable partition of
$J$ with $\pi(J_n) = 2^{-n}$. Define
\[
W(x,y)=
  \begin{cases}
    2^{2n/p}n^{-2/p}, & \text{if $x,y\in J_n$}, \\
    0, & \text{otherwise}.
  \end{cases}
\]
It is easy to see that $\|W\|_p^p = \sum_n n^{-2} < \infty$. On the other hand,
if $k > p$, then we have
\[
t(K_{k,p},W) = \sum_{n=1}^\infty 2^{-pn-kn}\Big(2^{2n/p}n^{-2/p}\Big)^{kp} = \sum_{n=1}^\infty 2^{-pn+kn}n^{-2k},
\]
which is infinite since $k > p$.}
\end{example}
}

Our next two examples show that $1$-regular graphons can be rather wild objects
in terms of spectral properties and subgraph densities.

\medskip

\begin{example}\label{EXA:NON-COMPACT}
Let $W:~[0,1]^2\to\mathbb{R}^2$ be the function whose value is defined by
$W(x,y)=2^k$ whenever $x,y\in I_k=(2^{-k},2^{-(k-1)})$, and $0$ otherwise. We
define $\pi$ as the Lebesgue measure on $[0,1]$. It is clear that for every
natural number $k\geq 1$, the \lljav{indicator} function of $I_k$ is an
eigenvector of $W$ with eigenvalue $1$. Thus the eigenspace of $W$ with
eigenvalue $1$ is infinite dimensional. This implies that $W$ is not a compact
operator. Direct calculation shows that if a connected graph $G=(V,E)$ is not a
tree, then $t(G,W)=\infty$. More precisely, since $G$ is connected, $W^G=0$
unless all nodes are mapped into the same interval $I_k$. Hence
\[
t(G,W)=\sum_{k=1}^\infty\intl_{I_k^V} W^G(x)\,dx =\sum_{k=1}^\infty
2^{k(|E(G)|-|V(G)|)},
\]
which is equal to $1$ if $|E(G)|=|V(G)|-1$ (i.e., $G$ is a tree) and $\infty$
otherwise.
\end{example}

\medskip
Note that in the preceding example the graphon is an $L^1$ function, whereas
any graphon in $L^2$ would at least have finite cycle densities (as the max
degree is 2). Changing the parameters we can obtain a family of examples in
$L^p$ ($1\leq p<2$) that get arbitrarily close to being Hilbert-Schmidt
kernels, yet still have infinite densities for all non-trees.

\begin{example}\label{EXA:PP}
Let $\varepsilon\in[0,1)$. Let $(a_n)_{n\in\Nbb}$ be a sequence of positive
reals such that $\sum_{n\in\Nbb} a_n=1$ and $\sum_{n\in\Nbb}
a_n^{1-\varepsilon}<\infty$. Let $(J_n:~ n=1,2,\dots)$ be a measurable
partition of $J$ with $\pi(J_n) = a_n$. Define the unbounded kernel $W$ by
\[
W(x,y)=
  \begin{cases}
    1/a_n, & \text{if $x,y\in J_n$}, \\
    0, & \text{otherwise}.
  \end{cases}
\]
It is easy to see that $W$ is 1-regular, and
\[
\|W\|_{1+\varepsilon}^{1+\varepsilon}=\sum_{n\in\mb{N}}
a_n^2/a_n^{1+\varepsilon}=\sum_{n\in\Nbb} a_n^{1-\varepsilon}<\infty.
\]
On the other hand, the density of any connected graph $G$ in $W$ can be
obtained as the sum of the densities in each of the diagonal blocks, i.e.,
\[
t(G,W)=\sum_{n\in\Nbb} a_n^{|V(G)|}/a_n^{|E(G)|}.
\]
The sum is equal to 1 for trees, and infinite for all other graphs $G$.
\end{example}

Although the above construction with an infinite number of independent blocks
seems to suggest that the key to infinite densities is non-compactness, this is
not quite the case. Indeed, the next example shows that compactness of the
operator defined by $W$ is by itself not enough to guarantee that subgraph
densities behave any better.

\begin{example}\label{EXA:GRAPHON}
Let $f:~I=[-1,1]\to\R$ be a function with the following properties: $f\ge 0$;
$f(-x)=f(x)$ for all $x\in I$; $\int_I f(x)\,dx=1$; $f$ is convex and monotone
decreasing for $x>0$. Define a graphon by
\[
W(x,y)=f(x-y)\qquad(x,y\in I),
\]
where $f$ is extended periodically modulo $2$. Clearly $W$ is symmetric and
$1$-regular.  As a kernel operator, $W$ is positive semidefinite and compact as
$L^2(\mu)\to L^2(\mu)$. In the special case
\[
f(x)=\frac1{|x|(2-\ln(|x|))^2},
\]
no operator power of $W$ has finite trace. So $t(C_n,W)=\infty$ for all $n$
(see Appendix \ref{APP:D} for details).
\end{example}

\subsection{Triangle-free graphs}\label{SSEC:TRIANGLE-FREE}

In this section we concentrate on sequential star decompositions. We need a
simple combinatorial lemma.

\begin{lemma}\label{LEM:REORDER}
Let $F:~S_V\to X$, where $S_V$ is the set of permutations of the node set of a
triangle-free graph $G=(V,E)$, and $X$ is any set. Assume that $F$ has the
following two invariance properties for every permutation $p=(v_1,\dots,v_n)$:

\smallskip

{\rm(i)} If $v_kv_{k+1}\notin E$, then interchanging $v_k$ and $v_{k+1}$ in $p$
does not change $F(p)$;

\smallskip

{\rm(ii)} If every node in $A=\{v_1,\dots,v_a\}$ is connected to every node in
$B=\{v_{a+1},\dots,v_{a+b}\}$, then interchanging the blocks $A$ and $B$ in $p$
does not change $F(p)$.

\smallskip

Then $F$ is constant.
\end{lemma}

\lljav{Note that in (ii), $A$ and $B$ must be independent node sets as $G$ is
triangle-free, so $A\cup B$ induces a complete bipartite graph.}

\begin{proof}
We use induction on $n$. For a fixed $v\in V$, the function
$F_v(x_1,\dots,x_{n-1})=F(x_1,\dots,x_{n-1},v)$ satisfies the conditions in the
lemma, so by the induction hypothesis, it is constant. This means that there is
a function $f:~V\to X$ such that $F(x_1,\dots,x_n) = f(x_n)$.

Let $u,v\in V$ be nonadjacent. Considering any permutation
$(v_1,\dots,v_{n-2},u,v)$, we see that
\[
\lljav{f(v)=F(v_1,\dots,v_{n-2},u,v)=F(v_1,\dots,v_{n-2},v,u)=f(u).}
\]

Now let $u,v\in V$ be adjacent. If there is a path in the complement
$\overline{G}$ connecting $u$ and $v$, then applying the previous observation
repeatedly we get that $f(u)=f(v)$. If there is no such path, then there is a
partition $V=A\cup B$ so that $u\in A$, $v\in B$, and every edge between $A$
and $B$ is present. Since $G$ is triangle-free, it follows that $G$ is a
complete bipartite graph. Let $A=\{u_1,\dots,u_a=u\}$ and
$B=\{v_1,\dots,v_b=v\}$, then
\[
f(v)=F(u_1,\dots,u_a,v_1,\dots,v_b) = F(v_1,\dots,v_b,u_1,\dots,u_a) =f(u).
\]
So $f$ is constant, and then so is $F$.
\end{proof}

Let $G$ be a triangle-free graph with maximum degree $k$, and let
$\Mb=(J,\BB,\eta)$. For a sequential star decomposition $\FF=(F_1,\dots, F_n)$
of $G$, determined by an ordering $p=(v_1,\dots,v_n)$ of the nodes, let
$\eta_p=\eta_{\FF}$ denote the measure on $J^V$ defined by \eqref{EQ:RECURS0}.
In general, $\eta_p$ will depend on the ordering $p$ and also on the measure
families $\Psi^{F_i}=(\psi^{F_i}_z:~z\in J^{L(F_i)})$, which are determined
only up to a set of indices $z\in J^{L(F_i)}$ of $\pi^{L(F_i)}$-measure zero.

Now we are ready to prove Theorem \ref{THM:MAIN-KAB}.
\begin{thm*}
Let $G=(V,E)$ be a triangle-free graph, and let $\Mb=(J,\BB,\eta)$ be a Markov
space such that every complete bipartite subgraph $K_{a,b}$ of $G$ is
well-measured in $\Mb$. Then $G$ is well-measured in $\Mb$.
\end{thm*}
\begin{proof*}{Theorem \ref{THM:MAIN-KAB}}
We prove the theorem by induction on $n$. The condition is clearly inherited by
induced subgraphs of $G$, so we may assume that every proper induced subgraph
of $G$ is well-measured in $\Mb$.

First we prove that for every ordering $p=(v_1,\dots,v_n)$ of the nodes of $G$,
the measure $\eta_p$ does not depend on the choice of the measure families
$\Psi^{F_i}$. We know by induction that $G'=G\setminus v_n$ is well-measured in
$\Mb$, so $\eta^{G'}$ does not depend on these choices. Consider the measures
$\left(\psi_z: z\in J^{N(v_n)}\right)$. Two different choices of the measures $\psi_z$ can
differ on a set $Z_0\in\BB^{L_n}$ of maps $z$ with $\pi^{L_n}(Z_0)=0$. By the
definition of well-measurability, we have
$(\eta_{\FF'})^{L_n}\ll\eta^{G[L_n]}$, where $\FF'=(F_1,\ldots,F_{n-1})$. Since $G$ is triangle-free, $L_n$ is an
independent set of nodes, so $\eta^{G[L_n]}=\pi^{L_n}$ and hence $(\eta_{\FF'})^{L_n}\ll\pi^{L_n}$. Thus
$\eta_\FF$ is uniquely determined by \eqref{EQ:RECURS0}.

\smallskip

To prove that for any two orderings $p$ and $q$ of the nodes of $G$, we have
$\eta_p= \eta_q$, we use Lemma \ref{LEM:REORDER}. For a permutation $p\in S_V$,
let $F(p)=\eta_p$. Condition (i) is trivial, and condition (ii) is also easy:
if the first $a+b$ nodes induce a complete bipartite subgraph, then the
sequential construction up \lljav{to} the first $a+b$ nodes results in the same
measure by the hypothesis of the theorem, and the completion of the
construction does not depend on the order of these $a+b$ nodes.

\smallskip

So the sequential construction provides a measure $\eta^G$ independent of the
ordering. Recall that the measures $\eta^{G[S]}$, where $S\subset V$, are also
given by induction. This family of measures is trivially normalized and, as
remarked before, sigma-finite. The decreasing property is easy: we can start
the sequential construction by any given set $S$, and the
$(\eta^G)^S\ll\eta^{G[S]}$ follows by repeated application of
\eqref{EQ:FF-ABS-CONT}. To prove the Markov property, let $V=U\cup T$ such that
there is no edge between $U\setminus S$ and $T\setminus S$ where $S=U\cap T$.
Consider an ordering $p$ of $V$ starting with $S$. Recall \eqref{EQ:RECURS0},
describing the recursive definition of $\eta^G$. It follows that the
disintegration $(\mu_{U,V,z}:~z\in J^U)$ of $\eta^{G[V]}$ by $\eta^{G[U]}$ has
the property that $\mu_{U,V,z}$ depends only on $z|_S$, and we have a similar
property with $U$ and $T$ interchanged. Hence for every $x\in J^S$,
\[
\mu_{S,V,x}=\mu_{S,U,x}\times\mu_{S,T,x},
\]
proving that the measure family $(\eta^{G[S]}:~S\subseteq V)$ is Markovian.
\end{proof*}

\begin{remark}\label{REM:KAB}
Note that the proof above only uses that the sequential construction of
$\eta^{K_{a,b}}$ gives the same measure if we start with one bipartition class
or the other. It is not hard to see, along the lines of the proof of Lemma
\ref{LEM:REORDER}, that this is equivalent with $K_{a,b}$ being well-measured.
We will return to the question of which complete bipartite graphs are
well-measured in a Markov space in Section \ref{SSEC:PART-APPROX}.
\end{remark}

\begin{remark}\label{REM:A-K}
As we have mentioned in the Introduction, if $G$ has girth at least $5$, then
the only complete bipartite subgraphs of $G$ are stars, and the condition means
that $\Mb$ is $k$-loose, where $k$ is the maximum degree of $G$. Also note that
the condition on $G$ is inherited by all subgraphs of $G$.

The condition that all degrees are bounded by $k$ could be relaxed: the
construction would work for all graphs that are $k$-degenerate (i.e.,
repeatedly deleting nodes with degree at most $k$, the whole graph can be
eliminated). For $k=1$ (which imposes no condition on the Markov space), we get
the measure $\eta^F$ for all trees. (Recall, however, that this does not imply
that trees are well-measured: the decreasing property fails.) The extension of
the considerations in Section \ref{SEC:TREES} is left for further study.

An important example of this more general setup would be the following. There
are Markov spaces $\eta$ whose $k$-th power $\eta^k$ (as introduced along with
the adjacency operator) is induced by a bounded graphon $W$, but they
themselves are not. For example, the orthogonality space in any dimension has
this property. If $\eta$ has this property and $G'$ is a $k$-subdivision of a
graph $G$ then $G'$ is $2$-degenerate. Working with subdivision decompositions
of $G'$, we can construct $\eta^{G'}$, which will be finite. So we see that
$\eta^G$ exists and $t(G,\eta)<\infty$ holds for such Markov spaces and for a
large set of graphs $G$ with no degree bound.
\end{remark}

\bjav{\begin{remark}\label{REM:XKAB} Note that the bi-Markov space analogue of Theorem
\ref{THM:MAIN-KAB} also holds and the proof is essentially the same mutatis
mutandis.
\end{remark}}

\subsection{Bigraphs and bi-Markov spaces}

The sequential construction of $\eta_G$ takes a particularly simple form when
$G$ is bipartite. Let $G=(U,W,E)$ be a bigraph and $\Mb=(I,J,\AA,\BB,\eta)$, a
bi-Markov space \lljav{$k$-loose from $J$}. Our considerations apply, in particular, to
$k$-loose Markov spaces.

To define $\eta^G$, we can use an ordering of the nodes that starts with $U$.
Then the nodes in $U$ will be mapped onto independent random points of $I$ from
distribution $\pi_I$. Furthermore, the points of $W$ will be mapped conditionally independently
given the image of $U$. For this to make sense, it suffices to
require that all nodes in $W$ have degree at most $k$.

For every finite sequence $(x_1,\dots,x_d)$ of points $d\ge 1$ of $I$, we have
a measurable family of measures $\Psi_d=(\psi_x:~x\in I^d)$ on $\BB$ defined by
the disintegration
\begin{equation}\label{EQ:PSIX}
\eta^{S_d} = \pi_I^d[\Psi_d].
\end{equation}
We can think of $\psi_x$ informally as the measure on the common neighbors of
$x=(x_1,\dots,x_d)$.

For a node $w\in W$, let $F_w$ denote the star formed by the edges incident
with $w$. We define the product measure
and the corresponding measurable family by
\[
\wh\psi_x = \prod_{w\in W}\psi_{x_{N(w)}} \qquad\text{and}\qquad \wh\Psi=(\wh\psi_x:~x\in I^U).
\]
Then we define
\begin{equation}\label{EQ:PI-PSI}
\eta^G=\pi_I^U[\wh\Psi],
\end{equation}
or explicitly,
\begin{equation}\label{EQ:ETA-X}
\eta^G(A\times B) = \intl_A \wh\psi_{x_{N(w)}}(B)\,d\pi_I^U(x)\qquad(A\in \AA^U, B\in \BB^W).
\end{equation}
This measure $\eta^G$ is well-defined, since the measures $\psi_{x_{N(w)}}$ can be
changed on a $\pi_I$-nullset only. Note that the definition is more general than
our construction in Section \ref{SSEC:TRIANGLE-FREE}, since no assumption is
necessary for the degrees of nodes in $U$.

Formula \eqref{EQ:PI-PSI} makes sense when the disintegration $\Psi^{F_w}$ in \eqref{EQ:PSI-DEF} can
be defined. By Proposition \ref{PROP:DISINT}, this happens if
$\sigma^{F_w}\ll\pi_I^{N(w)}$, that is, $\eta$ is $k$-loose \lljav{from $J$}, and
all degrees of $G$ in $W$ are bounded by $k$, for some $k\ge 1$. If this holds,
then the density function $s_w=s_{w,J}=s^{F_w}$ is well-defined in
\eqref{EQ:SK-DEF}, and
\[
\psi_{x_{N(w)}}(J^W) = s_w(x_{N(w)}) = s_{\deg(w)}(x_{N(w)}).
\]
In particular, we obtain the following formula for the density of the bigraph
$G$ in $\Mb$:
\begin{align}\label{EQ:TG-ETA}
t(G,\eta) = \eta^G(I^U\times J^W) =\intl_{I^U} \prod_{w\in W} s_{\deg(w)}(x_{N(w)})\,d\pi_I^U(x).
\end{align}

Formula \eqref{EQ:PI-PSI} does not define $\eta^G$ if $G_0$ is not a bigraph
but only a bipartite graph (so its bipartition classes are not fixed). It may
even happen that only one of these measures is well-defined (\lljav{for
example,} if the maximum degree in $U$ is larger than $k$).

But assume that both of them are well-defined; is then $\eta^G=(\eta^*)^{G^*}$
or at least \bjav{$t(G^*,\eta^*)=t(G,\eta)$}? By Theorem \ref{THM:MU-ETA}, this is the
case when $\eta$ is defined by a graphon, and by \bjav{the bi-Markov space analogue of
Theorem \ref{THM:MAIN-KAB} (see Remark \ref{REM:XKAB})}, this also
holds true if $G$ contains no quadrilaterals. Further sufficient conditions
will be given below.  We'll state such a theorem (Theorem \ref{THM:X-DENSITY})
later. On the other hand, Example \ref{EXA:KAB-ORT} shows that some condition
along these lines is necessary.

One of the difficulties caused by this asymmetry can be partly remedied as
follows.

\begin{lemma}\label{LEM:W-MARGIN}
\bjav{Let $\Mb=(I,J,\AA,\BB,\eta)$ be a bi-Markov space $k$-loose from $J$ and let
$G=(U,W,E)$ be a bigraph such that every vertex of $W$ has degree at most $k$
in $G$. Then} the marginal of $\eta^G$ on $U$ is absolutely continuous with
respect to $\pi_I^U$. \bjav{Furthermore} the marginal of $\eta^G$ on any node of
$W$ is absolutely continuous with respect to $\pi_J$.
\end{lemma}

\begin{proof}
It is clear by \eqref{EQ:ETA-X} that if $\pi_I^U(A)=0$, then $\eta^G(A\times
J^W)=0$, which implies the first assertion. Similar claim does not follow for a
general $B\in\BB^W$ from $\pi_J^W(B)=0$ (see Example \ref{EXA:KAB-ORT});
however, if $B$ is a box $B=\prod_{w\in W} B_w$ $(B_w\in\BB)$, then by
\eqref{EQ:ETA-X} we have
\begin{equation}\label{EQ:ETA-X2}
\eta^G(A\times B) = \intl_A \prod_{w\in W} \psi_{x_{N(w)}}(B_w)\,d\pi_I^U(x).
\end{equation}
\bjav{The bi-Markov space analogue of Lemma \ref{LEM:PSI-NULL} and the fact that
$\Mb$ is $k$-loose from $J$} implies that if $\pi_J(B_w)=0$ for some $w\in W$, then
$\psi_{x_{N(w)}}(B_w)=0$ for $\pi_I$-almost all $x_{N(w)}$, and so $\eta^G(I^U\times
J^{W\setminus \{w\}}\times B_w)=0$.
\end{proof}

\section{Approximation by graphons}\label{SEC:APPROX}

\subsection{Convergence of graphons to Markov spaces}

Suppose that a sequence of graphons $W_n$ ``tends to'' a Markov space
$(J,\BB,\eta)$ in some sense. Does this imply that for graphs $G$ satisfying
suitable conditions, we have $t(G,W_n)\to t(G,\eta)$? We prove two results
along these lines. The first was used (implicitly) in \cite{KLSz2}; the second
will be used later in this paper.

Let $(J,\BB,\eta)$ be a $k$-loose Markov space, and let $W_n$ ($n=1,2,\dots$)
be a sequence of $1$-regular graphons on $(J,\BB,\pi)$. We say that $\eta$ is
the {\it $k$-limit} of the sequence $(W_n)$, if
\[
s^{W_n}_k(x) \to s^\eta_k(x)
\]
for $\pi^k$-almost all $x\in J^k$, and there is a constant $C = C(\eta,k)$
independent of $x$ and an integer $n_0\geq1$, such that
\[
s^{W_n}_k(x) \le C s^\eta_k(x)
\]
for every $n\geq n_0$ and $\pi^k$-almost all $x\in J^k$.

We say that a $(k,p)$-loose Markov space $(J,\BB,\eta)$ is the {\it
$(k,p)$-limit} of the sequence $(W_1,W_2,\dots)$ of graphons, if $s^{W_n}_k \to
s^\eta_k$ in $L^p(\pi^k)$ (note that there then exists a constant $C>0$ such that
$\|s^{W_n}_k\|_p\le C$ for every $n$).

We need an important analytic tool that allows us to bound products of
functions in multivariate $L^p$ spaces, namely a special case of the general,
multivariate version of HÃ¶lder's inequality, called Finner's theorem
(\cite[Theorem 2.1]{Fin}). For the sake of self-containedness, we state this
special case, and its main corollary that will be relevant to us.

\begin{theorem}\label{THM:FINNER}
Let $(J,\BB,\pi)$ be a probability space, and $p$, $n$ and $m$ positive
integers. Let $f_k:~J^n\to\mathbb{R}$ ($1\leq k\leq m$) be measurable
functions, where $f_k$ depends only a set $M_k$ of variables. Assume that every
variable $x_i$ ($1\leq i\leq m$) is contained in at most $p$ sets $M_k$. Then
\begin{equation}\label{eqn:Finner}
\intl_{J^n}\prod_{k=1}^m f_k \,d\pi^n\leq \prod_{k=1}^m \|f_k\|_p.
\end{equation}
\end{theorem}

By a standard telescopic decomposition argument, this yields the following
convergence result.

\begin{corollary}\label{cor:Finner}
Let $(J,\BB,\pi)$ be a probability space, and $p$, $n$ and $m$ positive
integers. Let $f_k, f_{k,\ell}:~J^{M_k}\to\mathbb{R}$ ($1\leq k\leq m$,
$\ell\in\mathbb{N}$) be measurable functions, where $f_k$ and $f_{k,\ell}$
depend only on a set $M_k$ of variables. Assume that every variable $x_i$
($1\leq i\leq m$) is contained in at most $p$ sets $M_k$. Also assume that $f_k\in L^p(J,\BB,\pi)$ and
\[
\lim_{\ell\to\infty}\|f_{k,\ell}-f_k\|_p= 0
\]
holds for all $1\leq k\leq m$ and $\ell\in\mathbb{N}$. Then
\begin{equation}\label{eqn:Finner2}
\lim_{\ell\to\infty}\intl_{J^n}\prod_{k=1}^m f_{k,\ell}\,d\pi^n=\intl_{J^n}\prod_{k=1}^m f_{k}\,d\pi^n.
\end{equation}
\end{corollary}

\begin{theorem}\label{THM:APP}
Let $(J,\BB,\eta)$ be a $k$-loose Markov space, let $W_n$ $(n=1,2,\dots)$ be a
sequence of $1$-regular graphons on $(J,\BB)$ such that $\eta$ is the $k$-limit
of $(W_n)$. Let $G=(U,W,E)$ be a bigraph in which $\deg(w)\le k$ for all $w\in
W$, and assume that $t(G,\eta)<\infty$. Then
\[
t(G,\eta)=\lim_{n\to \infty} t(G,W_n).
\]
\end{theorem}

The right hand side is invariant under interchanging the bipartition classes of
$G$. Thus if, in addition to the conditions of Theorem \ref{THM:APP},
$\deg(u)\le k$ holds for all $u\in U$, then $t(G,\eta)=t(G^*,\eta)$.

\begin{proof}
We have
\[
t(G,\eta)= \intl_{J^U}\prod_{w\in W}s^\eta_{\deg(w)}(x_{N(w)})\,d\pi^U(x).
\]
and
\[
t(G,W_n)= \intl_{J^U}\prod_{w\in W}s^{W_n}_{\deg(w)}(x_{N(w)})\,d\pi^U(x).
\]
Here
\[
\prod_{w\in W}s^{W_n}_{\deg(w)}(x_{N(w)}) \to \prod_{w\in W}s^\eta_{\deg(w)}(x_{N(w)})
\]
almost everywhere, and
\[
\prod_{w\in W}s^{W_n}_{\deg(w)}(x_{N(w)}) \le C^{|W|} \prod_{w\in W}s^\eta_{\deg(w)}(x_{N(w)}).
\]
Since the function on the right is integrable by the condition that
$t(G,\eta)<\infty$, the theorem follows by Lebesgue's Dominated Convergence
Theorem.
\end{proof}

We state an analogous theorem under the stronger assumption of
$(k,p)$-looseness. Recall that a Markov space $(J,\BB,\eta)$ is $(k,p)$-loose $(k,p\in\Nbb)$, if it is $k$-loose and
$\|s^\eta_k\|_p$ is finite. Quite surprisingly it will turn out that
$(k,p)$-looseness of Markov spaces is a symmetric notion: a Markov space
$(J,\BB,\eta)$ is $(k,p)$-loose if and only if it is $(p,k)$-loose. \lljav{(For
bi-Markov spaces this symmetry property no longer holds, however.)}

\begin{theorem}\label{THM:APP2}
Let $(J,\BB,\eta)$ be a $(k,p)$-loose Markov space, let $W_n$ $(n=1,2,\dots)$
be a sequence of $1$-regular graphons on $(J,\BB)$ such that $\eta$ is the
$(k,p)$-limit of $(W_n)$. Let $G=(U,W,E)$ be a bigraph, and assume that
$\deg(u)\le p$ for $u\in U$ and $\deg(w)\le k$ for $w\in W$. Then
\[
t(G,\eta)=\lim_{n\to \infty} t(G,W_n)<\infty.
\]
\end{theorem}

\begin{proof}
Let $W=\{v_1,\dots,v_m\}$.
Then
\begin{align*}
t(G,\eta) = \intl_{J^U}\prod_{v\in W}s^\eta_{\deg(v)}(x_{N(v)}) \,d\pi^U(x)
\end{align*}
and
\begin{align*}
t(G,W_n) = \intl_{J^U} \prod_{v\in W}s^{W_n}_{\deg(v)}(x_{N(v)})\,d\pi^U(x).
\end{align*}
Each variable $x_u$ ($u\in U$) occurs in at most $p$ factors, and so Corollary
\ref{cor:Finner} implies the theorem.
\end{proof}

\subsection{Projection onto stepfunctions}

A natural approximation of a Markov space $(J,\BB,\eta)$ is the following. Let
$\PP=\{P_1,\dots,P_n\}$ be a finite, measurable, non-degenerate partition. For
a function $f\in L^1(\pi)$, we define
\[
f_\PP = \frac1{\pi(P_i)} \intl_{P_i} f\,d\pi \qquad(x\in P_i).
\]
We generalize this to every $k$-variable function $h:~J^k\to\R$ by
\[
h_\PP = h_{\PP^k}.
\]
In particular, for a graphon $W$ we have
\[
W_\PP(x,y)=\frac1{\pi(P_i)\pi(P_j)}\intl_{P_i\times P_j} W(x,y)\,d\pi(x)\,d\pi(y)
\qquad(x\in P_i,\ y\in P_j).
\]
The linear operator $\Ebb_\PP:~f\mapsto f_\PP$ (called a ``stepping operator''
in \cite{HomBook}) is a bounded linear operator $L^1(J,\AA,\pi)\to
L^\infty(J,\BB,\pi)$. If we consider it as an operator $L^2(J,\AA,\pi)\to
L^2(J,\BB,\pi)$, then it is self-adjoint and idempotent.

One property of the stepping operator that will be important for us is that it
is contractive with respect to most ``everyday'' norms \cite[Proposition
14.13]{HomBook}, in particular, with respect to all $L^p$-norms
($p\in[1,\infty]$):
\begin{equation}\label{EQ:STEP-CONTRACT}
\|f_\PP\|_p \le \|f\|_p
\end{equation}
for all $f\in L^p(J,\BB,\pi)$.

We can extend this construction to Markov spaces, where its image is a bounded
graphon $W=W_{\eta_\PP}$, defined by
\[
W_{\eta_\PP}(x,y)=\frac{\eta(P_i \times P_j)}{\pi(P_i)\pi(P_j)}\qquad(x\in P_i,\ y\in P_j).
\]
The edge measure associated with this graphon is
\[
\eta_\PP = \sum_{i,j=1}^k \frac{\eta(P_i \times P_j)}{\pi(P_i)\pi(P_j)}\,((\mathds{1}_{P_i}\pi)\times(\mathds{1}_{P_j}\pi)).
\]
Note that the marginals of $\eta_\PP$ are $\pi$, and so $W_{\eta_\PP}$ is
$1$-regular.

In terms of the adjacency operator $\Ab$ of the Markov space, the operator
$\Ab_\PP$ associated with $\eta_\PP$ can be expressed as the operator product
$\Ab_\PP=\Ebb_\PP \Ab \Ebb_\PP$.

We will also need the stepping operator for bi-Markov spaces. Let $(I,J,\AA,\BB,\eta)$
be a bi-Markov space, and let $\PP=\{P_1,\dots,P_k\}$ and $\QQ=\{Q_1,\dots,Q_m\}$ be
finite, measurable, nondegenerate partitions of $(I,\AA,\pi)$ and
$(J,\BB,\pi_J)$, respectively. We define the following measures on $\AA\times\BB$:
\[
(\Ebb_\PP\eta)(S\times T) = \sum_{i=1}^k \frac{\pi(S\cap P_i)}{\pi(P_i)}\,\eta(P_i \times T),
\]
and
\[
(\eta\Ebb_\QQ)(S\times T) = \sum_{j=1}^l \frac{\pi_J(T\cap Q_j)}{\pi_J(Q_j)}\,\eta(S\times Q_j).
\]
We can also partition both sigma-algebras, to obtain
\bjav{
\[
(\Ebb_\PP\eta\Ebb_\QQ)(S\times T) = \sum_{i=1}^k\sum_{j=1}^l
\frac{\pi(S\cap P_i)}{\pi(P_i)}\frac{\pi_J(T\cap Q_j)}{\pi_J(Q_j)}\eta(P_i\times Q_j).
\]
}

For a bi-Markov space, we also have a (non-self-adjoint) operator $\Ab$, and then the
measures $\Ebb_\PP\eta$, $\eta\Ebb_\QQ$ and $\Ebb_\PP\eta\Ebb_\QQ$ are
associated with the (non-self-adjoint) operators $\Ebb_\PP\Ab$, $\Ab\Ebb_\QQ$
and $\Ebb_\PP\Ab\Ebb_\QQ$, respectively. Clearly all three of these measures have the same
marginals $\pi_I$ and $\pi_J$ as $\eta$.

\begin{lemma}\label{LEM:SKEW-MARGIN}
For every bi-Markov space $(I,J,\AA,\BB,\eta)$ and finite, measurable, nondegenerate
partitions $\PP$ and $\QQ$ of $I$ and $J$, respectively, the measures
$\Ebb_\PP\eta$, $\eta\Ebb_\QQ$ and $\Ebb_\PP\eta\Ebb_\QQ$ are absolutely continuous with respect to
$\pi_I\times\pi_J$, with a bounded density function.
\end{lemma}

\begin{proof}
Checking this for $\Ebb_\PP\eta$, let $S\in \AA$ and $T\in\BB$. Then
\[
(\Ebb_\PP\eta)(S\times T) \le \sum_{i=1}^k \frac{\pi_I(S)}{\pi_I(P_i)}\,\eta(I\times T)
=\Big(\sum_{i=1}^k \frac{1}{\pi_I(P_i)}\Big) \pi_I(S)\pi_J(T),
\]
which implies that $\Ebb_\PP\eta$ is absolutely continuous with respect to
$\pi_I\times\pi_J$, and its density function is bounded by $\sum_i 1/\pi_I(P_i)$. The argument for $\eta\Ebb_\QQ$ is symmetric, and the result for $\Ebb_\PP\eta\Ebb_\QQ$ follows from the previous two, the fact that $\Ebb_\PP\eta\Ebb_\QQ=(\Ebb_\PP\eta)\Ebb_\QQ$ and the fact that the marginals of $\Ebb_\PP\eta$ are also $\pi_I$ and $\pi_J$.
\end{proof}

\subsection{Stepfunction approximation}

Let $W$ be a bounded graphon and let $(\PP_i)_{i=1}^\infty$ be an exhausting
partition sequence (see Subsection \ref{SEC:PARTITIONS}). The Martingale Convergence Theorem implies that
$W_{\PP_i}\to W$ almost everywhere on $J^2$, and hence $(W_{\PP_i})^G\to W^G$
almost everywhere on $J^V$ for every graph $G$. It is easy to check that the
sequence $W_{\PP_i}^G$ is uniformly integrable, and hence $W_{\PP_i}^G\to W^G$
in $L^1$, which implies that the corresponding measures also converge. In
particular,
\begin{equation}\label{EQ:TGWPP2TGW}
t(G,W_{\PP_i})\to t(G,W)
\end{equation}

How far does this fact extend beyond graphons? Under what conditions on $G$ and
$\eta$ does $\lim_{i_\to\infty} t(G,\eta_{\PP_i})$ exist for every exhausting
partition sequence $(\PP_i)_{i=1}^\infty$? Is the limit value independent of
the sequence of partitions?

Recall that we say that $\eta^G$ is partition approximable if $\eta_{\PP_i}^G\to\eta^G$ on boxes for every exhausting partition sequence.
Our goal in the next sections is to establish that
$\eta^G$ is partition approximable for reasonably large classes of graphs $G$
and Markov spaces $\eta$. To motivate this goal, let us state a simple
consequence about the normalized density $t^*$ (see Equation \eqref{EQ:TGH}).

\begin{prop}
\bjav{ Let $(J,\BB,\eta)$ be a Markov space. Then there is a sequence of simple graphs
$(H_i)_{i=1}^\infty$ such that
\[
\lim_{i\to\infty} t^*(G,H_i)=t(G,\eta)
\]
for every graph $G$ such that $\eta^G$
is partition approximable. }
\end{prop}

\begin{proof}
Let $(\PP_i)_{i=1}^\infty$ be an exhausting partition sequence. For every $i\ge
1$, there is an appropriate number $c_i>0$ such that $c_i\eta_{\PP_i}$ is a
graphon (with values in $[0,1]$), and so by dense graph limit theory, there is
a sequence of graphs $H_{i,1},H_{i,2},\dots$ such that $t(G,H_{i,j})\to
t(G,c_i\eta_{\PP_i})= c_i^{|E|} t(G,\eta_{\PP_i})$ for any $G$. In particular,
$t(K_2,H_{i,j})\to c_i t(K_2,\eta_{\PP_i})$, and hence $t^*(G,H_{i,j})\to
t(G,\eta_{\PP_i})$. Furthermore, $t(G,\eta_{\PP_i})\to t(G,\eta)$
$(i\to\infty)$ if $\eta^G$ is partition approximable. Since there are countably
many graphs $G$ to be considered, a standard diagonalization argument completes
the proof. \bjav{Note that two diagonalizations should happen: one to get rid of the partitions $\PP_i$ and one to make a single sequence for every $G$.}
\end{proof}

\subsection{Weakly norming graphs}

A graph $G$ is called {\it weakly norming} if
\[
\|W\|_G:=t(G,|W|)^{1/|E(G)|}
\]
is a norm on symmetric bounded measurable functions $W:~I^2\to\R$. This
property was introduced by Hatami \cite{HH}. It is easy to see that all weakly
norming graphs are bipartite; main examples are even cycles, hypercubes and
complete bipartite graphs.

Since the operator $W\mapsto W_\PP$ is contractive with respect to a large
class of norms, including all norms defined by graphs (see e.g.\ Proposition
14.13 in \cite{HomBook}), weakly norming graphs satisfy the inequality
\begin{equation}\label{EQ:PART-MON-DEF}
t(G,W_\PP)\le t(G,W)
\end{equation}
for every graphon $W$ and every finite, measurable, non-degenerate partition
$\PP$. This property is closely related to the well-known Sidorenko-Simonovits
conjecture, which says that $t^*(G,H)\ge 1$ for every bipartite graph $G$ and
every graph $H$. This is equivalent to saying that $t^*(G,W)\ge 1$ for every
bipartite graph $G$ and every graphon $W$. For the trivial partition
$\PP_0=\{J\}$ we have $t(G,W_{\PP_0})=t(K_2,W)^{|E(G)|}$, and hence every graph
$G$ satisfying \eqref{EQ:PART-MON-DEF} satisfies the Sidorenko conjecture.

Property \eqref{EQ:PART-MON-DEF} of a graph $G$, required for every graphon $W$
and every finite, measurable, non-degenerate partition $\PP$, was introduced in
\cite{KMPW}, and called the {\it step Sidorenko property}. It was proved in
\cite{DGHRR} that this property is equivalent to being weakly norming.

For us, however, the inequality \eqref{EQ:PART-MON-DEF} is relevant only for
$1$-regular graphons. Then it holds for more graphs besides weakly norming
ones, for example, for all trees. Therefore we name it the {\it weak step
Sidorenko property}. It is easy to see that only bipartite graphs can have this
property. As far as we can see, it might even hold for all bipartite graphs. If
the graph $G$ has the weak step Sidorenko property, then the convergence in
\eqref{EQ:TGWPP2TGW} is monotone.

\begin{remark}\label{RE:PART-DENSE}
These considerations motivate the following version of density, which we call
{\it partition-density}:
\begin{equation}\label{EQ:T-SUPDEF}
t_{\rm part}(G,\eta)=\sup_{\PP}t(G,\eta_\PP),
\end{equation}
where $(J,\BB,\eta)$ is a Markov space, and $\PP$ ranges over all finite,
measurable, non-degenerate partitions of $J$. Partition density may be
different from density even for ordinary graphs in place of $\eta$. For
example, if $H$ is bipartite and $G$ is not, and $H$ has at least one edge,
then for the trivial (indiscrete) partition $\PP$, we have
$\left(\eta_H\right)_\PP= c(\pi\times\pi)$, and so $t(G,H)=0$ but $t_{\rm
part}(G,\eta_H)>0$.

\daku{ On the other hand, the monotonicity from \eqref{EQ:PART-MON-DEF} and the
Martingale Convergence Theorem applied to $(W^G)_{\PP_i}$ along any exhaustive
partition sequence $(\PP_i)_{i=1}^\infty$ implies that $t_{\rm
part}(G,W)=t(G,W)$ for every weakly norming graph $G$ and every graphon $W$. }
It could be interesting to explore further properties of the partition-density.
\end{remark}

\begin{remark}\label{REM:X-SPACE}
The weakly norming property, the step Sidorenko property and its weak version
can be defined, mutatis mutandis, for bi-Markov spaces, and the above considerations
remain valid. In particular, even cycles, complete bigraphs and hypercubes
remain weakly norming, and hence have the step Sidorenko property.
\end{remark}

\subsection{Partition approximation of $(k,p)$-loose spaces}\label{SSEC:PART-APPROX}

While our main goal is to prove results about Markov spaces, we study
\lljav{$(k,p)$-looseness in} bi-Markov spaces first. We address the issues of
approximability by step functions. It turns out that for $(k,p)$-loose Markov
spaces and bi-Markov spaces, $\eta^G$ is partition approximable for a large class of
(bipartite and bi-) graphs.

We start with discussing the total measure of $\eta^G$. For a $(k,p)$-loose
Markov space, we can define the quantity
\begin{equation}\label{EQ:ETA-NORM-DEF}
\|\eta\|_{k,p}:=\|s^\eta_k\|_p^{1/k}.
\end{equation}
For a bi-Markov space $(k,p)$-loose from (say) $I$, we define similarly
\begin{equation}\label{EQ:ETA-NORM-DEF-X}
\|\eta\|_{I,k,p}:=\|s^\eta_{I,k}\|_p^{1/k}.
\end{equation}
If $p\geq 2$ and $\eta$ is not $(k,p)$-loose, then we define $\|\eta\|_{k,p}$
to be infinite. If $k=1$,
then $s_1^\eta=d\sigma_1/d\pi=1$, so $\|\eta\|_{k,p}=1$ by \eqref{EQ:ETA-NORM-DEF}.  \daku{When $p=1$, it may happen that $\eta$ is not $k$-loose and thus $s^\eta_k$ is not defined. 
However, the $L^1$ norm of a Radon--Nikodym derivative being the same as the total measure, we can extend the above definition to also encompass the non-$k$-loose cases and define $\|\eta\|_{k,1}:=\sigma_k(J^k)^{1/k}=1$ for \emph{any} $\eta$}.

Note that
\begin{equation}\label{EQ:ETA-NORM-KKP}
\|\eta\|_{k,p} = \|s^\eta_k\|_p^{1/k} =  t(K_{k,p},\eta)^{1/(kp)}
\end{equation}
by \eqref{EQ:TG-ETA}. We will show that for Markov spaces
$\|\eta\|_{k,p}=\|\eta\|_{p,k}$, or in other words,
$t(K_{k,p},\eta)=t(K_{p,k},\eta)$.

As cited above, Hatami \cite{HH} proved that
\begin{equation}\label{EQPK1}
\|W\|_{k,p}=t(K_{k,p},W)^{1/(kp)}
\end{equation}
is a norm on (not necessarily symmetric) bounded measurable functions
$W:~I\times J\to\R$. Clearly $t(K_{k,p},W)=t(K_{p,k},W^*)$ holds for every
bounded measurable function $W$, and so
\begin{equation}\label{EQPK2}
\|W\|_{k,p}=\|W^*\|_{p,k}.
\end{equation}
In particular, $\|W\|_{k,p}=\|W\|_{p,k}$ if $I=J$ and $W$ is symmetric. It is
easy to check that if $(J,\BB,\eta_W)$ is a Markov space defined by a 1-regular
graphon $W$, then
\[
\|\eta_W\|_{k,p} = \|W\|_{k,p}.
\]
\lljav{Formally the same equation holds for a bi-Markov space defined by a $1$-regular
bigraphon.}

Consider a bi-Markov space $(I,J,\AA,\BB,\eta)$. Let $\PP$ and $\QQ$ be finite,
measurable, non-degenerate partitions of $I$ and $J$, respectively. By Lemma
\ref{LEM:SKEW-MARGIN}, the measures $\Ebb_\PP\eta$ and $\eta^*\Ebb_\PP$ are
represented by bounded measurable functions $W_1$, $W_2$, where trivially
$W_1^*=W_2$. Hence \eqref{EQPK2} implies that
\begin{equation}\label{EQPK3}
\|\Ebb_\PP\eta\|_{k,p}=\|W_1\|_{k,p}=\|W_2\|_{p,k}=\|\eta^*\Ebb_\PP\|_{p,k}.
\end{equation}
Similarly we have $\|\eta\Ebb_\QQ\|_{k,p}=\|\Ebb_\QQ\eta^*\|_{p,k}$. For an
exhausting partition sequence, in the limit, we have more:

\begin{lemma}\label{tripeq}
Let $(I,J,\AA,\BB,\eta)$ be a bi-Markov space, and let $(\PP_i)_{i=1}^\infty$ and
$(\QQ_j)_{j=1}^\infty$ be exhausting partition sequences of $I$ and $J$,
respectively. Then
\[
\lim_{i\to\infty} \|\eta\Ebb_{\QQ_i}\|_{k,p}=\lim_{i\to\infty} \|\Ebb_{\PP_i}\eta\Ebb_{\QQ_i}\|_{k,p}
=\lim_{i\to\infty} \|\Ebb_{\PP_i}\eta\|_{k,p}=\|\eta\|_{k,p}.
\]
\end{lemma}

\begin{proof}
We start with the first equality. Since $K_{k,p}$ is weakly norming, it follows by the step Sidorenko property \eqref{EQ:PART-MON-DEF}
that both limits exist, and also that
$\|\eta\Ebb_{\QQ_i}\|_{k,p}\geq \|\Ebb_{\PP_i}\eta\Ebb_{\QQ_i}\|_{k,p}$. Hence
we obtain that
\[
\lim_{i\to\infty} \|\eta\Ebb_{\QQ_i}\|_{k,p}\ge \lim_{i\to\infty} \|\Ebb_{\PP_i}\eta\Ebb_{\QQ_i}\|_{k,p}.
\]

Let $j\in\mathbb{N}$ be an arbitrary fixed number. Since $W=\eta\Ebb_{\QQ_j}$
is a bounded measurable function, the uniformly bounded measurable functions
$\Ebb_{\PP_i}\eta\Ebb_{\QQ_j}=\Ebb_{\PP_i}W$ converge to $W$ in $L^1$ as
$i\to\infty$, and thus by (\ref{EQPK1}) we get
\[
\lim_{i\to\infty} \|\Ebb_{\PP_i}\eta\Ebb_{\QQ_j}\|_{k,p}=\|\eta\Ebb_{\QQ_j}\|_{k,p} .
\]
Again by the step Sidorenko property \eqref{EQ:PART-MON-DEF} we have that for $i>j$,
\[
\|\Ebb_{\PP_i}\eta\Ebb_{\QQ_i}\|_{k,p}\ge \|\Ebb_{\PP_i}\eta\Ebb_{\QQ_i}\Ebb_{\QQ_j}\|_{k,p}
= \|\Ebb_{\PP_i}\eta\Ebb_{\QQ_j}\|_{k,p}
\]
and so by taking limit on both sides,
\[
\lim_{i\to\infty} \|\Ebb_{\PP_i}\eta\Ebb_{\QQ_i}\|_{k,p}\ge \|\eta\Ebb_{\QQ_j}\|_{k,p}.
\]
This holds for every $j$, which proves the first equality. The second follows
by interchanging the coordinates.

Finally, we prove that
\[
\lim_{j\to\infty} \|\eta\Ebb_{\QQ_j}\|_{k,p}=\|\eta\|_{k,p}.
\]
If $p=1$ then the statement is trivial since all terms are $1$. Assume that
$p>1$. We have two cases. If $\eta$ is $k$-loose from $I$, then $s_{I,k}^\eta$ is in
$L^1(J^k,\pi_J^k)$, and so
\[
s_{I,k}^{\eta\Ebb_{\QQ_j}}=\mathbb{E}(s_{I,k}^\eta|\QQ_j^k).
\]
By Lemma \ref{appendix2}, we have that $(\QQ_j^k)_{i=1}^\infty$ is an
exhausting partition sequence for $\pi_J^k$ and so the (potentially infinite)
$L^p$-norm of $\mathbb{E}(s_{I,k}^\eta|\QQ_j^k)$ converges to the $L^p$-norm of
$s_{I,k}^\eta$ as $j\to\infty$.

Assume now that $\eta$ is not $k$-loose from $I$.
We have that $\sigma_{I,k}$ is not absolutely continuous with respect to $\pi_J^k$ and so
there is a measurable set $U\subset J^k$ such that $\pi_J^k(U)=0$ but
$c=\sigma_{I,k}(U)>0$. By Lemma \ref{appendix2}, for every $\epsilon>0$ and large
enough $j$, there is a set $U'$ that is the union of $\QQ_j^k$ partition sets
such that $\pi_J^k(U')\leq\epsilon$ and $\sigma_{I,k}(U')>c-\epsilon$. For such a
$j$,
\[
\intl_{U'} s_{I,k}^{\eta\Ebb_{\QQ_j}}\,d\pi_J^k=\intl_{U'} s_{I,k}^{\eta}\,d\pi_J^k=\sigma_{I,k}(U').
\]
H\"older's inequality implies that
\[
\intl_{U'} \Bigl(s_{I,k}^{\eta\Ebb_{\QQ_j}}\Bigr)^p~d\pi_J^k\geq
\frac{\left(\intl_{U'} s_{I,k}^{\eta\Ebb_{\QQ_j}}~d\pi_J^k\right)^p}{\pi_J^k(U')^{p-1}}
\geq\epsilon^{1-p}(c-\epsilon)^p.
\]
Applying this for every $\epsilon>0$ we obtain that
$\|s_{I,k}^{\eta\Ebb_{\QQ_j}}\|_p\to\infty$ as $j\to\infty$.
\end{proof}

From the previous lemma we obtain the next theorem.

\begin{theorem}\label{quadeq}
Let $(J,\BB,\eta)$ be a Markov space and $p,k\in\Nbb$. Then
\[
\|\eta\|_{p,k}=\|\eta\|_{k,p}=t_{\rm{part}}(K_{k,p},\eta)^{1/(pk)}.
\]
\end{theorem}

\daku{
\begin{proof}
Let $(\PP_i)_{i=1}^\infty$ be an arbitrary exhausting partition sequence.
To see
the first equality, observe that by
Lemma \ref{tripeq} and \eqref{EQPK3},
\[
\|\eta\|_{p,k}=\lim_{i\to\infty} \|\eta\Ebb_{\PP_i}\|_{p,k}
=\lim_{i\to\infty} \|\Ebb_{\PP_i}\eta\|_{k,p}=\|\eta\|_{k,p}.
\]
For the second equality, by
\eqref{EQPK1},
\[
\lim_{i\to\infty}
t(K_{k,p},\eta_{\PP_i})^{1/(pk)}=\lim_{i\to\infty} \|\eta_{\PP_i}\|_{k,p} =\|\eta\|_{k,p}.
\]
However, since $K_{k,p}$ has the step Sidorenko property \eqref{EQ:PART-MON-DEF}, this yields
\[\|\eta\|_{k,p}=\lim_{i\to\infty}
t(K_{k,p},\eta_{\PP_i})^{1/(pk)}=\sup_{i\in\Nbb}
t(K_{k,p},\eta_{\PP_i})^{1/(pk)}.\]
Since any partition can appear in an exhausting sequence, we obtain the desired equality.
\end{proof}
}

\begin{lemma}\label{convcrit}
Let $(J,\BB,\mu)$ be a probability space, and $p\geq1$. Let $(\PP_i)_{i=1}^\infty$ be an
exhausting partition system. Assume that a sequence of $L^p$ functions
$(f_i)_{i=1}^\infty$ and another $L^p$ function $f$ on $(J,\BB,\mu)$ satisfy
\begin{enumerate}
\item $\lim_{i\to\infty} \|\Ebb(f_i|\PP_j)-\Ebb(f|\PP_j)\|_1=0$ for every
    $j$
\item $\lim_{i\to\infty} \|f_i\|_p=\|f\|_p$.
\end{enumerate}
Then  $\lim_{i\to\infty}\|f_i-f\|_p=0$.
\end{lemma}

\begin{proof}
Let $\epsilon>0$. Then there is $\delta>0$ such that $\|1_Uf\|_p\leq\epsilon$
holds for every measurable set $U$ with $\mu(U)\leq\delta$. We can choose $j_0$
with the property that $\|f-\Ebb(f|\PP_j)\|_p\leq\epsilon$ holds for every
$j\geq j_0$. Then
\begin{equation}\label{kpl1}
\|1_U\Ebb(f|\PP_j)\|_p\leq \|1_Uf\|_p+\|1_U(f-\Ebb(f|\PP_j))\|_p\leq 2\epsilon
\end{equation}
hold for every $j\geq j_0$ and measurable set $U$ with $\mu(U)\leq\delta$. For
sufficiently big $i_0$ we can also guarantee that
$|\|f_i\|_p-\|f\|_p|\leq\epsilon$ holds for every $i\geq i_0$. For an arbitrary
$i\geq i_0$ we can choose $j\geq j_0$ such that both
$|f_i-\Ebb(f_i|\PP_j)|\leq\epsilon/2$ and
$|\Ebb(f_i|\PP_j)-\Ebb(f|\PP_j)|\leq\epsilon/2$ holds on a set $V$ of measure
at least $1-\delta$. It follows that $|f_i-\Ebb(f|\PP_j)|\leq\epsilon$ holds on
$V$. This implies that
\begin{equation}\label{kpl2}
\|1_Vf_i-1_V\Ebb(f|\PP_j)\|_p\leq\epsilon.
\end{equation}
Let $U$ be the complement of $V$. Using (\ref{kpl1}) we have that
\[
\|1_V\Ebb(f|\PP_j)\|_p\geq
\|\Ebb(f|\PP_j)\|_p-\|1_U\Ebb(f|\PP_j)\|_p\geq\|f\|_p-3\epsilon
\]
and thus by (\ref{kpl2})
\[
\|1_Vf_i\|_p\geq\|f\|_p-4\epsilon.
\]
Using the above inequalities we obtain
\begin{equation}\label{kpl3}
\|1_Uf_i\|_p^p=\|f_i\|_p^p-\|1_Vf_i\|_p^p\leq(\|f\|_p+\epsilon)^p-(\|f\|_p-4\epsilon)^p=:g(\epsilon).
\end{equation}
From (\ref{kpl2}), (\ref{kpl3}) and $f_i=1_Uf_i+1_Vf_i$ we get that
\[\|f_i-1_V\Ebb(f|\PP_j)\|_p\leq g(\epsilon)^{1/p}+\epsilon\]
By \eqref{kpl1}, we have
\[\|\Ebb(f|\PP_j)-1_V\Ebb(f|\PP_j)\|_p
=\|1_U\Ebb(f|\PP_j)\|_p\leq 2\epsilon\] and thus
\[\|f_i-\Ebb(f|\PP_j)\|_p\leq 3\epsilon+g(\epsilon)^{1/p}.\]
This implies
\[\|f_i-f\|_p\leq 4\epsilon+g(\epsilon)^{1/p}.\] Since $\lim_{\epsilon\to 0}g(\epsilon)=0$ the proof is complete.
\end{proof}

\begin{lemma}\label{pconv}
Consider a bi-Markov space $(I,J,\AA,\BB,\eta)$. Let $(\PP_i)_{i=1}^\infty$ and
$(\QQ_j)_{j=1}^\infty$ be exhausting partition sequences of $I$ and $J$,
respectively. Set $\eta_i=\Ebb_{\PP_i}\eta\Ebb_{\QQ_i}$. Then $s^{\eta_i}_k\to
s^{\eta}_k$ in $L^p$ as $i\to\infty$.
\end{lemma}

\begin{proof} \daku{We have
\[
\lim_{i\to\infty}\|s^{\eta_i}_k\|_p
=\lim_{i\to\infty}\|\eta_i\|_{p,k}^k=
\|\eta\|_{k,p}^k=\|s^\eta_k\|_p,
\]
where the second equality is from Lemma \ref{tripeq}}
and the remaining
equalities are just definitions. Now according to Lemma \ref{convcrit} it
suffices to prove that for every $j\in\mathbb{N}$ we have
\[
\lim_{i\to\infty}\Ebb(s^{\eta_i}_k|\QQ_j^k)=\Ebb(s^\eta_k|\QQ_j^k)
\]
in $L_1$.
To see this observe that
\[
\Ebb(s^{\eta_i}_k|\QQ_j^k)=s^{W_{i,j}}_k~~~{\rm and}~~~\Ebb(s^\eta_k|\QQ_j^k)=s^{W_j}_k,
\]
where
\[
W_{i,j}:=\Ebb_{\PP_i}\eta\Ebb_{\QQ_i}\Ebb_{\QQ_j}~~~{\rm
and}~~~W_j:=\eta\Ebb_{\QQ_j}.
\]
If $i\geq j$ then $\Ebb_{\QQ_i}\Ebb_{\QQ_j}=\Ebb_{\QQ_j}$ and so
$W_{i,j}=\Ebb_{\PP_i}W_j$. Since for fixed $j$ we have that $W_{i,j}$ is a
uniformly bounded sequence of measurable functions with $L_1$ limit $W_j$ the
integral form of $s^{W_{i,j}}_k$ and $s^{W_j}_k$ shows the required
convergence. More precisely, by abusing the notation, let us identify $W_{i,j}$
and $W_j$ with their representations by measurable functions. Then we have
\[
s^{W_{i,j}}_k(z_1,z_2,\dots,z_k)=\mathbb{E}_x S_{i,j}(x,z_1,z_2,\dots,z_k)
\]
and
\[
s^{W_j}_k(z_1,z_2,\dots,z_k)=\mathbb{E}_x S_j(x,z_1,z_2,\dots,z_k),
\]
where
\[
S_{i,j}(x,z_1,x_2,\dots,z_k):=W_{i,j}(x,z_1)W_{i,j}(x,z_2)\cdots W_{i,j}(x,z_k)
\]
and
\[
S_j(x,z_1,x_2,\dots,z_k):=W_j(x,z_1)W_j(x,z_2)\cdots W_j(x,z_k).
\]
Then
\begin{align*}
\|s^{W_{i,j}}_k-s^{W_j}_k\|_1&=\|\mathbb{E}_x(S_{i,j}-S_j)\|_1
\leq\|\mathbb{E}_x(|S_{i,j}-S_j|)\|_1\\
&=\|S_{i,j}-S_j\|_1\leq k\|W_{i,j}-W_j\|_1\|W_j\|_\infty^{k-1},
\end{align*}
where the last inequality follows by changing the terms in the product one by
one using the usual telescopic argument and the fact that
$\|W_{i,j}\|_\infty\leq\|W_j\|_\infty$. The fact that $W_{i,j}$ converges to
$W_j$ in $L_1$ completes the proof.
\end{proof}

Now we are ready to state and prove our main theorem in this section.

\begin{theorem}\label{THM:KP-BIP}
Let $(J,\BB,\eta)$ be \lljav{a} $(k,p)$-loose Markov space, and let $G=(U,W,E)$
be a bigraph such that $\deg(w)\le k$ for all $w\in W$ and $\deg(u)\le p$ for
all $u\in U$. Then $t(G,\eta)<\infty$, and for every exhausting partition
sequence $(\PP_i)_{i=1}^\infty$, we have
\[
t(G,\eta)=\lim_{n\to \infty} t(G,\eta_{\PP_n}).
\]
\end{theorem}

\begin{proof}
The proof is a consequence of Lemma \ref{pconv} and Theorem \ref{THM:APP2}. Let
$(\PP_i)_{i=1}^\infty$ be an exhausting partition sequence of $J$. Let
$W_i:=\Ebb_{\PP_i}\eta\Ebb_{\PP_i}$. Then by Lemma \ref{pconv} we have that
$s^{W_i}_k$ converges to $s^{\eta}_k$ in $L^p$ as $i\to\infty$. Thus $\eta$ is
the $(k,p)$-limit of the sequence of the $1$-regular graphons
$\{W_i\}_{i=1}^\infty$. Theorem \ref{THM:APP2} completes the proof.
\end{proof}

Note that a bi-Markov space version of Theorem \ref{THM:APP2} gives a bi-Markov space
generalization of Theorem \ref{THM:KP-BIP} is a similar way.

\begin{theorem}\label{THM:X-DENSITY}
\lljav{Let $\Mb=(I,J,\AA,\BB,\eta)$ be a bi-Markov space $(k,p)$-loose from $J$. Let
$(\PP_i)_{i=1}^\infty$ and $\{\QQ_j\}_{j=1}^\infty$ be exhausting partition
sequences of $I$ and $J$, respectively.  Let $G=(U,W,E)$ be a bigraph such that
$\deg(w)\le a$ for all $w\in W$ and $\deg(u)\le b$ for all $u\in U$. Then
$t(G,\eta)<\infty$, and
\[
t(G,\eta)=\lim_{i,j\to \infty} t(G,\Ebb_{\PP_i}\eta\Ebb_{\QQ_j}).
\]}
\end{theorem}

\subsection{Partition approximation of homomorphism measures}\label{subsect:thm1.4}

In this section we investigate an alternative approach to homomorphism measures
using finite partitions $\PP=\{P_1,P_2,\dots,P_k\}$ of the ground space,
approximating $\eta$ by the projections $\eta_\PP$ as in the previous section.
As before, the measure $\eta_\PP$ is defined by a graphon, and hence the
measures $\eta_\PP^G$ are defined (see Section \ref{SEC:UNBOUNDED}). It is
natural to define homomorphism measures $\eta^G$ as limits of homomorphism
measures $\eta_{\PP_i}^G$ for an exhausting partition sequence
$(\PP_i)_{i=1}^\infty$. This requires an appropriate convergence notion for
such measures. There are several notions of convergence we can use: strong
(pointwise) convergence; convergence in total variation norm; weak convergence
(after putting a compact topology on $J$) etc. We choose a more technical but
more convenient path, requiring convergence on sets in $\PP_i^V$, where
$\PP_i^V$ is the partition of $J^V$ whose elements are boxes of the form
$\prod_{v \in V} P_v$ where $P_v\in\PP_i$.

The measure $\eta^G$, defined in \eqref{EQ:ETA-X}, can be expressed as follows:
Let $A=\prod_{u\in U} A_u$ and $B=\prod_{w\in W} B_w$, where $A_u,B_w\in \BB$.
Then
\begin{equation}\label{EQ:ETA-BOX}
\eta^G(A\times B) = \intl_A \prod_{w\in W} \psi_{x_{N(w)}}(B_w)\,d\pi^U(x).
\end{equation}
A simple but important remark is that changing an $A_u$ or a $B_w$ on a set of
$\pi$-measure zero, the value $\eta^G(A\times B)$ is not changed. This is
trivial for the $A_u$, and follows by Lemma \ref{LEM:PSI-NULL} for $B_w$.

\begin{theorem}\label{THM:PARMEX}
Let $(J,\BB,\eta)$ be a $(k,p)$-loose Markov space, and let $G=(U,W,E)$ be a
bigraph such that $\deg(w)\le k$ for all $w\in W$ and $\deg(u)\le p$ for all
$u\in U$. Then $\eta^G$ is partition approximable.
\end{theorem}

\begin{proof}
Let $(\PP_i)_{i=1}^\infty$ be an exhausting partition sequence. We want to
prove that $\eta_{\PP_i}^G(C)\to\eta^G(C)$ for every Borel box $C=\prod_{v\in
V} B_v$ as $i\to\infty$. First we prove the assertion in a special case.

\begin{claim}\label{CLAIM:P1}
Suppose that $B_v\in\PP_j$ for some $j$ and all $v\in V$. Then
$\eta_{\PP_i}^G(C)\to\eta^G(C)$.
\end{claim}

We may restrict our attention to $i\ge j$. We may assume that the partition
sequence $(\PP_i)_{i=1}^\infty$ is generating, not only exhausting; by Lemma
\ref{LEM:EXHAUST}, this can be achieved by changing each partition class on a
set of measure zero.

We want to mimic the proof of Theorem \ref{THM:KP-BIP}, which is a related
assertion for the total measure $\eta^G(J^V)$. To this end, we express
homomorphism measures in terms of homomorphism densities of certain bi-Markov spaces.

For a Markov space $\eta$ and $B\in\BB$ with $\pi(B)>0$, we introduce a bi-Markov space $X(\eta,B)$ which is basically the restriction of $\eta$ to $J\times
B$. Since $\eta(J\times B)=\pi(B)$, we have to multiply the restriction of
$\eta$ with $\pi(B)^{-1}$ to obtain a proper bi-Markov space $(J,B,\eta|_{J\times B}/\pi(B))$, where $\eta|_{J\times B}(A):=\eta((J\times B)\cap A)$. It is
clear from the definition that if $\eta$ is $k$-loose then
$\pi(B)s_k^{X(\eta,B)}\le s_k^\eta$ almost surely on $J^k$. It follows that if
$\eta$ is $(k,p)$-loose then so is $X(\eta,B)$ for any subset $B\in\BB$ with
positive measure.

Let $(J,\BB,\eta)$ be a $(k,p)$-loose Markov space, and let $G=(U,W,E)$ be a
bigraph such that $\deg(w)\le k$ for all $w\in W$ and $\deg(u)\le p$ for all
$u\in U$. Let $A=\prod_{u\in U} B_u$ and $B=\prod_{w\in W} B_w$. Then
\begin{align*}
\eta^G(C) &= \eta^G(A\times B) = \intl_A \prod_{w\in W} \psi_{x_{N(w)}}(B_w)\,d\pi^U(x)\\
&= \intl_A \prod_{w\in W}\pi(B_w)s^{X(\eta,B_w)}_{\deg(w)}(x_{N(w)})~d\pi^U.
\end{align*}
For $w\in W$, let $\PP_{i,w}$ denote the restriction of $\PP_i$ to $B_w$.
Define
\[
X_{i,w}:=X(\eta_{\PP_i},B_w)=\mathbb{E}_{\PP_{i}}X(\eta,B_w)\mathbb{E}_{\PP_i,w},
\]
then
\[
\eta_{\PP_i}^G(C)=\intl_A \prod_{w\in W}\pi(B_w)s^{X_{i,w}}_{\deg(w)}(x_{N(w)})~d\pi^U(x).
\]
Lemma \ref{pconv} shows that
\[
\lim_{i\to\infty}s^{X_{i,w}}_{\deg(w)}(x_{N(w)})=s^{X(\eta,B_w)}_{\deg(w)}(x_{N(w)}),
\]
where convergence is in $L^p$. This completes the proof of Claim \ref{CLAIM:P1}
by Corollary \ref{cor:Finner}.

\medskip

Note that this Claim implies immediately that the same conclusion holds if
$B_v\in \wh\PP_i$ for all $v$, since such a box is a finite union of boxes in
$\PP_i^{U\cup W}$.

\begin{claim}\label{CLAIM:P2}
For every $\eps>0$ there is a $\delta>0$ and an $i_0\in\Nbb$ such that
\begin{equation}\label{EQ:P2}
\eta^G(X\times J^{V\setminus v})<\eps\qquad\text{and}\qquad \eta_{\PP_i}^G(X\times
J^{V\setminus v})<\eps
\end{equation}
for every $v\in V$, every $X\in\BB$ with $\pi(X)<\delta$, and every $i\ge i_0$.
\end{claim}

The first inequality (which is independent of $i$) is just a restatement of the
absolute continuity of the marginal $(\eta^G)^v$ with respect to $\pi$ (Lemma
\ref{LEM:W-MARGIN}). To prove the second, choose $\delta$ such that
$\eta^G(X\times J^{V\setminus v})<\eps/2$ for $\pi(X)<2\delta$. Let $Y\in\BB$
be a set with $\pi(Y)\le \delta$ maximizing $\eta^G_{\PP_i}(Y\times
J^{V\setminus v})$. We may assume that every partition class of $\PP_i$ has
$\pi$-measure at most $\delta$. Since the marginal $(\eta^G_{\PP_i})^v$ is
proportional to $\pi$ on every partition class of $\PP_i$, the maximizing $Y$
will consist of the union of at least one partition class and at most one
subset of a partition class. So there is a set $Z\in\wh\PP_i$ such that
$Y\subseteq Z$ and $\pi(Z)\le 2\delta$. Then
\begin{align*}
(\eta^G_{\PP_i})^v(X) \le (\eta^G_{\PP_i})^v(Y) \le (\eta^G_{\PP_i})^v(Z) = \eta^G_{\PP_i}(Z\times J^{V\setminus v}).
\end{align*}
Here the box $Z\times J^{V\setminus v}$ is the product of sets in the set
algebra $\wh\PP_i$, and so by Claim \ref{CLAIM:P1},
\[
\eta^G_{\PP_i}(Z\times J^{V\setminus v})\le \eta^G(Z\times J^{V\setminus v}) + \frac\eps2 \le \eps.
\]
if $i$ is large enough. Choosing $i_0$ so that if $i\ge i_0$, then this holds
for all $v$, completes the proof of Claim \ref{CLAIM:P2}.

\medskip

To complete the proof, let $C=\prod_{v\in V} B_v$ be any box with $B_v\in\BB$.
Lemma \ref{LEM:EXHAUST} implies that there are sets $\overline{B}_v\in\wh\PP_i$
for a sufficiently large $i$ such that $\pi(B_v\triangle
\overline{B}_v)\le\delta$ for all $v\in V$, where $\delta$ is chosen as in
Claim \ref{CLAIM:P2}. Let $\overline{C}=\prod_{v\in V} \overline{B}_v$. By
Claim \ref{CLAIM:P1},
\[
\big|\eta_{\PP_i}^G(\overline{C}) - \eta^G(\overline{C})\big| \le \eps
\]
if $i$ is large enough. Furthermore,
\[
C\triangle\overline{C}\subseteq \bigcup_{v\in V}  (B_v\triangle
\overline{B}_v)\times J^{V\setminus v},
\]
and so by Claim \ref{CLAIM:P2},
\begin{align*}
\big|\eta_{\PP_i}^G(\overline{C}) - \eta_{\PP_i}^G(C)\big|
\le \sum_{v\in V} \eta_{\PP_i}^G\big((B_v\triangle \overline{B}_v)\times J^{V\setminus v})\big)\le |V|\eps,
\end{align*}
and similarly
\begin{align*}
\big|\eta^G(\overline{C}) - \eta^G(C)\big|
\le \sum_{v\in V} \eta^G\big((B_v\triangle \overline{B}_v)\times J^{V\setminus v})\big)\le |V|\eps.
\end{align*}
Summing up,
\begin{align*}
&\big|\eta_{\PP_i}^G(C) - \eta^G(C)\big| \\
&\le \big|\eta_{\PP_i}^G(\overline{C}) - \eta_{\PP_i}^G(C)\big|
+ \big|\eta_{\PP_i}^G(\overline{C}) - \eta^G(\overline{C})\big| + \big|\eta^G(\overline{C}) - \eta^G(C)\big|\\
&\le (2|V|+1)\eps.
\end{align*}
This proves the Theorem.
\end{proof}

\begin{corollary}\label{COR:ETA-STAR}
Let $(J,\BB,\eta)$ be a $(k,p)$-loose Markov space. Let $G=(U,W,E)$ be a
bigraph such that $\deg(w)\le k$ for all $w\in W$ and $\deg(u)\le p$ for all
$u\in U$. Then $\eta^G=\eta^{G^*}$.
\end{corollary}

\begin{proof}
Choose an generating partition sequence $(\PP_i)_{i=1}^\infty$. Then
$\eta_{\PP_i}$ is a graphon, and so  $\eta_{\PP_i}^{G^*}=\eta_{\PP_i}^G$. By
Theorem \ref{THM:PARMEX}, we have
\[
\eta^G(C)=\lim_{i\to\infty}\eta_{\PP_i}^G(C)=\lim_{i\to\infty} \eta_{\PP_i}^{G^*}(C) = \eta^{G^*}(C)
\]
for every box $C\in\PP_j^V$. Since the sigma-algebra generated by such sets
contains all Borel sets, it follows that $\eta^G = \eta^{G^*}$.
\end{proof}

\begin{corollary}\label{COR:KAB-WELL}
If $\Mb$ is an $(a,b)$-loose Markov space, then $K_{a,b}$ is well-measured in
$M$.
\end{corollary}

\begin{proof}
Let $p$ be any ordering of $V(K_{a,b})$. Let $\{u_1,\dots,u_a\}$ and
$\{v_1,\dots,v_b\}$ be the color classes of $K_{a,b}$, and let
$q=(u_1,\dots,u_a,v_1,\dots,v_b)$ and $r=(v_1,\dots,v_b,u_1,\dots,u_a)$.
Similarly as in the proof of Lemma \ref{LEM:REORDER}, we may assume that
$\eta_p$ does not change if we reorder the first $a+b-1$ elements, and it does
not change if we flip consecutive non-adjacent nodes, so it follows that
$\eta_p=\eta_q$ or $\eta_p=\eta_r$ (depending on the color class of the last
node in $p$). But by Corollary \ref{COR:ETA-STAR},
$\eta^{K_{a,b}}=\eta^{K_{b,a}}$ and so $\eta_q=\eta_r$. Thus $\eta_p$ is
independent of $p$.
\end{proof}

Combining Corollary \ref{COR:KAB-WELL} with Theorem \ref{THM:MAIN-KAB}, we
obtain the following:

\begin{corollary}\label{COR:G-WELL}
Let $(J,\BB,\eta)$ be a $(k,p)$-loose Markov space. Let $G=(U,W,E)$ be a
bigraph such that $\deg(w)\le k$ for all $w\in W$ and $\deg(u)\le p$ for all
$u\in U$. Then $G$ is well-measured in $\Mb$.
\end{corollary}

This corollary implies Theorem \ref{THM:DENSITY}.

\subsection{Products of graphs}\label{prodgraphs}

In this section we investigate an interesting construction of a sparse graph
sequence, where the limit object is easily guessed, but it is more difficult to
tell in what sense do these graphs converge to this limit.

For two edge-weighted graphs $H_1$ and $H_2$, we define their product
$H_1\times H_2$ as the edge-weighted graph on $V(H_1)\times V(H_2)$, where the
edge-weight $w$ in the product is defined by
\[
w((x_1,x_2),(y_1,y_2))=w_1(x_1,y_1)w_2(x_2,y_2).
\]
If every edge weight in $H_i$ is $1/(2|E(H_i)|)$, then this is just the
categorical product of the two graphs, with the edges weighted analogously.

Let $H_n=(V_n,E_n)$, $n=1,2,\dots$ be simple graphs, and let $p_n=|V(H_n)|$,
$q_n=|E(H_n)|$. Define
\[
\wh{H}_n= H_1\times\dots\times H_n.
\]
We can also define the product of infinitely many graphs. Indeed, let $J=V_1\times
V_2\times\cdots$, with the Borel $\sigma$-algebra $\BB$. There is a natural
graph on $J$, in which $(u_1,u_2,\dots)$ is connected to $(v_1,v_2,\dots)$ if
and only if each $u_i$ is connected to $v_i$ in $H_i$ for every $i\in\Nbb$. We
need to define a measure on this edge set. A Markov step from a point
$(v_1,v_2,\dots)\in J$ is obtained by making a step of the random walk on $H_i$
from $v_i$, independently for different indices $i$. The measure of a cylinder
set $C=A_1\times\cdots \times A_n\times E_{n+1}\times\cdots$ $(A_i\subset V_i^2)$ is
\[
\eta(C)=
  \begin{cases}
    \prod_{j=1}^n \frac{|A_j|}{|E_j|}, & \text{if $A_j\subseteq E_j$ for all $1\le j\le n$}, \\
    0 & \text{otherwise}.
  \end{cases}
\]
We denote this Markov space by $H_\infty=(J,\BB,\eta)$. Let $\Ab_\infty$ denote
the adjacency operator of $H_\infty$.

In this section we study the question whether $\wh{H}_n\to H_\infty$ in any
reasonable sense.

Let $\lambda^{(n)}_1=1,\lambda^{(n)}_2,\dots$ be the eigenvalues of the
transition matrix of the random walk on $H_n$, with corresponding eigenvectors
$w^{(n)}_1,w^{(n)}_2,\dots$. For every choice of indices $1\le i_j\le p_j$, the
transition matrix of the graph $\wh{H}_n$ has an eigenfunction
\begin{equation}\label{EQ:EIGENF-HN}
f_{i_1\dots i_n}(u_1,u_2,\dots,u_n) = \prod_{j=1}^n w^{(j)}_{i_j,u_j}
\end{equation}
with eigenvalue
\begin{equation}\label{EQ:EIGENV-HN}
\lambda_{i_1\dots i_n}=\prod_{j=1}^n \lambda_{i_j}^{(j)}.
\end{equation}
These eigenvalues remain eigenvalues in $H_\infty$, and so do the corresponding
eigenfunctions, if we consider them as defined on $V(H_\infty)$ but depending
only on the first $n$ coordinates. We can also think of this as extending the
formulas \eqref{EQ:EIGENF-HN} and \eqref{EQ:EIGENV-HN} to infinite products,
but choosing the eigenvalue $1$ with eigenfunction identically $1$ for all
$j>n$. Let us call these eigenvalues {\it finitary}.

We may or may not obtain further nonzero eigenvalues as infinite products with
infinitely many nontrivial eigenvalues. This will not happen if and only if the
transition matrices of the graphs $H$ have a common eigenvalue gap in the sense
that for some $c>0$,
\begin{equation}\label{EQ:MU-BOUND}
\mu_n=\max_{j\ge 2}|\lambda^{(n)}_j|\le 1-c
\end{equation}
for every $n$.

Trivially, the multiplicity of a nonzero finitary eigenvalue may be infinite,
and these eigenvalues may have accumulation points other than $0$. It is easy
to see that the eigenvalues have no nonzero accumulation point if and only if
\begin{equation}\label{EQ:MU-CONV}
\mu_n\to 0\qquad (n\to\infty).
\end{equation}

The Markov space $H_\infty$ has a natural partition $\PP_n$ defined by the
first $n$ coordinates. More exactly, $\PP_n$ has partition classes $U_z$ $(z\in
V_1\times\dots\times V_n)$, consisting of all extensions of $z$. Then
$(H_\infty)_{\PP_n}$ is the graphon associated with the graph $\wh{H}_n$, with
edge weights $1/(q_1\cdots q_n)$. Let $G$ be a graph with $a$ nodes and $b$
edges, then
\[
t(G,(H_\infty)_{\PP_n}) = t^*(G,\wh{H}_n)= \prod_{j=1}^n t^*(G,H_j) =
\prod_{j=1}^n \frac{\hom(G,H_j)p_j^{2b-a}}{(2q_j)^b}.
\]
Let us define
\begin{equation}\label{EQ:PROD-DEF}
t^{\times}(G,H_\infty) = \prod_{j=1}^\infty t^*(G,H_j),
\end{equation}
provided the product is convergent. With this definition,
\[
t^*(G,(H_\infty)_{\PP_n}) = t^*(G, \wh{H}_n) \to t^{\times}(G,H_\infty)
\]

When does the product in \eqref{EQ:PROD-DEF} converge? Is the value
$t^{\times}(G,H_\infty)$ as defined above also the limit of
$t^*(G,(H_\infty)_{\QQ_n})$ for every exhausting partition sequence
$(\QQ_i)_{i=1}^\infty$? Is $t^{\times}(G,H_\infty)=t(G,H_\infty)$? For the
first question we give a reasonably general sufficient condition. The other two
remain open.

Let $(H_n)$ be a sequence of (very dense) simple graphs such that $p_n\ge n$.
Let $\overline{H}_n$ denote the complement of $H_n$, including all loops at the
nodes. Let $d_n$ denote the maximum degree of $\overline{H}_n$ and assume that
$d_n=O(1)$. \lljav{Let $r_n= p_n^2-2q_n$ be the number of oriented edges of
$\overline{H}$, then $r_n\le d_np_n=O(p_n)$.}

Let $G$ be a simple graph with $a$ nodes and $b$ edges. For $Y\subseteq E(G)$,
let $G_Y=(V(G),Y)$. Then by inclusion-exclusion,
\[
\hom(G,H_n) = \sum_{Y\subseteq E(G)} (-1)^{|Y|} \hom(G_Y,\overline{H}_n).
\]
Here $\hom(G_\emptyset,\overline{H}_n)=p_n^a$ and $\hom(G_Y,\overline{H}_n)=r_n
p_n^{a-2}$ if $|Y|=1$. If $|Y|\ge 2$, then selecting one node from each
connected component of $G_Y$, we get $a-c$ points, where $c\ge 2$. We can map
these points $p_n^{a-c}$ ways, but the remaining points in at most $d_n^c$
ways, so we get then
\[
\hom(G_Y,\overline{H}_n) \le d_n^c p_n^{a-c}=O(p_n^{a-2}).
\]
Hence
\[
\hom(G,H_n) = p_n^a-br_n p_n^{a-2} + O(p_n^{a-2}),
\]
and so
\[
t(G,H_n)= 1-\frac{br_n}{p_n^2} + O(p_n^{-2}).
\]
Clearly $t(K_2,H_n) = 1-r_n/p_n^2$, and so
\[
t(K_2,H_n)^b = 1-\frac{b r_n }{p_n^2} + O\Big(\frac{r_n^2}{p_n^4}\Big)
= 1-\frac{b r_n }{p_n^2} + O(p_n^{-2}).
\]
Thus
\[
t^*(G,H_n)=\frac{1-br_n/p_n^2+O(p_n^{-2})}{1-br_n/p_n^2+O(p_n^{-2})} = 1+O(p_n^{-2}).
\]
Using that $d_n=O(1)$ and $p_n\ge n$, it follows that the product in
\eqref{EQ:PROD-DEF} is convergent.

It is interesting to consider two special examples.

\begin{example}[Powers of a graph]\label{EXA:POWERS}
As remarked before, our methods above work for compact operators only. Here is
an example where extension of the results to operators that are ``almost''
compact would be very useful.

Let $H=(V,E)$ be a $d$-regular graph with $n$ nodes, and consider its direct powers
$H^{{\times} k}$, $k=1,2,\dots$. Let $\eta$ be the uniform distribution on the
edges of $H$, then the marginal of $\eta$ is the uniform distribution $\pi$ on
$V$, and the stationary distribution on $V(H^{{\times} k})$ is $\pi^k$.

Going to the limit $k\to\infty$, we get a limit object on $J=V^\Nbb$, with
sigma-algebra generated by sets $A_1\times A_2\times \cdots$ where all but a finite
number of factors are $V$, and stationary measure defined by
$\pi^\omega(A_1\times A_2\times\cdots)= \pi(A_1)\pi(A_2)\cdots$. The edge measure
$\eta^\omega$ is defined similarly. The edge measure is supported on the set
$E\times E\times\dots$, so it is quite singular with respect to
$\pi^\omega\times\pi^\omega$.

For a point $(v_1,v_2,\dots)\in V^\omega$ of the Markov space
$(V^\omega,\eta^\omega)$, a Markov step is generated by choosing a random
neighbor $u_i$ of $v_i$ independently for all $i$, and moving to
$(u_1,u_2,\dots)$.

The operator $\Ab$ associated with the Markov space $\eta^\omega$ is,
unfortunately, not compact. Let $\lambda_1=1,\lambda_2,\dots,\lambda_n$ be the
eigenvalues of $H$ (normalized by $d$), with corresponding eigenvectors
$w_1,\dots,w_n$. Then for every finite sequence of positive integers
$k_1<\dots<k_r$, and every choice of indices $2\le i_1,\dots,i_r\le n$, $\Ab$
has an eigenfunction
\begin{equation}\label{EQ:EIGENF-POW}
f_i(u_1,u_2,\dots) = \prod_{j=1}^r w_{i_j,u_{k_j}}
\end{equation}
with eigenvalue
\begin{equation}\label{EQ:EIGENV-POW}
\prod_{j=1}^r \lambda_{i_j}.
\end{equation}
The multiplicity of each of these eigenvalues is infinite, since there are a
countably infinite number of sequences $(k_i)$ with the same length. So $\Ab$
is not compact. On the other hand, the nonzero eigenvalues of $\Ab$ are
products of a finite number of normalized eigenvalues of $H$, so they form a
discrete set with only one accumulation point at $0$, so $\Ab$ does have some
resemblance of compact operators.

Can we define the density of a bipartite graph $G$ in $\eta^\omega$, and show
that this is nonzero? If, in addition, we can prove that $t^*(G,H^{{\times}
k})\to t(G,\eta^\omega)$, then Sidorenko's conjecture would follow.
\end{example}

\begin{example}[Products of complete graphs]\label{EXA:COMPLETE-PROD}
Let $H_n=K_n$ be the complete $n$-graph (without loops). For the product to be
nontrivial, we consider $K_2\times K_3\times\cdots$. For every graph $G$ with
$p$ nodes and $q$ edges,
\[
|\hom(G,K_n)|=\chi_G(n),
\]
where $\chi_G$ denotes the chromatic polynomial of $G$. It follows that for
$n\ge 2$,
\[
t(K_2,\wh{H}_n) = \prod_{j=2}^i\frac{j(j-1)}{j^2} = \frac1{i},
\]
showing that $(\wh{H}_1,\wh{H}_2,\dots)$ is a sparse graph sequence. It is well
known that $\chi_G$ is a polynomial of degree $p$, and it has the form
$x^p+a_1x^{p-1}+a_2x^{p-2}+\dots$ with $a_1=-q$. Hence
\[
t^*(G,H_n)= \frac{\chi_G(n) n^{q-p}}{(i-1)^q} =
\frac{i^q - q i^{q-1} + a_2 i^{q-2}+\cdots}{i^q-q i^{q-1} +\binom{q}{2}q^2 +\cdots}
=1+O\Big(\frac1{n^2}\Big),
\]
and so the product $\prod_{j=2}^\infty t^*(G,H_j)$ is convergent.

The Markov space $H_\infty$ is $(k,p)$-loose for every $k,p\ge 1$. Indeed, let
$x=(x_2,x_3,\dots)$ be a random point of $H_\infty$, and let
$y_1=(y_{12},y_{13},\dots)$, $\dots$, $y_k=(y_{k2},y_{k3},\dots)$ be $k$ random
steps from $x$. Then for $n>k$, the joint distribution of $y_{1n},\dots,y_{kn}$
is uniform over all $k$-tuples of points of $K_n$; for $n\le k$ it is not
uniform, but trivially it has a density function $F_{kn}$. Then
$s_k=F_{k2}F_{k3}\cdots F_{k,k}$ (not depending on the coordinates $n>k$) is the
density function of $(y_1,\dots,y_k)$. Trivially $s_k^p$ is a bounded function,
and so $s_k\in L^p$.

It follows that for every bipartite graph $G$ we have $\lim_{i\to\infty}
t^*(G,(H_\infty)_{\QQ_i})=t(G,H_\infty)$ for every exhausting partition
sequence $(\QQ_i)_{i=1}^\infty$ by Theorem \ref{THM:KP-BIP}.
\end{example}

\section{Cycle densities and the spectrum}

It is well known that the homomorphism number of the $k$-cycle $C_k$ in a graph
$G$ is the sum of the $k$-th powers of the eigenvalues of the adjacency matrix
of $G$. This can be generalized to graphons and even to bounded symmetric
measurable functions $W:\Omega^2\to\mathbb{R}$ where $(\Omega,\mu)$ is a
standard probability space. In this case $t(C_k,W)$ is equal to
$\sum_{i=1}^\infty \lambda_i^k$ where the numbers $\lambda_i$ are the
eigenvalues of $W$ as an integral kernel operator. In this section we push this
further to operators on $L^2$ spaces whose $k$-th Schatten norm is finite for
some $k$. In particular the main result of this section (see Theorem
\ref{cycdense2}) implies the following theorem.

\begin{theorem}\label{THM:CYCLE-APPROX}
Let $k$ be an integer and assume that the $k$-th Schatten norm of the adjacency
operator $\Ab$ of a Markov space $\Mb=(J,\mathcal{B},\eta)$ is finite. Then
$\eta^{C_k}$ is partition approximable, and $t(C_{k},\eta)$ is equal to the sum of
the $k$-th powers of the eigenvalues of $\Ab$.
\end{theorem}

Let $\|.\|^\bullet_p$ denote $p$-th Schatten norm. Also, given a compact
self-adjoint operator $A$ on an infinite dimensional Hilbert space $H$ and an
integer $k\geq 1$, let $\lambda_k^+(A)$ be the $k$-th largest (counting
multiplicities) positive eigenvalue of $A$, with the convention
$\lambda_k^+(A)=0$ if there are less than $k$ such eigenvalues. Similarly let
$\lambda_k^-(A)$ be the $k$-th smallest (counting multiplicities) negative
eigenvalue of $A$, with the convention $\lambda_k^+(A)=0$ if there are less
than $k$ such eigenvalues. Note that we then have
\[
\left(\|A\|^\bullet_p\right)^p=\sum_{k=1}^\infty |\lambda_k^+(A)|^p+\sum_{k=1}^\infty |\lambda_k^-(A)|^p.
\]

\begin{lemma}\label{le:CFW}
Let $U$ be a $d$ dimensional subspace in an infinite dimensional Hilbert space
$\mathcal{H}$ and let $A$ be a compact self-adjoint operator such that
$\|A\|_p^\bullet<\infty$ for some $p\geq 1$. Let
$\lambda'_1\geq\lambda'_2\geq\ldots\geq\lambda'_d$ be the eigenvalues of
$P_UAP_U|_U$. Then $\lambda_k^+(P_UAP_U)=\max\{\lambda_k',0\}$,
$\lambda_k^-(P_UAP_U)=\min\{\lambda_{d+1-k}',0\}$ and
$\lambda_k^+(A)\geq\lambda_k'\geq \lambda_{d+1-k}^-(A)$ for all $1\leq k\leq
d$.
\end{lemma}

\begin{proof}
The identities follow from the fact that the operator $P_UAP_U$ is reduced by
the subspace $U$, and is the zero operator on $U^\perp$. Concerning the
inequalities, by the Courant--Fischer--Weyl theorem (or minmax principle, see
\cite[Excercise 6.34]{EW}), we have the following:
\begin{align*}
\lambda_k^+(A)&=\sup_{{\rm dim}(W)=k}\min_{\substack{v\in W,\\ \|v\|=1}} \langle Av,v\rangle,\\
\lambda_{d+1-k}^-(A)&=\inf_{{\rm dim}(W)=d+1-k}\max_{\substack{v\in W,\\ \|v\|=1}} \langle Av,v\rangle,\\
\lambda_k'&=\sup_{\substack{{\rm dim}(W)=k,\\W\subset U}}\min_{\substack{v\in W,\\ \|v\|=1}} \langle (P_UAP_U|_U)v,v\rangle,\\
\lambda_k'&=\inf_{\substack{{\rm dim}(W)=d+1-k,\\W\subset U}}\max_{\substack{v\in W,\\ \|v\|=1}} \langle (P_UAP_U|_U)v,v\rangle.
\end{align*}
Now note that then
\begin{align*}
\lambda_k'=&\sup_{\substack{{\rm dim}(W)=k,\\W\subset U}}\min_{\substack{v\in W,\\
\|v\|=1}} \langle (P_UAP_U|_U)v,v\rangle= \sup_{\substack{{\rm dim}(W)=k,\\
W\subset U}}\min_{\substack{v\in W,\\ \|v\|=1}} \langle AP_Uv,P_Uv\rangle\\
=& \sup_{\substack{{\rm dim}(W)=k,\\W\subset U}}\min_{\substack{v\in W,\\ \|v\|=1}} \langle Av,v\rangle
\leq \sup_{\substack{{\rm dim}(W)=k}}\min_{\substack{v\in W,\\ \|v\|=1}} \langle Av,v\rangle=\lambda_k^+(A),
\end{align*}
with the other inequality following by symmetry.
\end{proof}

This immediately leads to the following result.

\begin{corollary}\label{schproj}
Let $U$ be a $d$ dimensional subspace in a Hilbert space and let $A$ be a
self-adjoint operator such that $\|A\|_p^\bullet<\infty$ for some $p\geq 1$.
Then $\|P_UAP_U\|_p^\bullet\leq\|A\|_p^\bullet$.
\end{corollary}

The above can be used to express the $\ell$-th Schatten norm of an operator as
the limit of that of its finite dimensional approximants.

\begin{prop}\label{cycdense}
Assume $A$ is a bounded, self-adjoint operator on a Hilbert-space $\mathcal{H}$
with $\|A\|_{\ell}^\bullet<\infty$. Assume that
$\{\mathcal{H}_j\}_{j=1}^\infty$ is a sequence of finite dimensional subspaces
of $\mathcal{H}$ such that $\mathcal{H}_j\subseteq\mathcal{H}_{j+1}$ holds for
every $j$ and $\,\bigcup_{j=1}^\infty\mathcal{H}_j$ is dense in $\mathcal{H}$. Then
\begin{align*}
&\sum_{k=1}^\infty \lambda_k^+(A)^\ell=\lim_{j\to\infty}\sum_{k=1}^\infty \lambda_k^+(P_{\mathcal{H}_j}AP_{\mathcal{H}_j})^\ell,\\
&\sum_{k=1}^\infty \lambda_k^-(A)^\ell=\lim_{j\to\infty}\sum_{k=1}^\infty \lambda_k^-(P_{\mathcal{H}_j}AP_{\mathcal{H}_j})^\ell,
\end{align*}
and also $\|A\|_{\ell}^\bullet=\lim_{j\to\infty}\|A_j\|_{\ell}^\bullet<\infty$.
\end{prop}

\begin{proof}
For $j\in\mathbb{N}$, let $A_j:=P_{\mathcal{H}_j}AP_{\mathcal{H}_j}$. By Lemma
\ref{le:CFW}, we have $0\leq\lambda_k^+(A_j)\leq\lambda_k^+(A)$ and
$0\geq\lambda_k^-(A_j)\geq\lambda_k^-(A)$ for every $k\geq1$. If we can show
that for every $k\geq1$, $\lim_{j\to\infty}\lambda_k^+(A_j)=\lambda_k^+(A)$ and
$\lim_{j\to\infty}\lambda_k^-(A_j)=\lambda_k^-(A)$ hold, then we are done by
the monotone convergence theorem.

Fix $\varepsilon>0$ and $k\geq1$, and let $W\subset\mathcal{H}$ be a $k$
dimensional subspace such that
\[
\min_{\substack{v\in W,\\ \|v\|=1}} \langle Av,v\rangle\geq\lambda_k^+(A)-\varepsilon.
\]
Since $\bigcup_{j=1}^\infty\mathcal{H}_i$ is dense in $\mathcal{H}$, we have that
$\lim_{j\to\infty}P_{\mathcal{H}_j}=I$ strongly, and so
\[
\lim_{j\to\infty}\min_{\substack{v\in W,\\ \|v\|=1}} \langle A_jv,v\rangle=\min_{\substack{v\in W,\\
\|v\|=1}} \langle Av,v\rangle\geq\lambda_k^+(A)-\varepsilon,
\]
implying that $\liminf_j \lambda_k^+(A_j)\geq \lambda_k^+(A)-\varepsilon$. As
$\lambda_k^+(A_j)\leq\lambda_k^+(A)$ for all $j$, and $\varepsilon>0$ was
arbitrary, we obtain $\lim_{j\to\infty}\lambda_k^+(A_j)=\lambda_k^+(A)$ as
desired. By symmetry the same holds for the negative eigenvalues, and we are
done.
\end{proof}

For the next theorem we need some preparation. Let $\Ab$ be a bounded, self
adjoint operator on $L^2(\Omega,\nu)$, where $(\Omega,\nu)$ is a standard
probability space. Assume that $\PP$ is a finite, measurable, non-degenerate
partition of $\Omega$. Then we have that $\mathbb{E}_\PP \Ab\mathbb{E}_\PP$ is
an integral kernel operator representable by a bounded measurable step-function
of the form $W:~\Omega^2\to\mathbb{R}$. In this context it makes sense to talk
about subgraph densities of the form
$t(H,\mathbb{E}_\PP\Ab\mathbb{E}_\PP):=t(H,W)$.

\begin{theorem}\label{cycdense2}
Assume that $\Ab$ is a bounded, self-adjoint operator on $L^2(\Omega,\nu)$
where $(\Omega,\nu)$ is a standard probability space. Assume that
$\|\Ab\|_{\ell}^\bullet<\infty$ for some $\ell\in\mathbb{N}$. Let
$(\PP_i)_{i=1}^\infty$ be an exhausting partition sequence of $\Omega$. Then
\[
\lim_{j\to\infty} t(C_{\ell},\mathbb{E}_{\PP_j}\Ab\mathbb{E}_{\PP_j})
=\sum_{k=1}^\infty \lambda_k^+(A)^\ell+\sum_{k=1}^\infty \lambda_k^-(A)^\ell
\]
\end{theorem}

\begin{proof}
For $j\in\mathbb{N}$, let $\mathcal{H}_j$ denote the finite dimensional space of
$\PP_j$-measurable functions. It is clear that the sequence
$\{\mathcal{H}_j\}_{j=1}^\infty$ satisfies the conditions of Proposition
\ref{cycdense}. Note that the operator $\mathbb{E}_{\PP_j}$ is equal to
$P_{\mathcal{H}_j}$. Since $A_j=\mathbb{E}_{\PP_j}A\mathbb{E}_{\PP_j}$ is
representable by a step function we have that
\[
t(C_{\ell},A_j)
=\sum_{k=1}^\infty \lambda_k^+(A_j)^\ell+\sum_{k=1}^\infty \lambda_k^-(A_j)^\ell.
\]
Then Proposition \ref{cycdense} completes the proof.
\end{proof}

Theorem \ref{THM:CYCLE-APPROX} follows from these results: Theorem
\ref{cycdense2} implies that for every exhausting partition sequence $(\PP_i)$
we have $t(C_k, \eta_{\PP_i})\to t(C_k,\eta)<\infty$. This in particular
implies $(2, 2)$-looseness, so Theorem \ref{THM:PARMEX} implies that
$\eta^{C_k}$ is partition approximable.

\section{Open problems}

\begin{prob}\label{PROB:TREE-INDEP}
Find general conditions under which the measure $\eta_\FF$ produced by a tree
decomposition $\FF$ of a graph $G$ (not necessarily a star decomposition) is
independent of the decomposition.
\end{prob}

\begin{prob}\label{PROB:SOFIC}
Is every $k$-loose Markov space the $k$-limit of graphons?
\end{prob}

\begin{prob}[$(k,p)$-profile]\label{REM:LL-ETA}
Let $\LL(\eta)\subseteq\mathbb{N}^2$ denote the set of pairs $(k,p)$ for which
$\eta$ is $(k,p)$-loose. Theorem \ref{THM:KP-BIP} expresses subgraph densities
in $\eta$ under appropriate conditions on its ``$(k,p)$-profile'' $\LL(\eta)$.
Some properties of the set $\LL(\eta)$ have been established above: it is
symmetric in the two coordinates and it is monotone in the sense that if
$(k,p)\in\LL(\eta)$ and $k'\leq k,p'\leq p$, then $(k',p')\in\LL(\eta)$. It
would be interesting to establish further properties. For example, for the
$d$-dimensional orthogonality Markov space $\eta_d$, we have
$(k,p)\in\LL(\eta_d)$ if and only if $k+p\le d$ (see Lemma 3 in \cite{KLSz2}).
How ``wild'' can the boundary of the set $\LL(\eta)$ be in general?
\end{prob}

\begin{prob}\label{PROB:LOOSEPP}
Is every \lljav{$1$-regular} $L^p$-graphon $(p+1,p)$-loose? Perhaps
$(k,p)$-loose for every $k$? \lljav{(This is false without the assumption that
the graphon is $1$-regular, as shown by a construction similar to Example
\ref{EXA:PP}. We are grateful to the anonymous referee for this remark.)}
\end{prob}

\begin{prob}\label{PROb:SUP}
Does $t(G,\eta)=t_{\rm part}(G,\eta)$ hold for every bipartite graph $G$ and
every Markov space $\eta$? Could this be true at least for all graphons?
\end{prob}

\begin{prob}[Measure family and partition approximation]\label{PROB:STEPPING}
Let $(J,\BB,\eta)$ be a Markov space, let $G$ be a graph, and let
$(\PP_i)_{i=1}^\infty$ be an exhausting partition sequence. Assume that there
is a normalized Markovian measure family on the induced subgraphs of $G$. Does
this imply that $\eta_{\PP_n}^G\to\eta^G$ on boxes? This is true if $G=K_2$,
but even this very special case is not absolutely trivial.
\end{prob}

\begin{prob}\label{PROB:STEPPING3}
Theorems \ref{THM:MAIN-KAB} and \ref{THM:PARMEX} suggest that a theorem along
the following lines should hold: Let $(J,\BB,\eta)$ be a $k$-loose Markov
space, let $G$ be a graph with girth at least $5$ and degrees at most $k$, and
let $(\PP_i)_{i=1}^\infty$ be an exhausting partition sequence. Then
$\eta_{\PP_n}^G\to\eta^G$ on boxes.
\end{prob}

\begin{prob}\label{PROB:EQUIVALENCE}
Are the definitions of subgraph densities based on approximations and based on
various sequential tree decompositions equivalent, under reasonably general
conditions?
\end{prob}

\begin{prob}[Measures for graphings]\label{PROB:SUBGRAPH-MEASURE}
For a graphing $\Hb$ and a connected graph $G$, a measure on homomorphisms
$G\to\Hb$ can be defined in a natural way: We label a node $u$ of $G$ to get a
rooted graph $G_u$. For each $x\in J$, let $\psi_{G,x}$ be the counting measure
on homomorphisms mapping $u$ onto $x$ (this is a finite set of bounded size for
a fixed $G$). Then $\Psi_G=(\psi_{G,x}:~x\in J)$ is a measurable family, and we
can define $\eta^G=\pi[\Psi_G]$. It can be shown (using the Mass Transport
Principle for graphings) that this measure is independent of the choice of the
root. Is there a common generalization with our results?
\end{prob}

\begin{prob}[Compactness, and cycles versus other graphs]\label{PR:1}
Assume that $t(C_{2k},\eta)=\infty$ for even cycles $C_{2k}$. Is
$t(G,\eta)=\infty$ for every connected bipartite graph $G$ that is not a tree?
If this implication is true, then in particular whenever $t(G,\eta)$ is finite
for at least one connected bipartite graph $G$ besides trees, the operator
$\Ab_\eta$ is of some Schatten-class, and hence compact. A weaker question is
therefore whether this compactness is a necessary condition in any well-defined
sense for the finiteness of at least one density.
\end{prob}

\begin{prob}[Regularity and variance]\label{PR:2}
In \cite{BCLSV2}, a weak regularity partition of a graphon was constructed as
a finite, measurable, non-degenerate partition $\PP$ into a given number of
classes for which $\|W_\PP\|_2^2$ is (nearly) maximized. Do partitions $\PP$
for which $\|\eta_\PP\|_2^2$ is maximized have special properties and uses?
\end{prob}

\begin{prob}[Regularity and spectral approximation]\label{PR:2x}
It seems that the regularity lemma can be defined inside certain sparsity
classes. Assume that we just consider measures such that $t(C_{2k},\eta)<c$ for
some fixed constant. Then there are at most $c/\eps^{2k}$ eigenvalues greater
than $\eps>0$. The corresponding spectral approximation of the operator
$\Ab_\eta$ (represented by some bounded measurable function) may serve as a
regularization of $\eta$.
\end{prob}

\begin{prob}[Quotient topology vs $t$]\label{PR:3}
In \cite{KLSz1} we introduced a distance of s-graphons using quotients. How
does it relate to subgraph densities? Is there some continuity in any
direction, generalizing the Counting Lemma and/or the Inverse Counting Lemma
for bounded graphons?
\end{prob}

\begin{prob}[Edge coloring model approach]\label{PR:4}
It was observed and used in dense graph limit theory that spectral sums can be
used to rewrite $t(G,W)$ as the value of a certain edge coloring model. As an
example, see the proof that forcible finite rank graphons are step functions in
\cite{LSz11}. Nothing prevents us from pushing this further to more general
compact operators $\Ab_\eta$.
\end{prob}

\begin{prob}[Limit object]\label{PR:5}
Assume that for a graph sequence $\{G_i\}_{i=1}^\infty$, the numerical sequence
$t^*(F,G_i)$ is convergent for every graph $F$ satisfying appropriate sparsity
constraints. Is there a limit object in the form of an s-graphon?
\end{prob}

\begin{prob}[Existence of limit]\label{PR:6}
Can it happen that $\lim_{i\to\infty}t(G,\eta_{\PP_i})$ is finite for certain
exhausting partition sequences and infinite for other ones? Could it oscillate
for a given partition sequence?
\end{prob}

\noindent{\bf Acknowledgement.} Our thanks are due to the anonymous referees of
the first version of this paper for their very thorough and thoughtful
comments, which has lead to the elimination of several errors, and to
substantial improvement in the presentation.

\section{Appendices}

\subsection{Absolute continuity and Radon-Nikodym derivatives}\label{APP:A}

We collect some measure theory facts that are probably known, but difficult to
quote.

\begin{lemma}\label{LEM:RAD-NIK}
Let $(J,\BB)$ be a standard Borel space, and $\mu,\nu$ two measures on $\BB$
such that $\nu\ll\mu$ and $\mu$ is sigma-finite. Then the Radon-Nikodym
derivative $d\nu/d\mu:~J\to[0,\infty]$ exists, and it is uniquely determined
$\mu$-almost everywhere.
\end{lemma}

\begin{proof}
To prove the existence, we can split $J$ into a countable number of Borel sets
with finite $\mu$-measure, and apply the lemma to each of these. In other
words, we may assume that $\mu(J)$ is finite.

We claim that there is a set $U\in \BB$ such that $\nu|_U$ sigma-finite and
$\nu(X)=\infty$ for every $X\subseteq J\setminus U$ with $\mu(X)>0$. Let
$c=\sup\{\mu(X):~X\in\BB, \nu|_X\ \text{sigma-finite}\}$. Let $Y_n\in\BB$ be
chosen so that $\nu|_{Y_n}$ is sigma-finite and $\mu(Y_n)>c-1/n$. Then
$U=\bigcup_n Y_n$ has the properties as desired. Clearly $\nu|_U$ is
sigma-finite, and $\mu(U)\ge c$. By the maximality of $c$, we have $\mu(U)=c$,
and every set $X\subseteq J\setminus U$ with $\nu(X)<\infty$ must have
$\mu(X)=0$.

The standard Radon-Nikodym theorem, applied to $\mu|_U$ and $\nu|_U$, gives
$f|_U$. Defining $f$ as constant $\infty$ on $X\setminus M$, we obtain a
measurable $f:~J\to[0,\infty]$ such that $\nu= f\cdot \mu$.

Uniqueness of $f$ follows by standard arguments.
\end{proof}

\begin{lemma}\label{LEM:0-EQUI}
Let $(I,\AA)$ and $(J,\BB)$ be Borel spaces. Let $\Phi=(\mu_x:~x\in I)$ be a
measurable family of measures on $(J,\BB)$ and $\alpha_1,\alpha_2\in\Mf(\AA)$.
If $\alpha_1\ll\alpha_2$ then $\alpha_1[\Phi]\ll\alpha_2[\Phi]$.
\end{lemma}

\begin{proof}
Suppose that $\alpha_2[\Phi](R)=0$ for some $R\in\AA\times\BB$. Let
$R(x)=\{y\in J:~(x,y)\in R\}$. Then
\begin{align*}
\alpha_2[\Phi](R) = \intl_I  \mu_x(R(x))\,d\alpha_2(x) =0
\end{align*}
implies that $\alpha_2\{x\in I:~\mu_x(R(x))>0\}=0$. But then $\alpha_1\{x\in
I:~\mu_x(R(x))>0\}=0$, implying by the same computation that
$\alpha_1[\Phi](R)=0$.
\end{proof}

\begin{lemma}\label{LEM:MARGINALS}
Let $\sigma$ be a probability distribution on $\BB^V$. Suppose that $\sigma \ll
\prod_{v\in V}\sigma^{\{v\}}$. Then $\sigma\ll\sigma^S\times\sigma^{V\setminus
S}$ for every $S\subseteq V$.
\end{lemma}

\begin{proof}
Let $\xi=\prod_{v\in V}\sigma^{\{v\}}$ and $f=d\sigma/d\xi$. For any
$S\subseteq V$, the function
\[
f^S(y)=\intl_{J^{V\setminus S}}f(y,z)\,d\xi^{V\setminus S}(z)\qquad (y\in J^S).
\]
satisfies $\sigma^S=f^S\cdot \xi^S$. Let $U=\{y:~f^S(y)=0\}$ and
$Z=\{z:~f^{V\setminus S}(z)=0\}$. Suppose that
$\left(\sigma^S\times\sigma^{V\setminus S}\right)(X)=0$. Then
\[
\left(\sigma^S\times\sigma^{V\setminus S}\right)(X) = \intl_X f^S(y) f^{V\setminus S}(z)\,d\xi(y,z)
\]
implies that $X\subseteq (U\times J^{V\setminus S})\cup(J^S\times Z)$
$\xi$-almost everywhere. Hence
\begin{align*}
\sigma(X) &\le \intl_{U\times J^{V\setminus S}} f(x)\,d\xi(x) + \intl_{J^S \times Z} f(x)\,d\xi(x)\\
&= \intl_{U} f^S(x)\,d\xi^S(x) + \intl_{Z} f^{V\setminus S}(x)\,d\xi^{V\setminus S}(x) =0.
\end{align*}
\end{proof}

\subsection{Markovian property and Markov random fields}\label{APP:B}

We show that Markovian measure families and Markov random fields on a graph
$G=(V,E)$ are related. This latter can be defined as a probability distribution
$\mu$ on $\BB^V$ such that the marginal family $(\mu^S:~S\subseteq V)$
satisfies the Markovian property for sets $U,W\subseteq V$ such that $U\cup
W=V$. More precisely, let $R_{S,T}=(\Rho_{S,T,z}:~z\in J^S)$ be a measurable
family of measures on $\BB^{T\setminus S}$ such that $\mu^S[R_{S,T}] =\mu^T$.
Then we require that whenever $U\cup W=V$, $S=U\cap W$, and there is no edge
between $U\setminus S$ and $W\setminus S$, then
\begin{equation}\label{EQ:RHO-INDEP}
\Rho_{S,V,x}=\Rho_{S,U,x}\times\Rho_{S,W,x}
\end{equation}
for $\mu^S$-almost all $x\in J^S$.

\begin{prop}\label{PROP:MARKOV-MARKOV}
If a family $\MM=(\mu_S:~S\subseteq V)$ of sigma-finite measures is Markovian
with respect to a graph $G$, and $\mu=\mu_V$ is a probability distribution,
then $\mu$ is a Markov random field on $G$.
\end{prop}

\begin{proof}
Recall that $\nu_{S,T,x}$ is the disintegration of $\mu_T$ with respect to
$\mu_S$, and $\Rho_{S,T,x}$ is the disintegration of $\mu^T$ with respect to
$\mu^S$. Our first step is to express $\pi_J$ in terms of $\nu$. Let $S\subseteq
T\subseteq V$. We claim that for all $B\in\BB^{T\setminus S}$ and
$\mu^S$-almost all $x\in J^S$,
\begin{equation}\label{EQ:RHO-NU0}
\Rho_{S,T,x}(B) = \intl_B  \frac{\nu_{T,V,xy}(J^{V\setminus T})}{\nu_{S,V,x}(J^{V\setminus S})} \,d\nu_{S,T,x}(y).
\end{equation}
First note that by \eqref{EQ:RN-MT}, we have
\[
\mu^S\left(\left\{x:~\nu_ {S,V,x}(J^{V\setminus S})=0\right\}\right)=
\int_{\left\{x:\,\nu_ {S,V,x}(J^{V\setminus S})=0\right\}}\nu_{S,V,x}(J^{V\setminus S}) \,d\mu_S(x)=0,
\]
hence the right hand side is well-defined for $\mu^S$-almost all $x\in J^S$. To
prove \eqref{EQ:RHO-NU0}, we integrate both sides on $A\in\BB^S$ with respect
to $\mu^S$. The left hand side turns into
\[
\intl_A\Rho_{S,T,x}(B)\,d\mu^S(x) = \mu^T(A\times B),
\]
whereas, using \eqref{EQ:RN-MT}, the right hand side becomes
\begin{align*}
\intl_A \intl_B &\frac{\nu_{T,V,xy}(J^{V\setminus T})}{\nu_{S,V,x}(J^{V\setminus S})}\,d\nu_{S,T,x}(y)\,d\mu^S(x)\\
&= \intl_{A} \intl_{B} \nu_{T,V,xy}(J^{V\setminus T})\,d\nu_{S,T,x}(y)\,d\mu_S(x)\\
&= \intl_{A \times B} \nu_{T,V,xy}(J^{V\setminus T}) \, d\mu_T(xy)
= \intl_{A\times B} 1\, d\mu^T(xy) = \mu^T(A\times B).
\end{align*}
This proves \eqref{EQ:RHO-NU0}. Hence for $B\in \BB^{U\setminus S}$ and
$C\in\BB^{W\setminus S}$,
\[
\Rho_{S,V,x}(B\times C)=\intl_{B\times C}  \frac{1}{\nu_{S,V,x}(J^{V\setminus S})} \,d\nu_{S,V,x}(y)
=\frac{1}{\nu_{S,V,x}(J^{V\setminus S})}\nu_{S,V,x}(B\times C).
\]
Using Lemma \ref{LEM:MARKOV-YDEP},
\begin{align*}
\Rho_{S,U,x}(B)&=\intl_B\frac{\nu_{U,V,xy}(J^{V\setminus U})}{\nu_{S,V,x}(J^{V\setminus S})} \,d\nu_{S,U,x}(y)
=\intl_B\frac{\nu_{S,W,x}(J^{V\setminus U})}{\nu_{S,V,x}(J^{V\setminus S})} \,d\nu_{S,U,x}(y)\\
&=\frac{\nu_{S,W,x}(J^{W\setminus S})}{\nu_{S,V,x}(J^{V\setminus S})} \nu_{S,U,x}(B).
\end{align*}
Using a similar expression for $\Rho_{x,W}(C)$, we get
\begin{align*}
(\Rho_{S,U,x}&\times\Rho_{S,W,x})(B\times C)
=\frac{\nu_{S,W,x}(J^{W\setminus S})\nu_{S,U,x}(J^{U\setminus S})}{\nu_{S,V,x}(J^{V\setminus S})^2}
\nu_{S,U,x}(B) \nu_{S,W,x}(C)\\
&=\frac{1}{\nu_{S,V,x}(J^{V\setminus S})} \nu_{S,V,x}(B\times C)=\Rho_{S,V,x}(B\times C).
\end{align*}
This proves Proposition \ref{PROP:MARKOV-MARKOV}.
\end{proof}

\subsection{Partition sequences}\label{APP:C}

We prove the following basic facts about exhausting partition sequences.

\begin{lemma}\label{LEM:EXHAUST}
Let $(J,\BB,\pi)$ be a standard Borel probability space, and let
$(\PP_i)_{i=1}^\infty$ be a partition sequence, with
$\RR=\bigcup_{i=1}^\infty\PP_i$. Then the following are equivalent:

\smallskip

{\rm(i)} $\RR$ is exhausting with respect to $\pi$, i.e., for every $X\in\BB$
there is a set $Y\in\overline{\RR}$ such that $\pi(X\triangle Y) =0$.

\smallskip

{\rm(ii)} For every $X\in\BB$ and every $\eps>0$ there is a set $Y\in\wh\RR$
such that $\pi(X\triangle Y) <\eps$.

\smallskip

{\rm(iii)} There is a generating partition sequence $(\QQ_i)_{i=1}^\infty$ and
a Borel set $U$ with $\pi(U)=0$ such that $\PP_i|_{J\setminus U} =
\QQ_i|_{J\setminus U}$ for all $i$.

\smallskip

{\rm(iv)} $\bigcup_{i=1}^\infty L^1(J,\overline{\PP}_i,\pi)$ is dense in
$L^1(J,\BB,\pi)$.
\end{lemma}

\begin{proof}
(i)$\Rightarrow$(ii): Let $\SS$ be the family of sets $X\in\BB$ for which for
every $\eps>0$ there is a set $Y\in\wh\RR$ such that $\pi(X\triangle Y)<\eps$.
Then $\SS$ is closed under complementation (trivially), and under finite union
and finite intersection (almost trivially). It follows that it is closed under
countable union. Indeed, let $\eps>0$, $X=X_1\cup X_2\cup\cdots$, where $X_i\in
\SS$, and $X_i'=X_i\setminus(X_1\cup\dots\cup X_{i-1})$. Since the $X_i'$ are
disjoint, we have $\sum_{i=N+1}^\infty \pi(X_i')<\eps/2$ for an appropriate
$N$. Since $X_i'\in\SS$, there are $Y_i\in\wh\RR$ such that $\pi(X_i'\triangle
Y_i)<\eps/(2N)$. Let $Y=\bigcup_{i=1}^N Y_i\in\wh\RR$, then
\[
\pi(X\triangle Y) \le \sum_{i=1}^N \pi(X_i'\triangle Y_i) + \sum_{i=N+1}^\infty \pi(X_i')<\eps.
\]
So $\SS$ is a sigma-algebra. Trivially $\RR\subseteq \SS$, so
$\overline{\RR}\subseteq\SS$. By (i), for every $S\in\BB$ there is a set
$Z\in\overline{\RR}$ such that $\pi(Z\triangle S)=0$, and then $Z\in\SS$
implies that there is a set $Y\in\wh\RR$ for which $\pi(Z\triangle Y)<\eps$.
Then $\pi(Y\triangle X)<\eps$.

\medskip

(ii)$\Rightarrow$(i): Let $X\in\BB$, and for $k\ge 1$, let $Y_k\in\wh\RR$ be a
set such that $\pi(X\triangle Y_k)<2^{-k}$. Consider the sets
\[
Z_n=\bigcap_{k=n}^\infty Y_k,\quad\text{and}\quad Z=\bigcup_{n=1}^\infty Z_n.
\]
Trivially $Z\in\overline{\RR}$. Furthermore,
\[
\pi(Z_n\setminus X)\le \lim\inf_k \pi(Y_k\setminus X)\le \lim\inf_k \pi(Y_k\triangle X)=0,
\]
and
\[
\pi(X\setminus Z_n)\le \sum_{k=n}^\infty \pi(X\setminus Y_k)\le \sum_{k=n}^\infty \pi(X\triangle Y_k)<2^{1-n}.
\]
Using this, a similar computation gives that $\pi(X\triangle Z)=0$.

\medskip

(i)$\Rightarrow$(iii): Let $B_1,B_2,\dots$ be a countable generating set of
$\BB$. For each $i$, there is a set $C_i\in\overline{\RR}$ such that
$\pi(B_i\triangle C_i)=0$. Let $U=\bigcup_{i=1}^\infty B_i\triangle C_i$, then
$\pi(U)=0$. Let $(\QQ'_i)_{i=1}^\infty$ be a generating partition sequence of
Borel subsets of $U$, and let $\QQ_i=\PP_i|_{J\setminus U}\cup \QQ'_i$. Then
$(\QQ_i)_{i=1}^\infty$ is a generating partition sequence in $(J,\BB)$ such
that $\PP_i|_{J\setminus U} = \QQ_i|_{J\setminus U}$ for all $i$.

\medskip

(iii)$\Rightarrow$(i): Let $(\QQ_i)_{i=1}^\infty$ be a generating sequence of
partitions and $U$, a Borel set with $\pi(U)=0$ such that
$\PP_i|_{J\setminus U} = \QQ_i|_{J\setminus U}$ for all $i$. Then
$\PP_i|_{J\setminus U}$ is a generating partition sequence for the Borel sets
in $J\setminus U$, and hence for every $C\in\AA$ there is a
$D\in\overline{\RR}|_{J\setminus U}$ for which $C\setminus U = D$. Then
$D=D_1\setminus U$ for some $D_1\in \overline{\RR}$, and $\pi(C\triangle
D_1)\le \pi(U)=0$.

\medskip

\{(i),(ii),(iii)\}$\Rightarrow$(iv): By (iii), we may assume that $\RR$ is
generating. It suffices to prove that every function $\one_S$ ($S\in\BB$) can
be approximated arbitrarily well by finite linear combinations of functions
$\one_A$ $(A\in \PP_i)$, since the functions $\one_S$ are dense in
$L^1(J,\BB,\pi)$, and $\one_A\in L^1(J,\PP_i,\pi)$. This follows by (ii).

\medskip

(iv)$\Rightarrow$(ii): For every $S\in\BB$ and $\eps>0$ there are sets
$A_1,\dots,A_k\in\RR$ and nonzero real numbers $\alpha_1,\dots,\alpha_k$ such
that
\[
\big\|\one_S-\alpha_1\one_{A_1}-\dots-\alpha_k\one_{A_k}\big\|_1 <\eps.
\]
Let $i$ be the least integer for which $A_1,\dots,A_k\in\wh{\PP}_i$. By
splitting an $A_j$ into partition classes in $\PP_i$ (and adjusting the
coefficients as necessary), we may assume that every $A_j\in P_i$. Then the
$A_j$ are disjoint. Replacing $\alpha_j$ by $1$ if $\pi(A_j\cap S)\ge
\pi(A_j)/2$, and by $0$ otherwise, we decrease the left hand side. Deleting
zero terms, we may assume that every $\alpha_j=1$, and then $Y=\bigcup_j A_j$
satisfies $\pi(S\triangle Y)<\eps$.
\end{proof}

\begin{lemma}\label{appendix2}
Let $(J,\BB,\pi)$ be a Borel probability space and assume that $\mu$ is a
measure on $J^k$ for some $k\in\mathbb{N}$ such that its marginal distribution
in each coordinate is $\pi$. Let $(\PP_i)_{i=1}^\infty$ be an exhausting
partition sequence with respect to $\pi$. Then the partition sequence
$(\PP^k_i)_{i=1}^\infty$ is exhausting to both $\pi^k$ and $\mu$.
\end{lemma}

\begin{proof}
Replacing ``exhausting'' by ``generating'', the assertion is easy. For
exhausting partition sequences, it follows by Lemma \ref{LEM:EXHAUST}(iii).
\end{proof}

\subsection{Unbounded graphons and non-acyclic graphs}\label{APP:D}

We give the details of the arguments for Example \ref{EXA:GRAPHON}. Recall that
$f:~I=[-1,1]\to\mathbb{R}$ has the following properties: $f\ge 0$; $f(-x)=f(x)$ for all
$x\in I$; $\int_I f(x)\,dx=1$; $f$ is convex and monotone decreasing for $x>0$.
This function defines a graphon by
\[
W(x,y)=f(x-y)\qquad(x,\in I),
\]
where $f$ is extended periodically modulo $2$. Clearly $W$ is symmetric and
$1$-regular. The stationary measure $\mu$ of the graphon is $\mu=\lambda/2$. We
claim that as a kernel operator, it is positive semidefinite and compact as
$L^2(\mu)\to L^2(\mu)$.

The eigenfunctions of $W$ are $\sin(k\pi x)$ and $\cos(k\pi x)$, and hence the
eigenvalues can be obtained as the Fourier coefficients of $f(x)$. By the
symmetry of $f$, eigenvalues associated with the eigenfunction $\sin(k\pi x)$
are zero. The other eigenvalues can be expressed for even $k\ge 0$ as
\begin{align}\label{EQ:LK0}
\lambda_k &= \intl_{-1}^1 f(x) \cos(k\pi x)\,dx  =
2\intl_{0}^1 f(x) \cos(k\pi x)\,dx\nonumber\\
&= \frac2{k}\intl_{0}^{k} f\Big(\frac{y}{k}\Big) \cos(\pi y)\,dy
= \frac2{k}\sum_{j=0}^{k/2-1} \intl_{0}^{2}  f\Big(\frac{y+2j}{k}\Big)  \cos(\pi y)\,dy.
\end{align}
To see that this is nonnegative, notice that $\cos(\pi y)=-\cos(\pi(1-y)) =
-\cos(\pi(1+y))=\cos(\pi(2-y))$, and so we can write \eqref{EQ:LK0} as
\begin{align}\label{EQ:LK}
\lambda_k = \frac2{k}\sum_{j=0}^{k/2-1} \intl_{0}^{1/2}  &\Big[f\Big(\frac{2j+y}{k}\Big)
- f\Big(\frac{2j+1-y}{k}\Big) - f\Big(\frac{2j+1+y}{k}\Big)\nonumber\\
& + f\Big(\frac{2j+2-y}{k}\Big) \Big] \cos(\pi y)\,dy.
\end{align}
Here each integrand is nonnegative by the convexity of $f$. For odd $k$, we get
an extra term
\[
\frac2{k}\intl_{k-1}^{k} f\Big(\frac{y}{k}\Big) \cos(\pi y)\,dy
\ge\frac2{k}\intl_{k-1}^{k} f\Big(\frac{y}{k}\Big) \,dy \intl_{k-1}^k \cos(\pi y)\,dy = 0,
\]
where we used Chebyshev's sum inequality on the monotone decreasing functions $f(y/k)$ and $\cos(\pi y)$.
This proves that $W$ is positive semidefinite.

By the Riemann--Lebesgue Lemma, $\lambda_k\to 0$. This implies that $W$ defines
a compact operator $L^2(\mu)\to L^2(\mu)$.

As a useful special case, we consider the function defined by
\[
f(x)=\frac1{x(2-\ln(x))^2}= \left(\frac1{2-\ln(x)}\right)'
\]
for $x> 0$, and $f(x)=f(-x)$ for $x<0$. (For $x=0$ we can define $f(x)=0$.) We
have
\[
\intl_{-1}^1 f(x)\,dx = 2\left[\frac1{2-\ln(x)}\right]_0^1 = 1.
\]
The conditions that $f$ is monotone decreasing and convex for $x>0$ are easy to
check. To determine the order of magnitude of $\lambda_k$, note that the first
term in \eqref{EQ:LK} is
\begin{align*}
a_k&=\frac1{k} \intl_{0}^{1/2} \left(f\Big(\frac{y}{k}\Big) - f\Big(\frac{1-y}{k}\Big)
- f\Big(\frac{1+y}{k}\Big) + f\Big(\frac{2-y}{k}\Big)\right) \cos(\pi y)\,dy\\
&\ge \frac1{k} \intl_{0}^{1/4} \left(f\Big(\frac{y}{k}\Big) - f\Big(\frac{1-y}{k}\Big)
- f\Big(\frac{1+y}{k}\Big) + f\Big(\frac{2-y}{k}\Big)\right) \cos(\pi y)\,dy
\end{align*}
Using the inequality $f(3x)\le f(x)/2$ and $f(5x)\le f(x)/4$ valid for
$x<1/(4k)$ if $k$ is large enough, we can estimate the expression in the large
parenthesis as
\begin{align*}
f\Big(\frac{y}{k}\Big) &- f\Big(\frac{1-y}{k}\Big)
- f\Big(\frac{1+y}{k}\Big) + f\Big(\frac{2-y}{k}\Big) \\
&\ge \Big(1-\frac12-\frac14\Big)f\Big(\frac{y}{k}\Big)>\frac14 f\Big(\frac{y}{k}\Big).
\end{align*}
Hence
\[
\lambda_k\ge a_k\ge  \frac1{4k}\intl_{0}^{1/4} f\Big(\frac{y}{k}\Big)\cos\frac\pi4\,dy
=  \frac1{4\sqrt{2}} \intl_{0}^{1/(4k)} f(x)\,dx = \frac1{4\sqrt{2}(2+\ln(4k))}.
\]
It follows that no operator power of $W$ has finite trace, so $t(C_n,W)=\infty$
for all $n$. Since $f$ is bounded away from $0$, it follows that
$t(G,W)=\infty$ for every graph $G$ containing a cycle.

\end{document}